\documentclass[11pt]{amsart}

\usepackage{graphicx}
\usepackage{hyperref}
\usepackage{xypic}
\usepackage{color}
\usepackage{latexsym}
\usepackage{amssymb}
\usepackage{lscape}
\usepackage[utf8x]{inputenc}
\usepackage{epsfig,url}
\usepackage[all,cmtip]{xy}
\usepackage[all]{xy}
\usepackage{amsmath}
\usepackage{stmaryrd}
\theoremstyle{plain} % text will be italics
\newtheorem{thmx}{Theorem}

\newtheorem{theorem}{Theorem}[section]
\newtheorem{proposition}[theorem]{Proposition}
\newtheorem{corollary}[theorem]{Corollary}
\newtheorem{lemma}[theorem]{Lemma}
\newtheorem{conjecture}[theorem]{Conjecture}
\newtheorem{assumption}[theorem]{Assumption}

\theoremstyle{definition} % text will be roman

\newtheorem{example}[theorem]{Example}

\newtheorem{remark}[theorem]{Remark}
\newtheorem{definition}[theorem]{Definition}

\DeclareMathOperator{\Kunn}{\mathtt{Ku}}

\DeclareMathOperator{\VV}{\mathbb{V}}
\DeclareMathOperator{\Ts}{T}
\DeclareMathOperator{\Sph}{S}
\DeclareMathOperator{\RS}{\mathbf{M}}
\DeclareMathOperator{\SR}{\mathbb{S}}

\DeclareMathOperator{\Span}{Span}
\DeclareMathOperator{\mon}{mon}
\DeclareMathOperator{\DData}{\mathtt{dec}}
\DeclareMathOperator{\forg}{\mathtt{for}}
\DeclareMathOperator{\Sum}{\mathbf{sum}}
\DeclareMathOperator{\triv}{\mathfrak{o}}

\DeclareMathOperator{\Ob}{Ob}

\DeclareMathOperator{\st}{st} 
\DeclareMathOperator{\pp}{\mathtt{pp}} 
\DeclareMathOperator{\Dec}{\mathtt{Dec}} 
\DeclareMathOperator{\Wt}{wt}
\DeclareMathOperator{\Jac}{Jac}

\DeclareMathOperator{\Co}{\mathtt{co}}
\DeclareMathOperator{\SC}{\mathtt{co}}

\DeclareMathOperator{\hm}{\mathtt{hm}}
\DeclareMathOperator{\No}{\mathbf{N}}
\DeclareMathOperator{\fib}{\mathbf{fib}}

\DeclareMathOperator{\sw}{\mathtt{sw}}
\DeclareMathOperator{\Vect}{\mathbf{Vect}}

\DeclareMathOperator{\euv}{\mathfrak{V}}
\DeclareMathOperator{\eue}{\mathfrak{E}}
\DeclareMathOperator{\SG}{\mathtt{Sym}}

\DeclareMathOperator{\eu}{\mathfrak{eu}}

\DeclareMathOperator{\Jor}{\mathtt{Jor}}

\DeclareMathOperator{\id}{id}
\DeclareMathOperator{\TS}{\mathtt{TS}}

\DeclareMathOperator{\Fr}{Fr}
\newcommand{\QQ}{\mathbb{Q}}
\DeclareMathOperator{\tw}{\mathtt{tw}}

\DeclareMathOperator{\prim}{prim}

\DeclareMathOperator{\Rep}{Rep}

\DeclareMathOperator{\Sp}{\mathfrak{P}}
\DeclareMathOperator{\fr}{fr}

\DeclareMathOperator{\SHom}{\mathcal{H}\textit{om}}

\DeclareMathOperator{\crit}{crit}
\DeclareMathOperator{\KK}{\mathrm{K_0}}
\DeclareMathOperator{\res}{\mathtt{res}}
\DeclareMathOperator{\tr}{\mathop{tr}}
\DeclareMathOperator{\Gl}{G}

\DeclareMathOperator{\Der}{D}
\DeclareMathOperator{\Aff}{\mathbb{A}}
\newcommand{\Ho}{\mathrm{H}}
\newcommand{\Do}{\mathrm{D}}

\newcommand{\state}[1]{\textbf{\{}#1\textbf{\}}}

\newcommand{\supp}{\mathop{\rm Supp}\nolimits}

\newcommand{\ext}{\mathop{\rm ext}\nolimits}
\newcommand{\Hom}{\mathop{\rm Hom}\nolimits}

\newcommand{\Db}[1]{\mathrm{D}^{\mathrm{b}}(#1)}
\newcommand{\Dlb}[1]{\mathrm{D}^{\mathrm{lb}}(#1)}
\newcommand{\Dblf}[1]{\mathrm{D}^\mathrm{b}_{\mathrm{lc}}(#1,\mathbb{Q})}
\newcommand{\Dbc}[1]{\mathrm{D}^{\mathrm{b}}_{\mathrm{c}}(#1,\mathbb{Q})}
\newcommand{\Dcc}[1]{\mathrm{D}_{\mathrm{c}}(#1,\mathbb{Q})}

\newcommand{\GL}{{\rm GL}}

\newcommand{\PGL}{{\rm PGL}}

\newcommand{\MF}{\mathop{\rm MF}\nolimits}
\newcommand{\MHM}{\mathop{\mathbf{MHM}}\nolimits}
\newcommand{\MMHS}{\mathop{\mathbf{MMHS}}\nolimits}
\newcommand{\EMHS}{\mathop{\mathbf{EMHS}}\nolimits}

\newcommand{\MHS}{\mathop{\mathbf{MHS}}\nolimits}

\newcommand{\Perv}{\mathop{\rm Perv}\nolimits}

\DeclareMathOperator{\real}{Re}
\newcommand{\Gr}{\mathop{\rm Gr}\nolimits}

\newcommand{\Sym}{{\rm Sym}}

\newcommand{\rat}{\mathop{\mathbf{rat}}\nolimits}

\newcommand{\pt}{\mathop{\rm pt}}

\renewcommand{\tilde}{\widetilde}
\newcommand{\Cp}{\mathbb{C}}

\newcommand{\tate}[1]{\mathbb{Q}\textbf{\{}#1\textbf{\}}}
\newcommand{\mtate}[2]{\mathbb{Q}_{#1}\textbf{\{}#2\textbf{\}}}

\newcommand{\D}[1]{\mathrm{D}(#1)}

\renewcommand{\AA}{\mathbb A}

\newcommand{\wt}{{\rm wt}}

\begin{document}
\author[B. Davison ]{Ben Davison}
\address{B. Davison: IST Austria}
\email{ben.davison@ist.ac.at}

\title[The critical CoHA of a quiver with potential]{The critical CoHA of a quiver with potential}
\begin{abstract} 
Pursuing the similarity between the Kontsevich--Soibelman construction of the cohomological Hall algebra of BPS states and Lusztig's construction of canonical bases for quantum enveloping algebras, and the similarity between the inetgrality conjecture for motivic Donaldson--Thomas invariants and the PBW theorem for quantum enveloping algebras, we build a coproduct on the cohomological Hall algebra associated to a quiver with potential.  We also prove a cohomological dimensional reduction theorem, further linking a special class of cohomological Hall algebras with Yangians, and explaining how to connect the study of character varieties with the study of cohomological Hall algebras.
\end{abstract}

\maketitle
\tableofcontents
\thispagestyle{empty}

\section{Introduction}

\subsection{Background and motivation} Let $Q$ be a quiver, i.e. a pair of sets $Q_1$ and $Q_0$ and a pair of maps $s,t\colon Q_1\rightarrow Q_0$.  We will always assume that these sets are finite.  Let $W\in\mathbb{C} Q/[\mathbb{C} Q,\mathbb{C} Q]$ be an element of the vector space quotient.  Such an element is called a \textit{potential}.  In the paper \cite{COHA} Kontsevich and Soibelman introduced the \textit{critical cohomological Hall algebra} (abbreviated to CoHA in this paper) $\mathcal{H}_{Q,W}$ associated to this data.  The main result of the present paper is that this algebra carries a kind of localised coproduct, in a sense that we explain in Section \ref{mresults}.  Firstly, we explain why it is natural to look for and expect such a coproduct to exist.

The algebra $\mathcal{H}_{Q,W}$ is graded by the monoid $\mathbb{N}^{Q_0}$.  For $\gamma\in\mathbb{N}^{Q_0}$, the degree $\gamma$ summand is defined to be the dual compactly supported cohomology
\begin{equation}
\label{CHlive}
\mathcal{H}_{Q,W,\gamma}:=\Ho_{c,\Gl_{\gamma}}(\RS_{Q,\gamma},\phi_{\tr(W)_{\gamma}}\mathbb{Q}[-1])^{\vee}[\chi(\gamma,\gamma)],
\end{equation}
where 
\begin{equation}
\label{RSdef}
\RS_{Q,\gamma}:=\bigoplus_{a\in Q_1}\Hom(\Cp^{\gamma(s(a))},\Cp^{\gamma(t(a))}).
\end{equation}
is an affine space of $\gamma$-dimensional representations of $Q$, acted on by the group
\begin{equation}
\label{Gdef}
\Gl_{\gamma}:=\prod_{i\textrm{ a vertex of }Q}\GL_{\Cp}(\gamma(i))
\end{equation}
by change of basis.  The sheaf complex $\phi_{\tr(W)_{\gamma}}\mathbb{Q}[-1]$ is the vanishing cycles complex associated to the function $\tr(W)_{\gamma}$ on the space of $\gamma$-dimensional $Q$-representaions $\RS_{Q,\gamma}$.  The definition of the vanishing cycles complex is recalled in Section \ref{sheaves_section}, while notions related to the geometry of the moduli space of $Q$-representations are recalled in Section \ref{coha_section}.  The shift here is defined by the expressions
\begin{align}
\label{ldefs}
\chi(\gamma_1,\gamma_2):=&l_0(\gamma_1,\gamma_2)-l_1(\gamma_1,\gamma_2)\\ \nonumber
l_0(\gamma_1,\gamma_2):=&\sum_{i\in Q_0}\gamma_1(i)\gamma_2(i)\\ \nonumber
l_1(\gamma_1,\gamma_2):=&\sum_{a\in Q_1}\gamma_1(s(a))\gamma_2(t(a)).
\end{align}
We will attempt to be agnostic, in this introduction, regarding the category (which we will denote $\mathcal{C}$) to which each cohomologically graded piece of (\ref{CHlive}) belongs, but we assume this category is Tannakian, and that we recover $\Ho_{c,\Gl_{\gamma}}(\RS_{Q,\gamma},\phi_{\tr(W)_{\gamma}}\mathbb{Q}[-1])$, defined as a cohomologically graded vector space by applying the fibre functor to $\Ho_{c,\Gl_{\gamma}}(\RS_{Q,\gamma},\phi_{\tr(W)_{\gamma}}\mathbb{Q}[-1])$ considered as an element of the unbounded derived category $\D{\mathcal{C}}$, with zero differential. Then 
\[
\mathcal{H}_{Q,W}:=\bigoplus_{\gamma\in\mathbb{N}^{Q_0}}\mathcal{H}_{Q,W,\gamma}
\]
belongs to $\D{\mathcal{C}_{\mathbb{Z}^{Q_0}}}$, the category of formal $\mathbb{Z}^{Q_0}$-indexed direct sums of objects in $\D{\mathcal{C}}$.

One can associate to the data $(Q,W)$ a different algebra, the Jacobi algebra $\textrm{Jac}(Q,W)$ --- see \cite{ginz} for a definition of this algebra.  The CoHA $\mathcal{H}_{Q,W}$ can be considered as a categorification of the Donaldson--Thomas invariants of $\textrm{Jac}(Q,W)$.  To explain this we assume firstly that $Q$ is symmetric, in the sense that for all pairs $i,j\in Q_0$ the number of arrows from $i$ to $j$ is the same as the number of arrows from $j$ to $i$.  Secondly we assume that $\mathcal{C}$ has a notion of a weight filtration such that $\Gr_{\Wt}(\mathcal{H}_{Q,W})$ belongs to $\Dlb{\mathcal{C}_{\mathbb{N}^{Q_0}}}$, the full subcategory of the derived category of $\mathcal{C}_{\mathbb{N}^{Q_0}}$ consisting of complexes such that each $\mathbb{Z}\oplus\mathbb{N}^{Q_0}$-graded piece of the associated graded object belongs to $\Db{\mathcal{C}}$, the bounded derived category of $\mathcal{C}$.  Here the associated graded is with respect to the weight filtration, and $\mathbb{Z}$ keeps track of the weight degree.  Thirdly, we assume that there is an element $V_{\prim}\in\Dlb{\mathcal{C}_{\mathbb{N}^{Q_0}}}$, such that 
\begin{equation}
\label{DTdef}
[\Sym(V_{\prim}\otimes\mathbb{Q}[u])]=[\mathcal{H}_{Q,W}]\in\KK(\mathcal{C}_{\mathbb{Z}^{Q_0}})
\end{equation}
where $u$ is a formal variable of weight 2.  The $\mathbb{N}^{Q_0}$-grading on $V_{\prim}\otimes\mathbb{Q}[u]$ is given by $(V_{\prim}\otimes\mathbb{Q}[u])_{\gamma}=V_{\prim,\gamma}\otimes\mathbb{Q}[u]$.  The first assumption can be relaxed, but the situation becomes more complicated --- in this case one should twist the natural monoidal structure on $\Dlb{\mathcal{C}_{\mathbb{Z}^{Q_0}}}$ by an analogue of the Tate twist, see Section \ref{Hdef}.  The second assumption will be satisfied in the case $\mathcal{C}=\MMHS$, the category of monodromic mixed Hodge structures introduced by Kontsevich and Soibelman and recalled in Section \ref{vss}.  The third assumption is automatically satisfied if the first two are.  Then an alternative definition of the $\KK(\mathcal{C})$-valued Donaldson--Thomas invariant $\Omega_{Q,W,\gamma}$ is 
\[
\Omega_{Q,W,\gamma}:=[V_{\prim,\gamma}]\in\KK(\Dlb{\mathcal{C}}).  
\]
The famous integrality conjecture of Joyce and Song or Kontsevich and Soibelman regarding Donaldson--Thomas invariants (see \cite{COHA} for a very general proof) can be expressed by saying that each $\Omega_{Q,W,\gamma}$ is in the image of the inclusion $\KK(\Db{\mathcal{C}})\rightarrow \KK(\Dlb{\mathcal{C}})$.  Defining the refined DT invariant by taking the weight polynomial: $\Omega_{Q,W,\gamma}^{\textrm{qntm}}(q^{1/2})=\chi_{q}(V_{\prim,\gamma})$, the integrality statement implies that the refined DT invariants are Laurent polynomials in $q^{1/2}$, as opposed to Laurent formal power series --- this is sufficient to deduce the Joyce--Song version of the integrality conjecture for the numerical DT invariants associated to Jacobi algebras.

The above formulation of the integrality conjecture makes it clear how understanding the structure of the CoHA $\mathcal{H}_{Q,W}$ might be hoped to offer a categorified upgrade of the integrality conjecture.  Namely, if one could show that there is a \textit{categorification} of (\ref{DTdef}), i.e. an isomorphism of algebras
\begin{equation}
\label{DTcat}
\Sym(V_{\prim}\otimes\mathbb{Q}[u])\rightarrow\mathcal{H}_{Q,W},
\end{equation}
and that each of the generating objects $V_{\prim,\gamma}$ is finite dimensional, then by passing to the Grothendieck group, one would deduce a new proof of the integrality conjecture.  This is exactly the result proved by Efimov in \cite{Ef12}, in the case that $Q$ is symmetric and $W=0$, with the corollary that the categorified integrality conjecture is true in this case.\footnote{This is also the result that was erroneously stated in greater generality in an earlier preprint of this paper.}

In general the statement that there is an isomorphism of algebras as in (\ref{DTcat}) is too strong, and the right formulation of the categorified integrality conjecture is that there is \textit{some} isomorphism as in (\ref{DTcat}).  For instance one may hope that there is a subobject $V_{\prim}\otimes\mathbb{Q}[u]\subset \mathcal{H}_{Q,W}$ such that the natural map (\ref{DTcat}) built from the algebra structure on $\mathcal{H}_{Q,W}$ is an isomorphism of objects in $\D{\mathcal{C}_{\mathbb{Z}^{Q_0}}}$.  If $\mathfrak{g}$ is a Lie algebra, and $\mathfrak{g}[u]$ is the associated disc algebra with Lie product defined by $[gu^i,g'u^{i'}]:=[g,g']u^{i+i'}$, the Poicar\'e--Birkhoff--Witt theorem makes just such a claim regarding the inclusion $\mathfrak{g}[u]\subset\mathcal{U}(\mathfrak{g}[u])$, and so the categorified integrality conjecture would follow from the statement, weaker than the statement that $\mathcal{H}_{Q,W}$ is a free commutative algebra, that $\mathcal{H}_{Q,W}\cong \mathcal{U}(\mathfrak{g}[u])$ for some $\mathbb{Z}^{Q_0}$-graded Lie algbra $\mathfrak{g}$ with finite dimensional $\mathbb{Z}^{Q_0}$-graded pieces.  

In practice even this statement is too strong, and one should look for the structure of a quantum enveloping algebra on $\mathcal{H}_{Q,W}$, i.e. one should prove that $\mathcal{H}_{Q,W}$ is a deformation of some $\mathcal{U}(\mathfrak{g}[u])$, where $\mathfrak{g}$ has finite-dimensional $\mathbb{Z}^{Q_0}$-graded pieces.  This statement is still enough to prove that $\mathcal{H}_{Q,W}$ has a PBW basis, and so it is enough to prove the categorified integrality conjecture.  The statement regarding quantum enveloping algebras is proved in a companion paper with Sven Meinhardt \cite{DaMe15b} using the coproduct constructed in this paper, and so we may say that the categorified integrality conjecture follows from the fact that we can build a kind of generalised Yangian algebra out of an arbitrary pair $(Q,W)$.  Regardless of whether one wants to prove that $\mathcal{H}_{Q,W}$ is a free commutative algebra, a universal enveloping algebra, or a deformed universal enveloping algebra, one has a coproduct --- this is our motivation for finding this additional structure on $\mathcal{H}_{Q,W}$.

\subsection{Main results}\label{mresults}
The main contribution of this paper is the construction of a localised coproduct
\begin{equation}
\label{cind}
\mathcal{H}_{Q,W}\xrightarrow{\Delta}\mathcal{H}_{Q,W}\tilde{\boxtimes}_+^{\tw}\mathcal{H}_{Q,W}
\end{equation}
satisfying the following theorem:
\begin{thmx}
The map $\Delta$ endows $\mathcal{H}_{Q,W}$ with the structure of a $Q$-localised bialgebra in the sense of Definition \ref{lbs} below.
\end{thmx}
The notion of a $Q$-localised bialgebra is introduced in Section \ref{comult_section}; it closely mimics the notion of a bialgebra in a symmetric monoidal category, with the modification that the target of $\Delta$ is localised, and the monoidal category in which we work is not symmetric so that the usual compatibility statement, that $\Delta$ should be an algebra homomorphism, becomes more delicate.  The right hand side of (\ref{cind}) is, up to a shift, a localisation of $\mathcal{H}_{Q,W}\otimes\mathcal{H}_{Q,W}$ in a sense that we briefly explain.  The $\gamma$-graded summand $\mathcal{H}_{Q,W,\gamma}\subset\mathcal{H}_{Q,W}$ carries an action of $\Ho_{\Gl_{\gamma}}(\pt,\QQ)\cong \mathbb{Q}[x_{1,1},\ldots,x_{1,\gamma(1)},\ldots,x_{n,1},\ldots,x_{n,\gamma(n)}]^{\SG_{\gamma}}$, where $\Gl_{\gamma}=\prod_{i\in Q_0}\GL_{\mathbb{C}}(\gamma(i))$ and $\SG_{\gamma}=\prod_{i\in Q_0}\SG_{\gamma(i)}$.  We localise by formally inverting the element
\begin{align*}
&\prod_{i,j\in Q_0}\prod_{\substack{s'=1,\ldots,\gamma_1(j)\\s''=1,\ldots,\gamma_2(i)}}(x^{(1)}_{j,s'}-x^{(2)}_{i,s''})\in\\&\mathbb{Q}[x^{(1)}_{1,1},\ldots,x^{(1)}_{1,\gamma_1(1)},\ldots,x^{(1)}_{n,1},\ldots,x^{(1)}_{n,\gamma_1(n)}]^{\SG_{\gamma_1}}\otimes\mathbb{Q}[x^{(2)}_{1,1},\ldots,x^{(2)}_{1,\gamma_2(1)},\ldots,x^{(2)}_{n,1},\ldots,x^{(2)}_{n,\gamma_2(n)}]^{\SG_{\gamma_2}}
\end{align*}
in the ring $\Ho_{\Gl_{\gamma_1}}(\pt,\QQ)\otimes \Ho_{\Gl_{\gamma_2}}(\pt,\QQ)$, for each pair $(\gamma_1,\gamma_2)$, and tensoring the localised ring with $\mathcal{H}_{Q,W,\gamma_1}\otimes\mathcal{H}_{Q,W,\gamma_2}$.  At the end of this introduction we will explicitly write down the coproduct for the case $W=0$, and the localisation will hopefully begin to look more natural.

As mentioned, in general the product is not commutative, but even in the general case the construction of the coproduct is an essential step in understanding the structure of the algebra $\mathcal{H}_{Q,W}$, considered as a quantum enveloping algebra, or Yangian associated to $(Q,W)$.  The connection with the existing\footnote{Specifically, existing in \cite{MaOk12}.} use of the word Yangian is as follows.  Associated to a quiver $Q$ is an extended doubled quiver $\tilde{Q}$ given by adding an arrow $a^*$ with $s(a^*)=t(a)$ and $t(a^*)=s(a)$ for every $a\in Q_1$, and then adding a loop $\omega_i$ at the vertex $i$ for every $i\in Q_0$.  We define $W_Q=\sum_{a\in Q_1} [a,a^*]\sum_{i\in Q_0}\omega_i$ and conjecture that
\begin{equation}
\label{Yconj}
\mathcal{H}_{\tilde{Q},W_Q}\cong \mathcal{Y}^+_Q,
\end{equation}
where the right hand side is the part of the Yangian for the quiver $Q$ defined by Maulik and Okounkov in \cite{MaOk12} generated by the positive part of the Lie algebra $\mathfrak{g}[u]$.

We will leave checking the above comparison to future work, though see in this connection the paper \cite{YaZh14}.  But one point that we \textit{will} address in this paper is the conceptual gap between the definitions of the left and the right hand side of (\ref{Yconj}).  The right hand side is defined in terms of Nakajima quiver varieties, which may be considered as moduli spaces of modules over the (twisted) preprojective algebra associated to $Q$, i.e. moduli spaces of objects in a 2-dimensional Calabi--Yau (CY2) category.  On the other hand, the left hand side is built out of the cohomology of moduli spaces of representations of the Jacobi algebra $\Jac(\tilde{Q},W_Q)$, which form a CY3 category, and is furthermore not cohomology with respect to the constant sheaf, but a vanishing cycles complex.  In fact, at least at the level of Grothendieck groups, there is a very general procedure to pass from vanishing cycle cohomology of objects in a CY3 category to usual cohomology of objects in a CY2 category that goes by the name of dimensional reduction, introduced in \cite{BBS}.  In the appendix we prove the cohomological refinement of dimensional reduction that we need, which is the other main contribution of this paper towards the general project of understanding the CoHA $\mathcal{H}_{Q,W}$ in terms of quantum enveloping algebras.
\begin{thmx}
Let $f\colon X\times \mathbb{A}^n\rightarrow \mathbb{C}$ be a $\mathbb{C}^*$-equivariant regular function, where $\mathbb{A}^n$ and $\mathbb{C}$ carry the scaling action and $X$ carries the trivial action, so that we may write $f=\sum_{s=1}^nf_sx_s$ where the $x_s$ are the coordinates on $\mathbb{A}^n$ and the $f_s$ are functions on $X$.  Let $Z=Z(f_1,\ldots,f_n)$.  There is a natural isomorphism in $\Db{\mathcal{C}}$
\[
\Ho_c(X\times\mathbb{A}^n,\varphi_f\mathbb{Q}[-1])\cong\Ho_c(Z\times\mathbb{A}^n,\mathbb{Q}).
\]
\end{thmx}

As a consequence, the cohomology of other moduli spaces of objects in CY2 categories acquires the structure of a quantum enveloping algebra, for example the cohomology of the moduli stack of representations of the fundamental group of a closed Riemann surface --- we present this example in Section \ref{charSec}.  A not entirely unrelated example is the example of the compactly supported cohomology of the stack of representations of the preprojective algebra $\Pi_Q$, for an arbitrary quiver $Q$.  In \cite{Chicago3} we use dimensional reduction in this case to give a new proof of the Kac positivity conjecture regarding the coefficients of the polynomials counting absolutely indecomposable $Q$-representations over $\mathbb{F}_q$, already proved in \cite{HLRV13}.

\subsection{Why localise?  Example: the case $W=0$}
\label{NoPot}
Let $Q$ be a quiver, which for simplicity we continue to assume is symmetric, and let $B$ be the matrix defined by 
\[
b_{ij}=\delta_{ij}-\#\{a\in Q_1|s(a)=i,\hbox{ }t(a)=j\}.  
\]
Consider the \textit{non critical} cohomological Hall algebra 
\[
\mathcal{H}_Q:=\bigoplus_{\gamma\in\mathbb{N}^{Q_0}}\mathcal{H}_{Q,\gamma}, 
\]
where $\mathcal{H}_{Q,\gamma}$ is defined to be $\bigotimes_{i\in Q_0}\Cp[x_{i,1},\ldots,x_{i,\gamma(i)}]^{\SG_{\gamma(i)}}$.  Define $\mathbb{Z}_{\SC}:=\mathbb{Z}$; the subscript here is a reminder that this copy of $\mathbb{Z}$ keeps track of the cohomological grading.  The spaces $\mathcal{H}_{Q,\gamma}$ are given a $\mathbb{Z}_{\SC}$-grading by putting monomials of polynomial degree $d$ in $\mathbb{Z}_{\SC}$-degree $\chi(\gamma,\gamma)+2d$.  The $\mathbb{Z}_{\SC}$-graded multiplication operation is defined by identifying these spaces of symmetric polynomials with equivariant cohomology and then using correspondences in cohomology, as in \cite[Sec.2.2]{COHA}.  As noted in \cite[Sec.2.4]{COHA} there is, however, a direct formula for the multiplication, given as follows:
\begin{align}
\label{expform}
&m(f_1, f_2)(x_{1,1},\ldots, x_{1,\gamma_1(1)+\gamma_2(1)},\ldots,x_{n,1},\ldots,x_{n,\gamma_1(n)+\gamma_2(n)})=\\\nonumber&\sum_{\pi\in\mathcal{P}(\gamma_1,\gamma_2)}f_1(x_{1,\pi_1(1)},\ldots,x_{1,\pi_1(\gamma_1(1))},\ldots,x_{n,\pi_n(1)},\ldots,x_{n,\pi_n(\gamma_1(n))})\cdot \\&\nonumber f_2(x_{1,\pi_1(\gamma_1(1)+1)},\ldots,x_{1,\pi_1(\gamma_1(1)+\gamma_2(1))},\ldots,x_{n,\pi_n(\gamma_1(n)+1)},\ldots,x_{n,\pi_n(\gamma_1(n)+\gamma_2(n))})\cdot\\\nonumber&\prod_{i,j\in Q_0}\prod_{\alpha=1}^{\gamma_1(i)}\nolimits\prod_{\beta=\gamma_1(j)+1}^{\gamma_1(j)+\gamma_2(j)}\nolimits(x_{j,\pi_{j}(\beta)}-x_{i,\pi_{i}(\alpha)})^{-b_{ij}}
\end{align}
where $\mathcal{P}(\gamma_1,\gamma_2)$ is the subset of $\pi\in\SG_{\gamma_1+\gamma_2}$ such that for each $i\in Q_0$, $\pi_i$ preserves the ordering of $\{1,\ldots,\gamma_1(i)\}$ and $\{\gamma_1(i)+1,\ldots,\gamma_1(i)+\gamma_2(i)\}$ for each $i\in Q_0$.  In other words, $\mathcal{P}(\gamma_1,\gamma_2)$ is the set of shuffles of $(\gamma_1,\gamma_2)$ into $\gamma$.  Although it is possible for $-b_{ij}$ to be negative, it turns out that (\ref{expform}) always defines a polynomial function, so gives a well defined multiplication. 
\smallbreak 
To start with, let $Q$ be the quiver with one vertex and one loop.  Then $B=0$ and from (\ref{expform}), $\mathcal{H}_Q$ is a shuffle algebra on countably many variables.  This has an obvious coproduct, which we recall.  For $\gamma_1+\gamma_2=\gamma$, there is a natural inclusion 
\[
i_{\gamma_1,\gamma_2}\colon \mathcal{H}_{Q,\gamma}\rightarrow\mathcal{H}_{Q,\gamma_1}\otimes\mathcal{H}_{Q,\gamma_2}
\]
given by considering a $\SG_{\gamma}$ invariant polynomial as a $\SG_{\gamma_1}\times\SG_{\gamma_2}$ invariant one.  We define 
\[
\Delta=\sum_{\gamma_1+\gamma_2=\gamma}i_{\gamma_1,\gamma_2}.
\]
It is a worthwhile check to see that this really does define a compatible coassociative coproduct $\Delta$, where by compatible we mean that $\Delta$ is an algebra homomorphism, i.e. the following diagram commutes
\begin{equation}
\label{comultalg1}
\xymatrix{
\mathcal{H}_{Q}\otimes \mathcal{H}_{Q}\ar[d]^{(a,b)\mapsto a\cdot b}\ar[r]^-{\Delta\otimes\Delta}&\mathcal{H}_{Q}\otimes \mathcal{H}_{Q}\otimes\mathcal{H}_{Q}\otimes \mathcal{H}_{Q}\ar[d]^{(a,b,c,d)\mapsto (a\cdot c,b\cdot d)}\\\mathcal{H}_{Q}\ar[r]^-{\Delta}&\mathcal{H}_{Q}\otimes \mathcal{H}_{Q}.
}
\end{equation}
It is worth noting also that $\Delta$ is the same as the coproduct $\Delta'$ one gets by considering $\mathcal{H}_Q$ as a free commutative algebra in the first place (in the category of commutative unital algebras it is freely generated by the space of polynomials in one variable, i.e. $\mathcal{H}_{Q,1}$), as for each $i$ we have 
\[
\Delta'(x^i_{1,1}):=1\otimes x_{1,1}^i+x_{1,1}^i\otimes 1=\Delta(x^i_{1,1}).
\]
When we try to generalise $\Delta$, for other quivers $Q$, we need to adjust for the product terms in the last line of (\ref{expform}).  In fact it is not hard to guess what the correct approach to this should be, and the guess turns out to be almost right:
\begin{equation}
\label{naiveGuess}
\Delta=\sum_{\gamma_1+\gamma_2=\gamma}i_{\gamma_1,\gamma_2}\prod_{i,j\in Q_0}\prod_{\alpha=1}^{\gamma_1(i)}\nolimits\prod_{\beta=\gamma_1(j)+1}^{\gamma_1(j)+\gamma_2(j)}\nolimits(x_{j,\beta}-x_{i,\alpha})^{b_{ij}}.
\end{equation}
One should also introduce a sign $(-1)^{\chi(|b|,|c|)}$ into the right hand vertical map of (\ref{comultalg1}) --- see Section \ref{QLocBi}.  However the `almost' here is not referring to noncommutativity of (\ref{comultalg1}), but to the fact that $\Delta$ no longer really defines a morphism between the constituent terms of (\ref{comultalg1}).  The problem is that in changing the signs of the $b_{ij}$ in the exponent of (\ref{naiveGuess}), we lose the guarantee we had in the case of the definition of the product that the result of feeding in polynomials is a new polynomial, as opposed to a rational function.  However if we treat the shuffle product and $\Delta$ as operations on rational functions in variables $\{x_{1,1},\ldots,x_{1,\gamma(1)},\ldots,x_{n,1},\ldots,x_{n,\gamma(n)}\}$, the square (\ref{comultalg1}) commutes (after introducing the sign $(-1)^{\chi(|b|,|c|)}$).  The appearance of rational functions in the naive definition of the comultiplication generalising the natural comultiplication on the shuffle algebra explains the appearance of localisations of dual compactly supported equivariant cohomology in this paper.

\subsection{Acknowledgements}
While doing this research I was supported by the SFB/TR 45 ``Periods, Moduli Spaces and Arithmetic of Algebraic Varieties'' of the DFG (German Research Foundation), and I would like to sincerely thank Daniel Huybrechts for getting me to Bonn. During the completion of the first version of this paper, and the subsequent extensive rewrite, I was a postdoc at EPFL, supported by the Advanced Grant ``Arithmetic and physics of Higgs moduli spaces'' No. 320593 of the European Research Council. I would also like to thank Northwestern University for providing excellent working conditions during the writing of this paper, and Ezra Getzler for suggestions while I was there, and Kevin Costello for encouraging me to pursue comultiplications. Thanks also go to Davesh Maulik for helpful conversations, and J\"org Sch\"urmann for pointing out corrections regarding mixed Hodge modules. This work was partly supported by the NSF RTG grant DMS-0636646. Thanks also go to Kathy for providing a wonderful place to begin the job of writing this work down.

\section{Constructible sheaves and vanishing cycles}
\label{sheaves_section}
\subsection{Verdier duality}
\label{Vsec}
Let $X$ be a complex variety.  We denote by $\Dbc{X}$ the bounded derived category of $\QQ$ vector spaces on $X$ with constructible cohomology.  Let $f\colon X\rightarrow Y$ be a morphism of manifolds.  Then the direct image with compact support functor $f_!$ defines a functor $\Dbc{X}\rightarrow \Dbc{Y}$, with right adjoint $f^!$.  This adjoint functor is in general not the derived functor of a functor from the category of sheaves of $\mathbb{Q}$ vector spaces on $Y$ to the category of sheaves of $\mathbb{Q}$ vector spaces on $X$, though if $f$ is an affine fibration or a locally closed inclusion it is, at least up to a cohomological shift depending on the relative dimension of $f$.  
\smallbreak
The ``six functors'', namely $f_!, f^!, f_*, f^*, \SHom$, and $\otimes$ themselves descend from functors on the categories $\Db{\MHM(X)}$ and $\Db{\MHM(Y)}$, and the adjunction between $f_!$ and $f^!$ lifts to the level of mixed Hodge modules too.  For an introduction to the theory of mixed Hodge modules see \cite{Sai89}.  For the reader that prefers never to think about mixed Hodge modules, the paper can still be read, with the rule that the functor 
\[
\mathcal{F}\mapsto\mathcal{F}\state{{-d}}
\]
defined later should be understood as the functor 
\[
\mathcal{F}\mapsto\mathcal{F} \otimes\mathbb{Q}_{2d,2d},
\]
the operation of tensoring with the rational vector space concentrated in $\mathbb{Z}_{\SC}\oplus\mathbb{Z}_{\Wt}$-degree $(2d,2d)$.  Here $\mathbb{Z}_{\SC}=\mathbb{Z}$ and $\mathbb{Z}_{\Wt}=\mathbb{Z}$, and the subscripts are merely to remind us that one copy of $\mathbb{Z}$ is keeping track of cohomological degree, while the other is keeping track of the weight degree.
\smallbreak
Let $p\colon Y\rightarrow \pt$ be the projection from an equidimensional manifold $Y$ to a point.  There is an isomorphism 
\begin{equation}
\label{viso}
\triv_Y\colon  \mathbb{Q}_Y\otimes p^*\mathbb{Q}({\dim(Y)})[2\dim(Y)]\rightarrow p^!\mathbb{Q}
\end{equation}
 in the category of mixed Hodge modules on $Y$.  Here $\mathbb{Q}({\dim(Y)})$ is the pure one-dimensional weight $-2\dim(Y)$ Hodge structure.  A choice of isomorphism (\ref{viso}) corresponds canonically (since complex manifolds carry a canonical orientation), via Poincar\'e duality to a choice of a dual class to $\Ho_0(Y,\mathbb{Q})$.  There is, then, a canonical isomorphism (\ref{viso}), given by the map sending the class of any point in $\Ho_0(Y,\mathbb{Q})$ to $1$.  When considering mixed Hodge modules we define 
\[
\tate{d}:=\mathbb{Q}(d)[2d]
\]
if we consider $\mathbb{Q}$ as a one dimensional mixed Hodge structure.  If we consider just the underlying vector space we define
\[
\tate{d}:=\mathbb{Q}[2d].
\]
In either case, we define
\[
\mtate{Y}{d}:=\mathbb{Q}_Y\otimes p^*\tate{d},
\]
where again we are using the same notation for the shift, regardless of the category to which we are assuming $\mathbb{Q}_Y$ belongs.  We define the derived functor $\mathcal{F}\rightarrow\mathcal{F}\state{d}$ via
\begin{equation}
\label{statedef}
\mathcal{F}\mapsto \mathcal{F}\otimes\mtate{{Y}}{{d}}.
\end{equation}

Recall that the Verdier dual $D\mathcal{F}$ of $\mathcal{F}$ is defined to be $R\SHom(\mathcal{F},p^!\mathbb{Q})$, and so for $Y$ an equidimensional manifold there is a canonical isomorphism of functors $D\mathcal{F}\rightarrow \mathcal{F}^{\vee}\otimes\mtate{{Y}}{{\dim(Y)}}$, where $\mathcal{F}\mapsto\mathcal{F}^{\vee}$ is the duality functor $\mathcal{F}\mapsto R\SHom(\mathcal{F},\mathbb{Q}_Y)$.  The multiplication map $\mathbb{Q}_Y\otimes\mathbb{Q}_Y\rightarrow \mathbb{Q}_Y$ induces an isomorphism $\mathbb{Q}_Y^{\vee}\rightarrow\mathbb{Q}_Y$, and so we deduce that there is a canonical isomorphism 
\begin{equation}
\label{CanDuPrim}
D\mathbb{Q}_Y\rightarrow \mtate{Y}{{\dim(Y)}}
\end{equation}
which we can decompose in the following way
\begin{align}
\label{DD}
D\mathbb{Q}_Y&:=R\SHom(\mathbb{Q}_Y, p^!\mathbb{Q})
\\ &\cong R\SHom(\mathbb{Q}_Y, \mathbb{Q}_Y) \otimes p^!\mathbb{Q}\nonumber
\\ &\cong \mathbb{Q}_Y^{\vee}\otimes \mathbb{Q}_Y\state{{\dim(Y)}}\nonumber
\\ &\cong \mathbb{Q}_Y\otimes \mathbb{Q}_Y\state{{\dim(Y)}}\nonumber
\\ &\cong \mathbb{Q}_Y\state{{\dim(Y)}}\nonumber.
\end{align}
\subsection{Vanishing cycles of sheaves}
\label{vss}
We first discuss vanishing cycles of sheaves, without the added worry of Hodge structures --- that will come later.  Let $Y$ be a connected complex manifold, and let $Z\subset Y$ be a closed subspace.  Then for $\mathcal{F}$ an Abelian sheaf on $Y$, we define the underived functor
\[
\Gamma_Z \mathcal{F}(U)=\ker\big(\mathcal{F}(U)\rightarrow \mathcal{F}(U\setminus Z)\big).
\]
%The associated derived functor naturally commutes with Verdier duality --- this is easy to see after replacing a constructible sheaf $\mathcal{F}$ by its Godemont resolution and using the natural isomorphism $\Gamma_Z\left(\coprod_{x\in X} \mathcal{F}_x\right)\cong\coprod_{x\in Z}\mathcal{F}_x$.  The following diagram commutes too, for $c$ the map induced by the canonical inclusion
%\begin{equation}
%\label{okcomm}
%\xymatrix{
%R\Gamma_{Z}\mathbb{Q}_Y\state{{\dim(Y)}}\ar[d]_{\Gamma_{Z}\circ\nu}\ar[r]^-c& \mathbb{Q}_Y\state{{\dim(Y)}}\ar[d]^{\nu}
%\\
%DR\Gamma_{Z}\mathbb{Q}_Y&D\mathbb{Q}_Y\ar[l]_-{Dc}
%}
%\end{equation}
%where we have identified $R\Gamma_{Z}D$ with $DR\Gamma_{Z}$.
%\smallbreak
Let $f\colon Y\rightarrow \Cp$ be a holomorphic function.  We define
\[
\phi_f\mathcal{F}[-1]:=(R\Gamma_{\{\real(f)\leq 0\}}\mathcal{F})|_{f^{-1}(0)},
\]
the shift of the vanishing cycle functor for $f$.  We can consider the vanishing cycle functor as a functor $\phi_f\colon \Dbc{Y}\rightarrow \Dbc{Y}$ between the derived category of sheaves of $\mathbb{Q}$ vector spaces on $Y$ with constructible cohomology and itself.  It is perhaps more standard to consider $\phi_f$ as a functor $\Dbc{Y}\rightarrow \Dbc{f^{-1}(0)}$, but we prefer to keep all our sheaves as sheaves on smooth manifolds, as we use Verdier duality a great deal and always take Verdier duals in categories of sheaves on smooth manifolds.  Here we are using a nonstandard definition for $\phi_f$ that is equivalent to the usual one in the complex case (see Exercise VIII.13 of \cite{KS90}).  Often we will abbreviate $\phi_f\mathbb{Q}_Y[-1]$ to just $\varphi_f$.  The isomorphism (\ref{CanDuPrim}) induces ismorphisms in $\Dbc{Y}$
\begin{equation}
\label{CanDu}
\phi_f\mathbb{Q}_Y\state{{\dim(Y)}}\rightarrow\phi_fD\mathbb{Q}_Y.
\end{equation}
In addition there is a natural isomorphism
\begin{equation}
\label{CanDu2}
\phi_f D\cong D\phi_f
\end{equation}
by the main result of \cite{Ma09}.  The first of these isomorphisms will be used heavily in the sequel in order to construct `umkehr' maps.  For instance these umkehr maps are required for the definition of compactly supported equivariant cohomology with coefficients in the vanishing cycles complex.  
\begin{remark}
The isomorphism (\ref{CanDu2}) gives rise to an isomorphism 
\begin{equation}
\label{VDD}
\Ho_{c,\Gl_{\gamma}}(\RS_{Q,\gamma},\varphi_{\tr(W)})^{\vee}\state{{\chi(\gamma,\gamma)/2}}\cong\Ho_{\Gl_{\gamma}}(\RS_{Q,\gamma},\varphi_{\tr(W)})\state{{-\chi(\gamma,\gamma)/2}}
\end{equation}
between the cohomological Hall algebra as we define it, in terms of dual compactly supported cohomology, and the cohomological Hall algebra defined in terms of ordinary cohomology.  The reasons we prefer the slightly awkward left hand side of (\ref{VDD}) are twofold.  Firstly, the isomorphism fails after replacing $\varphi_{\tr(W)}$ with $\varphi_{\tr(W)}|_{\RS^{\Sp}_{Q,\gamma}}$, where $\RS^{\Sp}_{Q,\gamma}$ is a subvariety of $\RS_{Q,\gamma}$.  The left hand side is the true ``motivic'' invariant, in the sense of obeying cut and paste relations, after passing to classes in the Grothendieck group.  Secondly, the dimensional reduction theorem \ref{dim_red_prop} is expressed in terms of the cohomology theory on the left hand side of (\ref{VDD}).
\end{remark}
\begin{remark}
From now on all functors will be considered as derived functors unless explicitly stated, and we will abbreviate $R\Gamma_{-}$ to $\Gamma_{-}$, $Rj_*$ to $j_*$, etc.
\end{remark}

If $Y'\xrightarrow{j} Y\xrightarrow{f} B$ is a composition of maps of manifolds, and $Z\subset B$ is a closed subspace, then there is a natural transformation of functors $\Dbc{Y}\rightarrow\Dbc{Y}$
\begin{equation}
\label{pbs}
\Gamma_{f^{-1}(Z)}\rightarrow j_*\Gamma_{(fj)^{-1}(Z)}j^*
\end{equation}
which induces a natural transformation
\begin{equation}
\label{phibs}
\phi_{f}\rightarrow j_*\phi_{fj}j^*
\end{equation}
in the case that $f$ is a regular function to $\mathbb{C}$ and $Z=\mathbb{R}_{\leq 0}+i\mathbb{R}$.  The natural transformation (\ref{phibs}) is generally not an isomorphism.  On the other hand, if $j$ is a closed embedding then the composition
\begin{equation}
\label{pbr}
\Gamma_{f^{-1}(Z)}j_*\rightarrow j_*\Gamma_{(fj)^{-1}(Z)}j^*j_*\xrightarrow{\cong}j_*\Gamma_{(fj)^{-1}(Z)}
\end{equation}
is a natural isomorphism of functors $\Dbc{Y'}\rightarrow\Dbc{Y}$, and
\begin{equation}
\label{pbh}
\phi_fj_*\rightarrow j_*\phi_{fj}
\end{equation}
is a natural isomorphism of functors $\Dbc{Y'}\rightarrow\Dbc{Y}$.  Similarly, if $j$ is an affine fibration, then (\ref{phibs}) is a natural equivalence, which when evaluated on $\mathbb{Q}_Y$ gives a quasi-isomorphism of complexes of sheaves
\[
\phi_f\rightarrow j_*\phi_{fj}.
\]
Additionally for \textit{any} smooth $j$ the natural transformation
\begin{equation}
\label{commWithSmooth}
j^*\phi_f\rightarrow \phi_{jf}j^*.
\end{equation}
is a natural isomorphism.  So for instance if $j\colon Y'\rightarrow Y$ is an open embedding, there is an isomorphism $j^*\varphi_f\cong \varphi_{f|_{Y'}}$ in $\Dbc{Y'}$.
\smallbreak
It is easy to find examples showing that if $i\colon X\rightarrow Y$ is the inclusion of a closed subspace, and $f\colon Y\rightarrow \Cp$ is a regular function, then tensoring with $i_*\mathbb{Q}_X$ does not commute with taking vanishing cycles.  For example, one can show that the sheaf of vanishing cycles $\varphi_f$ is supported on the critical locus of $f$, while the complex $\phi_fi_*\mathbb{Q}_X$ has cohomology supported on the critical locus of $f|_X$ by (\ref{pbh}).  However we do have the following useful fact:
\begin{proposition}
\label{tensout}
Let $f\colon X\rightarrow\mathbb{C}$ be a regular function on a complex algebraic manifold $X$.  Let $\Dblf{X}$ be the full subcategory of the bounded derived category of sheaves of $\mathbb{Q}$ vector spaces on $X$ consisting of objects with locally constant cohomology.  There is a natural equivalence of bifunctors $\Dblf{X}\times\Dbc{X}\rightarrow\Dbc{X}$ 
\begin{equation}
\label{nuudef}
\nu\colon \big((\mathcal{L}\times \mathcal{F}\mapsto \mathcal{L}\otimes\phi_f(\mathcal{F})\big)\rightarrow \big(\mathcal{L}\times \mathcal{F}\mapsto \phi_f(\mathcal{L}\otimes\mathcal{F})\big).
\end{equation}
\end{proposition}
\begin{proof}
First note that there is a natural transformation of \textit{underived} bifunctors of Abelian sheaves
\begin{equation}
\label{prenu}
\mathcal{G}\otimes\Gamma_Z(\mathcal{F})\rightarrow \Gamma_Z(\mathcal{G}\otimes \mathcal{F}).
\end{equation}
Taking $Z=f^{-1}(\mathbb{R}_{\leq 0}+i\mathbb{R})$ and restricting (\ref{prenu}) to $f^{-1}(0)$, the induced natural transformation between associated derived functors is the natural transformation which we denote $\nu$.  The natural transformation of underived functors underlying $\nu$ is clearly an isomorphism when $\mathcal{G}$ is locally constant.  Fix a locally constant sheaf $\mathcal{L}$, it follows that (\ref{nuudef}), considered as a natural transformation of functors (with argument $\mathcal{F}$) is a natural isomorphism.  For arbitrary $\mathcal{G}\in\Dblf{X}$ it follows that if $\mathcal{G}$ has cohomology concentrated in one position then (\ref{nuudef}) is a natural isomorphism, and the general case follows from the five lemma and induction on the length of the interval $[m,m']$ in which $\mathcal{G}$ has nonzero cohomology.
\end{proof}
\begin{corollary}
\label{ppoutside}
Let $Y'\xrightarrow{j} Y$ be a smooth map of manifolds such that there is a natural equivalence $j_!j^*\mathcal{F}\rightarrow \mathcal{F}\otimes j_!\mathbb{Q}_{Y'}$,  and $j_!\mathbb{Q}_{Y'}$ has locally free cohomology, for example a map that is analytically locally a Cartesian product, or more specifically an affine fibration.  Then there is a natural equivalence $\phi_{f}j_!j^*\rightarrow j_!\phi_{fj}j^*$.
\end{corollary}
\begin{proof}
We have the sequence of natural isomorphisms
\begin{align*}
\phi_fj_!j^*\mathcal{F}\cong &\phi_f(\mathcal{F}\otimes j_!\mathbb{Q}_{Y'})
\\ \cong &\phi_f\mathcal{F}\otimes j_!\mathbb{Q}_{Y'}
\\ \cong & j_!j^*\phi_f\mathcal{F}
\\ \cong & j_!\phi_{fj}j^*\mathcal{F},
\end{align*}
where the final isomorphism follows from the smoothness of $j$.
\end{proof}
\subsection{Monodromic mixed Hodge structures}
\label{mmhs_sec}
The critical CoHA $\mathcal{H}_{Q,W}^{\Sp}$ defined below will always have an underlying algebra object in the category $\Dcc{\mathbb{N}^{Q_0}}$ of positively $\mathbb{Z}^{Q_0}$-graded vector spaces over $\mathbb{Q}$, as in the introduction, but can also be considered as a $\mathbb{Z}^{Q_0}$-graded algebra in a richer category $\D{\mathcal{C}}$.  Our favoured category $\mathcal{C}$, in the notation of the introduction, is the category of \textit{monodromic mixed Hodge structures} introduced in \cite[Sec.7.4]{COHA}.  For example, one cannot recover the theory of refined DT invariants without at least considering the extra $\mathbb{Z}_{\Wt}$-grading on $\Gr_{\Wt}(\mathcal{H}^{\Sp}_{Q,W})$ coming from the weight filtration associated to the underlying monodromic mixed Hodge structure.
\smallbreak
We will give an outline of the relevant definitions and propositions here.  For more details see \cite[Sec.7.4]{COHA}.  For many applications, the following remark will suffice: there is a full and faithful exact tensor functor $h^*s_*\colon \MHS\rightarrow \MMHS$, and in the event that all our vanishing cycles complexes $\varphi_f$ lie in the image of this functor, we may as well consider $\mathcal{H}_{Q,W}^{\Sp}$ as an algebra object in the derived category of $\mathbb{Z}^{Q_0}$-graded mixed Hodge structures.  In many interesting cases in which it is true, the statement that the relevant vanishing cycles lie in the image of $h^*s_*$ is a byproduct of the notion of dimensional reduction --- see Appendix A, and for an extended natural example of this phenomenon see Section \ref{charSec}.
\smallbreak
The category $\MHM(X)$ of mixed Hodge modules on a complex algebraic manifold $X$ is a full sub-tensor category of the category $F_W\MF_{rh}(X)$, which consists of filtered objects of $\MF_{rh}(X)$, which itself consists of triples:
\begin{enumerate}
\item
A perverse sheaf $L$ of $\mathbb{Q}$ vector spaces on $X$.
\item
A regular holonomic $D_X$-module $M$ with an isomorhism $DR(M)\cong L\otimes_{\mathbb{Q}}\mathbb{C}$.
\item
A good filtration $F$ on $M$.
\end{enumerate}
Actually describing which objects of $F_W\MF_{rh}(X)$ belong to $\MHM(X)$ is a rather complicated matter, to which we refer the reader to \cite{Sai89}.  There is a forgetful functor $\rat\colon \Der(\MHM(X))\rightarrow \Dcc{X}$, given by remembering only the underlying complex of perverse sheaves, and this functor is exact (when the target triangulated category is given the perverse t structure) and the restricted functor $\MHM(X)\rightarrow \Perv(X)$ is faithful.
\smallbreak 
The category $\MMHS$ is the full subcategory of $\EMHS$ consisting of mixed Hodge modules unramified on $\mathbb{C}^*\subset \mathbb{A}^1$, where $\EMHS$ is the category of \textit{exponential mixed Hodge structures}.  The category $\EMHS$ is in turn defined as the full subcategory of $\MHM(\mathbb{A}^1)$ containing $L\in\Ob(\MHM(\mathbb{A}^1))$ such that the corresponding perverse sheaf $\rat(L)$ satisfies $R\Gamma(\rat(L))=0$.  Let 
\[
\Sum\colon \mathbb{A}^1\times\mathbb{A}^1\rightarrow\mathbb{A}^1
\]
be the addition morphism, then the tensor product in $\EMHS$ is defined as 
\begin{equation}
\label{SMS}
\mathcal{F}_1\boxtimes_+\mathcal{F}_2:=\Sum_*(\mathcal{F}_1\boxtimes\mathcal{F}_2).  
\end{equation}
This restricts to a tensor product on $\MMHS$.  Let 
\[
j\colon \mathbb{A}^1\setminus\{0\}\rightarrow\mathbb{A}^1
\]
be the inclusion.  There is an obvious inclusion of categories $h_*\colon \EMHS\rightarrow\MHM(\mathbb{A}^1)$, with $h^*:=\Sum_*(-\boxtimes j_!\mathbb{Q}_{\mathbb{A}^1\setminus \{0\}}[1])$ its left adjoint.  The functor $h^*$ is a tensor functor, where $\MHM(\mathbb{A}^1)$ is given the symmetric monoidal structure (\ref{SMS}).  There is a tensor functor $s_*\colon \MHS\rightarrow \MHM(\mathbb{A}^1)$ given by pushforward along the inclusion 
\[
s\colon \{0\}\rightarrow\mathbb{A}^1
\]
and so a tensor functor 
\begin{equation}
\label{MMHSinc}
h^*s_*\colon \MHS\rightarrow\EMHS.
\end{equation}
In fact one can easily see that $h^*s_*$ defines a functor to $\MMHS$.  In this way we realise the category of mixed Hodge strcutures as a subcategory of the category of monodromic mixed Hodge structures.  The weight filtration for $\mathcal{F}\in\EMHS$ is defined by $W^{\EMHS}_{\leq m}\mathcal{F}:=h^*W_{\leq m}^{\MHM(\mathbb{A}^1)}h_*\mathcal{F}$.
\begin{definition}
\label{shpure}
We say an object of $\Db{\EMHS}$ is pure if its $m$th cohomology is pure of weight $m$.
\end{definition}
\smallbreak
Since all our exponential mixed Hodge structures will in fact belong to the subcategory of monodromic mixed Hodge structures, the fibre functor to vector spaces admits a simple description: let $i\colon \{1\}\rightarrow \mathbb{A}^1$ be the inclusion of the point, then 
\[
\rat i^*[-1]\colon \MMHS\rightarrow \Vect
\]
provides a fibre functor (i.e. this functor commutes with tensor products and their symmetry isomorphisms, and is exact and faithful).  
\begin{definition}
If $V$ is a vector space, a monodromic mixed Hodge structure on $V$ consists of an object $\tilde{V}\in\MMHS$ and an isomorphism $\rat i^*\tilde{V}[-1]\cong V$.  
\end{definition}
If $f\colon X\rightarrow \Cp$ is a holomorphic function on a smooth manifold $X$, and $\mathcal{F}$ is an object of $\MHM(X)$, we define the element $\phi_f\mathcal{F}[-1]\in\Ob(\MHM(X))$ as in \cite{Sai90}.  Where there is no room for confusion we will abbreviate $\phi_f\mathbb{Q}_X[-1]$ to $\varphi_f$.
\smallbreak
Let $Z\subset X$ be a subvariety, and let $f$ be a function on $X$.  The cohomology $\Ho_{c}(Z,\varphi_f)$ is given a monodromic mixed Hodge structure as follows.  Let $u$ be the coordinate on $\Cp^*$, and consider the following object of $\Db{\MHM(\mathbb{A}^1)}$:
\[
(\Cp^*\rightarrow\mathbb{A}^1)_!(Z\times\Cp^*\rightarrow \Cp^*)_!(Z\times \Cp^*\rightarrow X\times \Cp^*)^*\phi_{f/u}\mathbb{Q}_{X\times \Cp^*}\state{0}.
\]
This is an object $L$ of $\Db{\MMHS}\subset\Db{\EMHS}$, and we have a natural isomorphism 
\[
\rat i^*L[-1]\cong\Ho_{c}(Z,\varphi_f)
\]
giving $\Ho_{c}(Z,\varphi_f)$ the structure of a monodromic mixed Hodge structure as in \cite{COHA}.  Note here that $f$ is considered throughout as a function on $X$, not $Z$.

\begin{example}We define the Verdier duality functor $D_{\mon}$ on $\MMHS$ to be $h^*Dh_*$, where $D$ is the usual Verdier duality functor for $\MHM(\mathbb{A}^1)$.  Let $\mathcal{F}\in\Ob(\MMHS)\subset \Ob(\MHM(\mathbb{A}^1))$ be a monodromic mixed Hodge structure, providing a monodromic mixed Hodge structure on $H=\rat i^* \mathcal{F}[-1]$.  Then it is easy to check that $H^{\vee}\cong \rat i^*(D_{\mon}\mathcal{F})[-1]$, providing a monodromic mixed Hodge structure for $H^{\vee}$, the vector dual of $H$.
\end{example}

\begin{proposition}\cite{Sai90}
\label{cruclift}
Let $X'\rightarrow X$ be a morphism of complex manifolds, and let $X\rightarrow\Cp$ be a holomorphic function.  The natural transformation (\ref{phibs}), applied to $\mathbb{Q}_X$, lifts to a morphism of mixed Hodge modules.
\end{proposition}
For the projective case, this is as in \cite[Thm.2.14]{Sai90}, while for the case of an open embedding one defines $f_*$ as in \cite[Thm.4.3]{Sai90} where it is a simple check that the Kashiwara-Malgrange filtration for $D$-modules commutes with pushforward --- see \cite[Thm.III.4.10.1]{Meb89} and its proof.  

By Proposition \ref{cruclift} and faithfulness of the fibre functor, all the isomorphisms of the previous section can (and will) be lifted to isomorphisms of monodromic mixed Hodge structures.

\subsection{Compactly supported equivariant cohomology}
Assume that a complex algebraic manifold $Y$ carries a $G$-action, where $G$ is a complex algebraic group, and we are given an inclusion $G\subset \GL_{\Cp}(n)$ of complex algebraic groups for some $n$.  We will assume that every $y\in Y$ is contained in a $G$-equivariant open affine neighbourhood.  For $N\geq n$ we define $\fr(n,N)$ to be the space of $n$-tuples of linearly independent vectors in $\Cp^N$, and we define $(Y,G)_N:=Y\times\fr(n,N)$ and $\overline{(Y,G)}_N:= Y\times_{G}\fr(n,N)$, i.e. the quotient under the action $g\cdot (y,h)=(g\cdot y,g^{-1}h)$.  This quotient exists in the category of schemes by \cite[Prop.23]{EdGr98}.  If we assume furthermore that $f\colon Y\rightarrow B$ is a $G$-invariant map to a complex variety $B$, and $Z\subset B$ is a closed subvariety, we obtain maps
\begin{equation}
f_N\colon \overline{(Y,G)}_N\rightarrow B
\end{equation}
and objects $R\Gamma_{f_N^{-1}(Z)}(\mathbb{Q}_{\overline{(Y,G)}_N})$ on each of the spaces $\overline{(Y,G)}_N$.  There is a natural inclusion $\Cp^{N}\rightarrow \Cp^{N+1}$ sending $(x_1,\ldots,x_N)\mapsto (x_1,\ldots,x_N,0)$, inducing maps $\fr(n,N)\rightarrow \fr(n,N+1)$ and 
\[
i_N\colon \overline{(Y,G)}_N\rightarrow \overline{(Y,G)}_{N+1}.  
\]
Now we consider the case $B=\mathbb{C}$ and $Z=\mathbb{R}_{\leq 0}\times i\mathbb{R}$.  Combining (\ref{CanDu}) and (\ref{pbh}), where we use also that $i_{N,*}=i_{N,!}$ since $i_N$ is proper, there are maps
\begin{align}
\label{Gyspf}
&i_{N,!}\phi_{f_N}\mathbb{Q}_{\overline{(Y,G)}_N}\state{{\dim(\overline{(Y,G)}_N)}}\rightarrow i_{N,!}\phi_{f_{N}}D\mathbb{Q}_{\overline{(Y,G)}_N}\rightarrow \\&\phi_{f_{N+1}}D i_{N,!} \mathbb{Q}_{\overline{(Y,G)}_N}\xrightarrow{e} \phi_{f_{N+1}}D\mathbb{Q}_{\overline{(Y,G)}_{N+1}}\rightarrow\phi_{f_{N+1}}\mathbb{Q}_{\overline{(Y,G)}_{N+1}}\state{{\dim(\overline{(Y,G)}_{N+1})}}\nonumber
\end{align}
where $e$ is defined by applying Verdier duality to the natural map 
\[
\mathbb{Q}_{\overline{(Y,G)}_{N+1}}\rightarrow i_{N,*}\mathbb{Q}_{\overline{(Y,G)}_{N}}.
\]
Let $Y'$ be a sub $G$-equivariant subvariety of $Y$, and define $Y'_N\subset \overline{(Y,G)}_N$ to be the subspace of points projecting to $Y'$.  Applying $(Y'_{N+1}\hookrightarrow \overline{(Y,G)}_{N+1})^*$ to (\ref{Gyspf}) and applying $a_{N+1,!}$, where $a_{N+1}\colon Y'_{N+1}\rightarrow \pt$ is the projection to a point, we obtain maps
\[
\Ho_c(Y'_N,\varphi_{f_N})\state{{\dim\left(\overline{(Y,G)}_N\right)}}\rightarrow \Ho_c(Y'_{N+1},\varphi_{f_{N+1}})\state{{\dim\left(\overline{(Y,G)}_{N+1}\right)}},
\]
and we define
\begin{equation}
\label{eqvcdef}
\Ho_{c,G}(Y',\varphi_f):=\varinjlim\left(\Ho_c(Y'_N,\varphi_{f_N})\state{{\dim(\fr(n,N))}}\right).
\end{equation}
This limit makes sense since $\Ho_c(Y'_N,\varphi_{f_N})\state{{\dim(\fr(n,N))}}$ stabilises in each cohomological degree.
\begin{remark}
\label{dimconv}
We have picked the unique normalizing $\tate{{-}}$ twist in (\ref{eqvcdef}) such that if $Y$ is acted on freely by $G$ in a $G$-equivariant open neighbourhood $U$ of $Y'$, we recover the usual definition of the compactly supported cohomology of $\varphi_g$ on $Y'/G$, where $g$ is the induced function on $U/G$.
\end{remark}
We have defined compactly supported equivariant cohomology in an analogous way to the perhaps more familiar definition of equivariant cohomology, which we recall also.  We assume that we have the same setup as above, but this time instead of (\ref{Gyspf}) consider the map
\begin{align}
\label{nGyspb}
\phi_{f_{N+1}}\mathbb{Q}_{\overline{(Y,G)}_{N+1}}\rightarrow \phi_{f_{N+1}}i_{N,*}\mathbb{Q}_{\overline{(Y,G)}_{N}}\rightarrow i_{N,*}\phi_{f_{N}}\mathbb{Q}_{\overline{(Y,G)}_{N}}.
\end{align}
Applying $a_{N,*}(Y'_N\hookrightarrow\overline{(Y,G)}_N)^*$ to the composition (\ref{nGyspb}) we obtain maps
\[
\Ho(Y'_{N+1},\varphi_{f_{N+1}})\rightarrow\Ho(Y'_{N},\varphi_{f_N})
\]
and we define $\Ho_G(Y',\varphi_f):=\varprojlim\left(\Ho(Y'_N,\varphi_{f_N})\right)$.
\subsection{Thom--Sebastiani isomorphism}
Let $Y_1$ and $Y_2$ be a pair of complex algebraic manifolds, acted on by complex algebraic groups $G_1$ and $G_2$ respectively, where again we have embeddings $G_i\subset \GL_{\Cp}(n_i)$.  Let $f_1\colon Y_1\rightarrow \Cp$ and $f_2\colon Y_2\rightarrow \Cp$ be $G_1$- and $G_2$-invariant functions, respectively.  The inclusion of closed subsets 
\[
(f_1)_N^{-1}(\mathbb{R}_{\leq 0}+i\mathbb{R})\times (f_2)_N^{-1}(\mathbb{R}_{\leq 0}+i\mathbb{R})\subset ((f_1)_N+ (f_2)_N)^{-1}(\mathbb{R}_{\leq 0}+i\mathbb{R})\subset \overline{(Y_1,G_1)}_N\times \overline{(Y_2,G_2)}_N
\]
induces a morphism of objects in $\Dbc{\overline{(Y_1,G_1)}_N\times \overline{(Y_2,G_2)}_N}$
\[
R\Gamma_{\{\real((f_1)_N)\leq 0\}}\mathbb{Q}_{\overline{(Y_1,G_1)}_N}\boxtimes R\Gamma_{\{\real((f_2)_N)\leq 0\}}\mathbb{Q}_{\overline{(Y_2,G_2)}_N}\rightarrow R\Gamma_{\{\real((f_1)_N\boxplus(f_2)_N)\leq 0\}}\mathbb{Q}_{\overline{(Y_1,G_1)}_N\times\overline{(Y_2,G_2)}_N}
\]
inducing a map
\[
\Ho_{c}((f_1)_N^{-1}(0),\varphi_{(f_1)_N})\otimes\Ho_{c}((f_2)_N^{-1}(0),\varphi_{(f_2)_N})\rightarrow \Ho_{c}(((f_1)_N^{-1}(0)\times (f_2)_N^{-1}(0)),\varphi_{(f_1\boxplus f_2)_N})
\]
and by the main result of \cite{Ma01} this is an isomorphism, inducing an isomorphism of cohomologically graded vector spaces
\begin{equation}
\label{TSisopre}
\Ho_{c,G_1}((Y_1)_0,\varphi_{f_1})\otimes\Ho_{c,G_2}((Y_2)_0,\varphi_{f_2})\rightarrow \Ho_{c,G_1\times G_2}((Y_1)_0\times (Y_2)_0,\varphi_{f_1\boxplus f_2})
\end{equation}
in the limit.  Here $(Y_i)_0:=f_i^{-1}(0)$.  
\begin{remark}
\label{critzero}
Throughout the paper we assume that if $f\colon Y\rightarrow \mathbb{C}$ is a regular function, the critical locus of $f$ is contained in $f^{-1}(0)$.  If this assumption does not hold, one can shrink $Y$ so that it does.  This assumption, along with the Thom--Sebastiani isomorphism, implies that there is a natural isomorphism 
\begin{equation}
\label{TSiso}
\Ho_{c,G_1}(Y_1,\varphi_{f_1})\otimes\Ho_{c,G_2}(Y_2,\varphi_{f_2})\rightarrow \Ho_{c,G_1\times G_2}(Y_1\times Y_2,\varphi_{f_1\boxplus f_2}).
\end{equation}
\end{remark}
\begin{remark}
\label{Tis}
The morphism (\ref{TSiso}) respects the symmetric monoidal structures on $\mathbb{Z}_{\Co}$-graded spaces (with tensor product) and complex analytic varieties equipped with a regular function (with product $(X,f)\times (Y,g):=(X\times Y,f\boxplus g)$).  In other words, they make $(Y,f)\mapsto \Ho_c(Y,\varphi_f)$ into a symmetric monoidal functor.  The symmetric structure, in particular, becomes important when considering the version of the integrality conjecture and its proof found in \cite[Sec.3.2]{DaMe15b}.
\end{remark}
\begin{remark}
\label{MMHS_rem}
In order for the cohomological Hall algebra, defined in Section \ref{coha_section}, to be an algebra object in the category of monodromic mixed Hodge structures, one needs to upgrade (\ref{TSiso}) to an isomorphism in $\MMHS$.  An easy calculation, setting $Y_1=Y_2=\mathbb{A}^1$, $G_1=G_2=\{\id\}$ and $f_1=f_2=x^2$, shows that (\ref{TSiso}) is in fact not an isomorphism for the usual category of mixed Hodge structures, with its usual tensor product, so it is indeed necessary to work in $\MMHS$ and not $\MHS$.  By unpublished work of Saito, see also the comparison result of Sch\"urmann found in the Appendix of \cite{Br12}, (\ref{TSiso}) is indeed an isomorphism in $\MMHS$.
\end{remark}
Working with unpublished results is not ideal; there are three solutions to this situation.
\begin{enumerate}
\item
One can go ahead and consider the CoHA as an algebra object in the category of monodromic mixed Hodge structures.
\item
One can forget the Hodge structure on both sides of (\ref{TSiso}) and consider it as an isomorphism in $\Dbc{\pt}$.
\item
One can assume that both sides of (\ref{TSiso}) come from genuine Hodge structures, in the sense that they lie in the image of the map (\ref{MMHSinc}).  This is indeed a safe assumption in a wide range of cases, most notably those coming from ``dimensional reduction'' --- see Appendix A.  If the functions $f_1$ and $f_2$ satisfy the assumptions of Theorem \ref{dim_red_prop}, then the Thom--Sebastiani isomorphism (at the level of mixed Hodge structures) is a direct consequence of the Kunneth isomorphism (see Proposition \ref{TScomm}).  These cases include, for example, the noncommutative conifold, and enough examples to reprove the Kac positivity conjecture (see \cite{Chicago3}) and analyse Hodge structures on twisted and untwisted character varieties using CoHAs (see Section \ref{charSec}).
\end{enumerate}
\subsection{The $\Ho_{\Gl_{\gamma}}(\pt,\QQ)$-module structure}  \label{moduleStruc}Let $X$ be a $G$-equivariant complex variety, where $G$ is a complex algebraic group.  The vector space $\Ho_{G}(X,\mathbb{Q})$ is an algebra via the usual cohomology operations, and $\Ho_{G}(X,\mathbb{Q})$ is a module over $\Ho_{G}(\pt,\mathbb{Q})$.  For example, each $\gamma$-graded piece $\mathcal{H}_{Q,\gamma}$ of the \textit{non critical} cohomological Hall algebra $\mathcal{H}_Q$ of Section \ref{NoPot} carries a $\Ho_{\Gl_{\gamma}}(\pt,\mathbb{Q})$-action, since in fact there is an isomorphism in cohomology $\mathcal{H}_{Q,\gamma}\cong\Ho_{\Gl_{\gamma}}(\pt,\mathbb{Q})$.  Sticking with quivers without potentials, consider the spaces $\Ho_{c,\Gl_{\gamma}}(\RS_{Q,\gamma},\mathbb{Q}_{\RS_{Q,\gamma}})^{\vee}$, the vector space dual of compactly supported equivariant cohomology.
  Since 
\[
\dim(\overline{(\RS_{Q,\gamma},\Gl_{\gamma})}_N)=\dim(\RS_{Q,\gamma})+\dim(\fr(n,N))-\dim(\Gl_{\gamma})
\]
we have isomorphisms 
\[
\mathbb{Q}_{\overline{(\RS_{Q,\gamma},\Gl_{\gamma})}_N}\state{{\dim(\RS_{Q,\gamma})}+\dim(\fr(n,N))-\dim(\Gl_{\gamma})}\rightarrow D\mathbb{Q}_{\overline{(\RS_{Q,\gamma},\Gl_{\gamma})}_N}
\]
and via the natural isomorphisms $Da_{N,!}D\cong a_{N,*}$ we obtain an isomorphism
\begin{equation}
\label{ccshift}
\Ho_{c,\Gl_{\gamma}}(\RS_{Q,\gamma},\mathbb{Q}_{\RS_{Q,\gamma}})^{\vee}\rightarrow \Ho_{\Gl_{\gamma}}(\RS_{Q,\gamma},\mathbb{Q}_{\RS_{Q,\gamma}})\state{{-\chi(\gamma,\gamma)}},
\end{equation}
and so we deduce that the \textit{dual} of the compactly supported cohomology with trivial coefficients is the space that naturally inherits the $\Ho_{\Gl_{\gamma}}(\pt,\mathbb{Q})$-action.  Given a $G$-equivariant complex algebraic manifold $X$ and a $G$-invariant regular function $f\in\Gamma(X,\mathbb{C})^G$ we construct a $\Ho_{G}(\pt,\mathbb{Q})$-action on $\Ho_{c,G}(X,\varphi_f)^{\vee}$ that becomes the above action in the special case $G=\Gl_{\gamma}$, $X=\RS_{Q,\gamma}$ and $f=0$.

Define
\begin{align*}
\Delta_N\colon &(X\times_{G}\fr(n,N))\rightarrow (X\times_G \fr(n,N))\times (\pt\times_{G}\fr(n,N))\\ &(x,z)\mapsto ((x,z), (\pt,z)).
\end{align*}
Let $B_N$ be the target of $\Delta_N$ and $A_N$ be the domain.  Let $f_N$ be the function induced by $f$ on $A_N$, and let $g_N$ be the function induced by $f$ on $B_N$.  Then applying Verdier duality to
\[
\mathbb{Q}_{B_N}\rightarrow \Delta_{N,*}\mathbb{Q}_{A_N}
\]
and applying the vanishing cycle functor we obtain a morphism
\[
\phi_{g_N}\big(\Delta_{N,!}D\mathbb{Q}_{A_N}\rightarrow D\mathbb{Q}_{B_N}\big),
\]
and via (\ref{CanDu}), morphisms
\[
\phi_{g_N}\big(\Delta_{N,!}\mathbb{Q}_{A_N}\rightarrow \mathbb{Q}_{B_N}\state{{\dim(\fr(n,N)/G)}}\big).
\]
As $\Delta_N$ is proper we obtain morphisms (see the discussion around (\ref{pbh}))
\[
\Delta_{N,!}\varphi_{f_N}\rightarrow\varphi_{g_N}\state{{\dim(\fr(n,N)/G)}},
\]
and taking the dual of compactly supported cohomology, morphisms
\begin{equation}
\label{athere}
\Ho_{c}\big(\overline{(X,G)}_N\times \overline{(\pt,G)}_N,\varphi_{g_N}\big)\state{{\dim(\fr(n,N)/G)}}^{\vee}\rightarrow \Ho_{c}(\overline{(X,G)}_N,\varphi_{f_N})^{\vee}.
\end{equation}
Applying the Thom-Sebastiani isomorphism, observing that $g_N=f_N\boxplus 0$, we may rewrite the left hand side of (\ref{athere}) as
\[
\Ho_{c}(\overline{(X,G)}_N,\varphi_{f_N})^{\vee}\otimes \Ho_c(\overline{(\pt,G)}_N,\mathbb{Q})\state{{\dim(\fr(n,N)/G)}}^{\vee}
\]
and so as with (\ref{ccshift}) we obtain a morphism
\[
\Ho_{c}(\overline{(X,G)}_N,\varphi_{f_N})^{\vee}\otimes \Ho(\overline{(\pt,G)}_N,\mathbb{Q})\rightarrow \Ho_{c}(\overline{(X,G)}_N,\varphi_{f_N})^{\vee}
\]
which gives the action of $\Ho_{G}(\pt,\mathbb{Q})$ on $\Ho_{c,G}(X,\varphi_{f})^{\vee}$.
\smallbreak
Now let 
\[
\overline{\Delta}_N\colon (X\times_G \fr(n,N))\rightarrow (X\times_G \fr(n,N))\times (X\times_G \fr(n,N))
\]
be the diagonal embedding.  Denote the target by $\overline{B}_N$.  We define a function $\overline{g}_N=f_N\pi_1$, where $\pi_1$ is projection onto the first factor of $\overline{B}_N$.  Then again, $\overline{g}_N\overline{\Delta}_N=f_N$, and we build in the same way an extended action
\begin{equation}
\label{extendedAction}
\Ho_{G}(X,\mathbb{Q})\otimes\Ho_{c,G}(X,\varphi_f)^{\vee}\rightarrow \Ho_{c,G}(X,\varphi_f)^{\vee}.
\end{equation}
In many of our applications, $X$ will be $G$-equivariantly contractible, and there will be no difference between the two actions.

\subsection{Umkehr maps in localised compactly supported cohomology}\label{umkehr_sec}
Let $g\colon X\rightarrow Y$ be a morphism of complex algebraic manifolds, let $f$ be a regular function on $Y$, and let $Y^{\Sp}\subset Y$ be a submanifold, and denote $X^{\Sp}=g^{-1}(Y^{\Sp})$.  Then we have a chain of morphisms
\begin{align*}
g_!\varphi_{fg}\cong& Dg_*D\varphi_{fg}\\
\cong&Dg_*\phi_{fg}D\QQ_X[-1]&\textrm{using (\ref{CanDu2})}\\
\rightarrow &D\phi_f g_*D\QQ_X[-1]&\textrm{using (\ref{pbh})}\\
\cong &\phi_fDg_*D\QQ_X[-1]&\textrm{using (\ref{CanDu2})}\\
\cong &\phi_f g_!\QQ_X[-1]
\end{align*}
and composing with the map $\phi_fg_!\QQ_X[-1]\rightarrow \phi_f\QQ_Y[-1]\state{{\dim(Y)-\dim(X)}}$ obtained by applying $\phi_f$ to the shifted Verdier dual of the adjunction map $\QQ_Y[-1]\rightarrow g_*\QQ_X[-1]$, and applying $D\circ(Y^{\Sp}\rightarrow \pt)_!(Y^{\Sp}\rightarrow Y)^*$ to the resulting map, we obtain the pullback map
\[
\Ho_{c}(Y^{\Sp},\varphi_{f})^{\vee}\state{{\dim(X)-\dim(Y)}}\rightarrow \Ho_c(X^{\Sp},\varphi_{gf})^{\vee}.
\]

Associated to maps $g\colon X\rightarrow Y$ of $G$-equivariant complex algebraic manifolds, with $G\subset \GL_{\Cp}(n)$ a complex algebraic group and $f$ a $G$-invariant function on $Y$, we will often want to associate maps going \textit{both} ways between $\Ho_{c,G}(X^{\Sp},\varphi_{fg})^{\vee}$ and $\Ho_{c,G}(Y^{\Sp},\varphi_{f})^{\vee}$.  The maps $g$ for which we wish to do this fall into essentially two different types.
\smallbreak
Firstly, let $\pi\colon X\rightarrow Y$ be a $G$-equivariant affine fibration.  Then the pullback map
\begin{equation}
\label{pbaff}
\pi^*\colon \Ho_{c,G}(Y^{\Sp},\varphi_f)\state{{-{\dim(\pi)}}}^{\vee}\rightarrow \Ho_{c,G}(X^{\Sp},\varphi_{f\pi})^{\vee} 
\end{equation}
is an isomorphism.  In more detail: denote by $\pi_N$ the natural projection $\overline{(X,G)}_N\rightarrow\overline{(Y,G)}_N$.  There is a natural isomorphism
\begin{equation}
\label{startout}
\mathbb{Q}_{\overline{(Y,G)}_N}\rightarrow (\pi_N)_*\mathbb{Q}_{\overline{(X,G)}_N}.
\end{equation}
Applying $\phi_{f_N}$ to the Verdier dual of (\ref{startout}) we obtain maps
\[
\phi_{f_N}\left(\pi_{N,!}D\mathbb{Q}_{\overline{(X,G)}_N}\rightarrow D\mathbb{Q}_{\overline{(Y,G)}_N}\right).
\]
By (\ref{CanDu}) this gives us an isomorphism
\begin{equation}
\label{otw}
\phi_{f_N}\left(\pi_{N,!}\mathbb{Q}_{\overline{(X,G)}_N}\rightarrow \mathbb{Q}_{\overline{(Y,G)}_N}\state{{-{\dim(\pi)}}}\right).
\end{equation}
From Corollary \ref{ppoutside} and (\ref{otw}) we obtain an isomorphism
\[
\pi_{N,!}\varphi_{f_N\pi_N}\rightarrow\varphi_{f_N}\state{{-{\dim(\pi)}}}
\]
and restricting to $\overline{(Y^{\Sp},G)}_N$, taking compactly supported cohomology, passing to the limit, and taking duals, the isomorphism (\ref{pbaff}).

\smallbreak
We define the Euler characteristic of $\pi$ as follows.  Let $V=T_{X/Y}$ be the relative tangent bundle of $\pi$.  Let $z\colon X\rightarrow V$ be the inclusion of the zero section.  Then consider the composition
\begin{equation}
\label{eudef}
z_*\mathbb{Q}_{\overline{(X,G)}_N}\rightarrow \mathbb{Q}_{\overline{(V,G)}_N}\state{{{\dim(\pi)}}}\rightarrow z_*\mathbb{Q}_{\overline{(X,G)}_N}\state{{{\dim(\pi)}}}
\end{equation}
where the first morphism is obtained by taking the Verdier dual of the second.  We define
\[
\eu(\pi)\cdot\colon \Ho_G(X,\mathbb{Q})\rightarrow\Ho_G(X,\mathbb{Q})
\]
by taking cohomology of (\ref{eudef}), and abbreviate $\eu(\pi):=\eu(\pi)\cdot 1$.
By abuse of notation we will consider $\eu(\pi)$ as an element of $\Ho_G(Y,\mathbb{Q})$ as well via the natural isomorphism $\Ho_G(Y,\mathbb{Q})\cong\Ho_G(X,\mathbb{Q})$ induced by $\pi$.  By applying $\phi_f$ to (\ref{eudef}), via the isomorphism $\Ho_G(Y,\varphi_f)\cong \Ho_G(X,\varphi_{f\pi})$, we obtain the map
\[
\eu(\pi)\cdot\colon \Ho_G(Y,\varphi_f)\rightarrow\Ho_G(Y,\varphi_f).
\]
The notation is justified by the first part of Proposition \ref{maneq}.  We make the further assumption that the Euler characteristic of $\pi$ is not a zero divisor in $\Ho_{c,G}(Y,\varphi_f)^{\vee}$ for the extended action (\ref{extendedAction}), and define the \textit{pushforward map} associated to $\pi$ to be
\[
\pi_*:=(\pi^*)^{-1}\cdot \eu(\pi)^{-1}\colon \Ho_{c,G}(X^{\Sp},\varphi_{f\pi})^{\vee}\rightarrow\Ho_{c,G}(Y^{\Sp},\varphi_f)^{\vee}[\eu(\pi)^{-1}].
\]
\begin{remark}
Note that the pushforward map preserves degree.
\end{remark}

Assume instead that $p\colon X\rightarrow Y$ is a proper map of $G$-equivariant complex algebraic manifolds, inducing proper maps $p_N\colon \overline{(X,G)}_N\rightarrow\overline{(Y,G)}_N$.  Then later in (\ref{pfdef}) we use the pushforward
\[
p_*\colon \Ho_{c,G}(X^{\Sp},\varphi_{fp})^{\vee}\rightarrow \Ho_{c,G}(Y^{\Sp},\varphi_f)^{\vee}
\]
defined via the maps
\[
\phi_{f_N}(\mathbb{Q}_{\overline{(Y,G)}_N}\rightarrow p_{N,*}\mathbb{Q}_{\overline{(X,G)}_N})
\]
and the isomorphism $\phi_{f_N}p_{N,*}\mathbb{Q}_{\overline{(X,G)}_N}\cong p_{N,*}\phi_{f_Np_N}\mathbb{Q}_{\overline{(Y,G)}_N}$ of Corollary \ref{ppoutside}, using that $p_{N,*}\cong p_{N,!}$ since each $p_N$ is proper.
\smallbreak

\begin{proposition}
\label{mixingprop}
Let
\begin{equation}
\label{cornersquare}
\xymatrix{
X\ar[r]^-{g}\ar[d]^j& X'\ar[d]^{j'}
\\
Y\ar[r]^{h}&Y'
}
\end{equation}
be a Cartesian diagram of complex connected algebraic manifolds in which $g$ and $h$ are either proper, or affine fibrations.  Assume there is a natural isomorphism $\upsilon\colon g^*j'^!\cong j^!h^*$ such that the diagrams
\begin{equation}
\label{fcomp}
\xymatrix{
j^*h^*\state{\mathrm{reldim}(j)}\ar[r]\ar[dr]&j^!h^*
\\& g^*j'^!\ar[u]^{\upsilon}
}
\end{equation}
and
\begin{equation}
\label{verdCompat}
\xymatrix{
h^*j'_!j'^!\ar[d]\ar[r]&j_!j^!h^*\ar[dl]\\
h^*
}
\end{equation}
commute.  The existence of $\upsilon$ implies that $\mathrm{reldim}(j)=\mathrm{reldim}(j')$; the diagonal morphism in (\ref{fcomp}) is the resulting natural isomorphism.  If $g$ and $h$ are affine fibrations, we assume invertibility of the corresponding Euler classes.  Let $f$ be a holomorphic function on $Y'$.  Let $Y'^{\Sp}\subset Y'$ be a subvariety, and define $Y^{\Sp}:=h^{-1}Y'^{\Sp}$, $X'^{\Sp}:=j'^{-1}Y'^{\Sp}$ and $X^{\Sp}:=j^{-1}Y^{\Sp}$.  Then either $h$ and $g$ are proper and the following diagram commutes
\begin{equation}
\label{descom}
\xymatrix{
\Ho_{c}(X^{\Sp},\varphi_{fj'g})^{\vee}\ar[r]^{g_*}&\Ho_c(X'^{\Sp},\varphi_{fj'})^{\vee}
\\
\Ho_c(Y^{\Sp},\varphi_{fh})^{\vee}\ar[u]^{j^*}\ar[r]^{h_*}&\Ho_c(Y'^{\Sp},\varphi_f)^{\vee}\ar[u]^{j'^*},
}
\end{equation}
or $h$ and $g$ are affine fibrations and the following diagram commutes
\begin{equation}
\label{descom2}
\xymatrix{
\Ho_{c}(X^{\Sp},\varphi_{fj'g})^{\vee}\ar[r]^-{g_*}&\Ho_c(X'^{\Sp},\varphi_{fj'})^{\vee}[\eu(g)^{-1}]
\\
\Ho_c(Y^{\Sp},\varphi_{fh})^{\vee}\ar[u]^{j^*}\ar[r]^-{h_*}&\Ho_c(Y'^{\Sp},\varphi_f)^{\vee}[\eu(h)^{-1}]\ar[u]^{j'^*}.
}
\end{equation}

\end{proposition}
\begin{proof}
If $h$ is an affine fibration, then since (\ref{cornersquare}) is Cartesian, we have $j'^*\eu(h)=\eu(g)$ where we abuse notation and denote by $j'^*$ also the usual map $j'^*\colon \Ho(Y',\mathbb{Q})\rightarrow\Ho(X',\mathbb{Q})$; we define $j'^*$ in the rightmost column by extension of scalars.  It follows that it is enough to prove that the diagram
\[
\xymatrix{
\Ho_{c}(X^{\Sp},\varphi_{fj'g})^{\vee}&\Ho_c(X'^{\Sp},\varphi_{fj'})^{\vee}\ar[l]^{g^*}
\\
\Ho_c(Y^{\Sp},\varphi_{fh})^{\vee}\ar[u]^{j^*}&\Ho_c(Y'^{\Sp},\varphi_f)^{\vee}\ar[u]^{j'^*}\ar[l]^{h^*}.
}
\]
commutes, which follows from the commutativity of (\ref{cornersquare}).  We next deal with the case in which $g$ and $h$ are proper.
Consider the following diagram
\[
\xymatrix{
j'_!j'^*\mathbb{Q}_{Y'}\state{d} \save[]+<1.6cm,-2cm>*\txt{\Large \textbf{A}}\restore \ar[r]\ar[dddd]&j'_!j'^!\mathbb{Q}_{Y'} \save[]+<3.7cm,-1.5cm>*\txt<20pt>{\Large \textbf{C}}
\restore \ar[rr]\ar[d]\ar[ddr]&&\mathbb{Q}_{Y'}\ar[dddd]
\\
&j'_!g_!g^*j'^!\mathbb{Q}_{Y'}\ar[dd]^{j'_!g_!\upsilon}\ar[rd]_-{h_!\varepsilon j'^!}\save[]+<.5cm,.5cm>*\txt{\Large \textbf{E}}\restore
\\
&&h_!h^*j'_!j'^!\mathbb{Q}_{Y'}\ar[rdd]
\\
&j'_!g_!j^!h^*\mathbb{Q}_{Y'}\ar[d]^{=}
\\
j'_!g_!g^*j'^*\mathbb{Q}_{Y'}\state{d}\ar[r]\ar[uuur]\save[]+<1.6cm,1.1cm>*\txt{\Large \textbf{B}}\restore 
&h_!j_!j^!h^*\mathbb{Q}_{Y'} \save[]+<2.3cm,2cm>*\txt{\Large \textbf{D}}\restore \ar[rr]&&h_!h^*\mathbb{Q}_{Y'}
}
\]
The map $\varepsilon$ is the base change isomorphism.  Applying $(Y'^{\Sp}\rightarrow Y')^*\phi_f$ to the square formed by the corners of this diagram, and taking dual compactly supported cohomology, we obtain diagram (\ref{descom}).   So it will be enough to prove that the labelled sub-diagrams commute.
\begin{itemize}
\item \textbf{A} commutes since it is obtained by applying the adjunction $\id\rightarrow g_!g^*$ to the morphism $j'^*\mathbb{Q}_{Y'}\state{d}\rightarrow j'^!\mathbb{Q}_{Y'}$ and applying $j'_!$ to the resulting commutative square.
\item Commutativity of \textbf{B} and \textbf{D} is given by applying $h_!$ to the commutativity conditions on $\upsilon$.
\item \textbf{C} commutes since it is obtained by applying the adjunction $\id\rightarrow h_!h^*$ to the morphism $j'_!j'^!\mathbb{Q}_{Y'}\rightarrow\mathbb{Q}_{Y'}$.
\item The commutativity of \textbf{E} is an exercise in category theory.  Following \cite[Prop.2.5.11]{KS90} we may rewrite the base change map $\overline{\varepsilon}^{-1}\colon  h^*j'_!\rightarrow j_!g^*$ as the composition
\[
h^*j'_!\rightarrow h^*j'_!g_*g^*=h^*h_*j_!g^*\rightarrow j_!g^*.
\]
We claim that in the following diagram of natural transformations, the top half commutes, and $\xi\xi'=\id$
\[
\xymatrix{
&j'_!\ar[ddl]\ar[dr]
\\
&&h_*h^*j'_!\ar[d]
\\
j'_!g_*g^*\ar[ddr]^=\ar[rr]^{\xi'}&&h_*h^*j'_!g_!g^*\ar[d]^=
\\
&&h_*h^*h_*j_!g^*\ar[dl]_{\xi}
\\
&h_*j_!g^*.
}
\]
The top half commutes since it is obtained by applying the adjunction $\id\rightarrow h_*h^*$ to the adjunction $j'_!\rightarrow j'_!g_*g^*$.  The morphism $\xi\xi'$ is obtained by postcomposing the composition $h_*\rightarrow h_*h^*h_*\rightarrow h_*=\id$ with $j_!g^*$.  In particular, the diagram of natural transformations 
\[
\xymatrix{
j'_!\ar[d]\ar[dr]\\
j'_!g_*g^*&h_*h^*j'_!\ar[l]
}
\]
commutes, and commutativity of \textbf{E} follows.

\end{itemize}

\end{proof}
\begin{corollary}
\label{ppCor}
Assume that we have a Cartesian diagram as in (\ref{cornersquare}), and that either
\begin{enumerate}
\item
$j$ and $j'$ are \'etale locally trivial fibrations with smooth fibres or
\item
$h$ and $j'$ are inclusions of transversally intersecting submanifolds of $Y'$.
\end{enumerate}
Then the diagram (\ref{descom}) commutes if $h$ is proper, and in the case that $h$ is an affine fibration, the diagram (\ref{descom2}) commutes.
\end{corollary}
\begin{proof}
It is easy to check that we have the required isomorphism $\upsilon$ and the result follows.
\end{proof}
%In extending to the equivariant case we have to take care of $G$ actions on the relative Verdier dual $p^!\mathbb{Q}_Y$.  In (\ref{noneqcase}) we have used the trivialization of the dualizing sheaves to identify 
%\[
%\tau^!\mathbb{Q}\otimes p_*p^!\tau^!\mathbb{Q}^{-1}\cong\mathbb{Q}\state{{\dim(Y)-\dim(X)}},
%\]
%where $\tau$ is the structure morphism for $Y$.  However the trivializations are not in general liftable to the equivariant category, and we have in the $G$-equivariant case
%\[
%\tau^!\mathbb{Q}\otimes p_*p^!\tau^!\mathbb{Q}^{-1}\cong\orien(p)\state{{\dim(Y)-\dim(X)}}
%\]
%where $\orien(p)$ is the sheaf of relative orientations of $p$.
\begin{proposition}
\label{maneq}
Let $g\colon X\rightarrow Y$ be a $G$-equivariant map of manifolds. Then
\begin{itemize}
\item
If $g$ is a closed embedding, the map $g^*g_*\colon \Ho_{c,G}(X,\varphi_{fg})^{\vee}\rightarrow\Ho_{c,G}(X,\varphi_{fg})^{\vee}$ is given by multiplication by the Euler class of the normal bundle $\No_{X/Y}$
\item
If $g$ is an affine fibration, and the Euler class $\eu(g)$ is a non zero divisor in $\Ho_{G}(X)$, then the induced map
\[
g^*g_*\colon \Ho_{c,G}(X,\varphi_f)^{\vee}\rightarrow \Ho_{c,G}(X,\varphi_f)^{\vee}[\eu(g)^{-1}]
\]
is given by division by $\eu(g)$.
%\item
%If $g$ is a proper map, $\eu(g)\cdot :\Ho_{c,G}(X,\phi_f)^*\rightarrow \Ho_{c,G}(X,\phi_f)^*$ is given by multiplication by $\eu(g)\in\Ho(X,\mathbb{Q})$.

\end{itemize}
\end{proposition}

\section{Cohomological Hall algebra}
\label{coha_section}
\subsection{Spaces of quiver representations} \label{Hdef}
Let $Q$ be a quiver, as in the introduction.  We denote by $Q_0$ the vertices of $Q$, and by $Q_1$ the arrows, and denote by $s,t\colon Q_1\rightarrow Q_0$ the two maps taking an arrow to its source and target.  Let $\mathcal{C}$ be a symmetric tensor category with monoidal product $\boxtimes$.  We assume also that $\D{\mathcal{C}}$ is equipped with an invertible shift functor $\state{1/2}$ and natural isomorphisms of bifunctors $\state{1/2}\circ \boxtimes\cong \boxtimes\circ(\id\otimes \state{1/2})\cong\boxtimes\circ(\state{1/2}\otimes \id)$.  In all our examples this is achieved in a canonical way by first defining $\mathbf{1}_{\mathcal{C}}\state{1/2}$ and then defining $M\state{1/2}:=M\boxtimes \mathbf{1}_{\mathcal{C}}\state{1/2}$.  We define $\mathcal{C}_{\mathbb{Z}^{Q_0}}$ to be the category with objects formal direct sums of objects in $\mathcal{C}$ indexed by $\mathbb{Z}^{Q_0}$, we write such objects as $\bigoplus_{\gamma\in\mathbb{Z}^{Q_0}}\mathcal{L}_{\gamma}$.  If $\mathcal{L}=\bigoplus_{\gamma\in\mathbb{Z}^{Q_0}}\mathcal{L}_{\gamma}$ and $\mathcal{L}'=\bigoplus_{\gamma\in\mathbb{Z}^{Q_0}}\mathcal{L}'_{\gamma}$ are two objects of $\mathcal{C}_{\mathbb{Z}^{Q_0}}$, we define 
\[
\Hom_{\mathcal{C}_{\mathbb{Z}^{Q_0}}}(\mathcal{L},\mathcal{L}'):=\prod_{\gamma\in\mathbb{Z}^{Q_0}}\Hom_{\mathcal{C}}(\mathcal{L}_{\gamma},\mathcal{L}'_{\gamma}).  
\]
We define the full subcategory $\mathcal{C}_Q\subset \mathcal{C}_{\mathbb{Z}^{Q_0}}$ containing those objects $\mathcal{L}$ satisfying $\mathcal{L}_{\gamma}=0$ unless $\gamma\in\mathbb{N}^{Q_0}=(\mathbb{Z}_{\geq 0})^{Q_0}$.  We make $\mathcal{C}_Q$ a tensor category by setting
\[
\bigoplus_{\gamma_1\in\mathbb{N}^{N_0}}\mathcal{L}_{\gamma_1}\boxtimes_+^{\tw}\bigoplus_{\gamma_2\in\mathbb{N}^{Q_0}}\mathcal{L}'_{\gamma_2}:=\bigoplus_{\gamma\in\mathbb{N}^{Q_0}}\left(\bigoplus_{\gamma_1+\gamma_2=\gamma} \mathcal{L}_{\gamma_1}\otimes\mathcal{L}'_{\gamma_2}\right)\state{\chi(\gamma_1,\gamma_2)/2-\chi(\gamma_2,\gamma_1)/2}.
\]
This monoidal structure is not isomorphic to the symmetric monoidal structure defined by
\begin{equation}
\label{symmMon}
\bigoplus_{\gamma_1\in\mathbb{N}^{Q_0}}\mathcal{L}_{\gamma_1}\boxtimes_+\bigoplus_{\gamma_2\in\mathbb{N}^{Q_0}}\mathcal{L}'_{\gamma_2}:=\bigoplus_{\gamma\in\mathbb{N}^{Q_0}}\left(\bigoplus_{\gamma_1+\gamma_2=\gamma} \mathcal{L}_{\gamma_1}\otimes\mathcal{L}'_{\gamma_2}\right),
\end{equation}
indeed the monoidal structure $\boxtimes_{+}^{\tw}$ can not in general even be made into a braided monoidal structure, while the monoidal structure (\ref{symmMon}) extends to a symmetric monoidal structure.
\smallbreak
Let $W\in \Cp Q/[\Cp Q,\Cp Q]$ be a potential for $Q$.  We call a pair $(Q,W)$ a QP from now on.  If a quiver $Q$ is fixed, we will abbreviate $\RS_{Q,\gamma}$ as defined in Equation (\ref{RSdef}) to $\RS_{\gamma}$.  On each $\RS_{\gamma}$ there is a function $\tr(W)_{\gamma}$, which is invariant with respect to the action of $\Gl_{\gamma}$.  

If $\gamma_1,\gamma_2\in\mathbb{N}^{Q_0}$ is a pair of dimension vectors, we denote by $\RS_{\gamma_1,\gamma_2}$ the affine space
\[
\bigoplus_{a\in Q_1}\big\{f_a\in\Hom\big(\Cp^{\gamma_1(s(a))}\oplus\Cp^{\gamma_2(s(a))},\Cp^{\gamma_1(t(a))}\oplus\Cp^{\gamma_2(t(a))}\big)|f_a(\Cp^{\gamma_1(s(a))})\subset\Cp^{\gamma_1(t(a))}\big\}.
\]
If $\gamma_1+\gamma_2=\gamma$, there is a natural inclusion $\eta\colon \RS_{\gamma_1,\gamma_2}\hookrightarrow \RS_{\gamma}$.
\smallbreak
Define $\Gl_{\gamma_1,\gamma_2}:=\prod_{i\in Q_0}\GL_{\Cp}(\gamma_1(i),\gamma_2(i))$, where for $m,n\in\mathbb{N}$, $\GL_{\Cp}(m,n)$ is the subgroup of $\GL_{\Cp}(m+n)$ preserving the flag $0\subset \Cp^m\subset\Cp^{m+n}$.
\smallbreak
We define in the same way the function $\tr(W)_{\gamma_1,\gamma_2}$ on $\RS_{\gamma_1,\gamma_2}$, which is again invariant with respect to the $\Gl_{\gamma_1,\gamma_2}$-action.  There is a natural projection $p\colon \RS_{\gamma_1,\gamma_2}\rightarrow \RS_{\gamma_1}\times \RS_{\gamma_2}$ and an inclusion $\eta\colon \RS_{\gamma_1,\gamma_2}\rightarrow \RS_{\gamma_1+\gamma_2}$ and equalities
\begin{align*}
\tr(W)_{\gamma_1,\gamma_2}=&p^*(\tr(W)_{\gamma_1}\boxplus\tr(W)_{\gamma_2})\\=&\eta^*\tr(W)_{\gamma_1+\gamma_2}.
\end{align*}
In what follows we will use the symbol $\Sp$ to denote a property of $Q$-representations, stable under isomorphisms of complex $Q$-representations.  Since complex points of $\RS_{\gamma}$ represent $Q$-representations, we can define subsets $\RS^{\Sp}_{\gamma}(\Cp)\subset \RS_{\gamma}(\Cp)$ as those subsets of representations satisfying property $\Sp$.  We will always pick $\Sp$ so that this inclusion is the inclusion of complex points induced by an inclusion of algebraic varieties.  Furthermore, since we assume that $\Sp$ is stable under isomorphisms, it will follow that $\RS_{\gamma}^{\Sp}\subset \RS_{\gamma}$ is an inclusion of $\Gl_{\gamma}$-equivariant varieties.
\begin{assumption}
\label{closed_under}
The subspaces $\RS^{\Sp}_{\gamma}$ are required to satisfy the property that there is an equality for all pairs $\gamma_1,\gamma_2\in\mathbb{N}^{Q_0}$
\begin{align}
&\RS_{\gamma_1,\gamma_2}^{\Sp,\ext}:=p^{-1}(\RS_{\gamma_1}^{\Sp}\times \RS_{\gamma_2}^{\Sp})= \eta^{-1} (\RS_{\gamma_1+\gamma_2}^{\Sp})=:\RS^{\Sp}_{\gamma_1,\gamma_2}.
\end{align}
\end{assumption}

We will also assume that $\RS_0= \RS_{0}^{\Sp}$, which is equivalent to not all of the $\RS^{\Sp}_{\gamma}$ being empty.  Assumption \ref{closed_under} may be re-expressed as the condition that the full subcategory of representations having the property $\Sp$ is a Serre subcategory.
\smallbreak

\begin{remark}
For convenience we assume that $\RS_{\gamma}^{\Sp}\cap \crit(\tr(W)_{\gamma})\subset \tr(W)_{\gamma}^{-1}(0)$ for all $\gamma\in\mathbb{N}^{Q_0}$.  This last requirement can be relaxed at the expense of some extra minor complications, but none of the applications we are interested in will require this.
\end{remark}
\begin{example}
Let $\zeta\in\mathbb{H}_+^{Q_0}:=\{re^{i\theta}|r\in\mathbb{R}_{>0},\theta\in(0,\pi]\}^{Q_0}$.  Such a $\zeta$ provides a Bridgeland stability condition for the category of finite-dimensional $Q$-representations.  The \textit{slope} of a $Q$-representation $\rho$ is defined to be $\mu(\rho):=\arg\left(\dim(\rho)\cdot \zeta\right)$.  A representation $\rho$ of $Q$ is called $\zeta$-semistable if for all nonzero subrepresentations $\rho'\subset \rho$, there is an inequality $\mu(\rho')\leq \mu(\rho)$, and $\rho$ is called $\zeta$-stable if this inequality is strict for all proper $\rho'\subset\rho$.  Fix $\theta\in (0,\pi]$.  One can easily check that the condition on a $Q$-representation $\rho$ of being $\zeta$-semistable and satisfying $\mu(\rho)=\theta$ satisfies Assumption \ref{closed_under}.  The property of being $\zeta$-stable with slope a fixed $\theta$ does not satisfy Assumption \ref{closed_under} (stability is not preserved under taking direct sums, for instance).  We may define the (underlying object of the) cohomological Hall algebra of $\zeta$-semistable $Q$-representations of slope $\theta$ with potential $W$:
\[
\mathcal{H}^{\zeta-ss}_{Q,W,\theta}:=\bigoplus_{\gamma\in\mathbb{N}^{Q_0}|\arg(\zeta\cdot\gamma)=\theta}\Ho_{c,\Gl_{\gamma}}(\RS^{\zeta-ss}_{\gamma},\varphi_{\tr(W)_{\gamma}})\state{{-\chi(\gamma,\gamma)/2}}^{\vee}.
\]
where the twist is as defined in (\ref{statedef}).  Here, since $\RS^{\zeta-ss}_{\gamma}$ is an open subscheme of $\RS_{\gamma}$, it does not matter whether we take vanishing cycles on $\RS^{\zeta-ss}_{\gamma}$ of the restriction of the function $\tr(W)_{\gamma}$, or instead take the restriction to $\RS^{\zeta-ss}_{\gamma}$ of the vanishing cycles complex $\varphi_{\tr(W)_{\gamma}}$.  However, for general $\Sp$ these two objects are \textit{not} the same, and we will always consider the latter.
\end{example}

For general $\Sp$ we define 
\begin{equation}
\label{uldef}
\mathcal{H}^{\Sp}_{Q,W,\gamma}:=\Ho_{c,\Gl_{\gamma}}(\RS^{\Sp}_{\gamma},\varphi_{\tr(W)_{\gamma}})^{\vee}\state{{{\chi(\gamma,\gamma)/2}}},
\end{equation}
and define
\begin{equation}
\label{CoHADef}
\mathcal{H}^{\Sp}_{Q,W}:=\bigoplus_{\gamma\in\mathbb{N}^{Q_0}}\mathcal{H}^{\Sp}_{Q,W,\gamma}.
\end{equation}
In the case in which all $\RS^{\Sp}_{\gamma}=\RS_{\gamma}$ we denote $\mathcal{H}^{\Sp}_{Q,W}$ by $\mathcal{H}_{Q,W}$.

We consider (\ref{CoHADef}) as an object in $\D{\mathcal{C}_Q}$, where $\mathcal{C}$ can be, for example, each of the following choices of symmetric tensor categories.
\begin{itemize}
\item
$\mathcal{C}$ can be taken to be $\Vect$, the category of finite-dimensional vector spaces.  In this case we lose all Hodge theoretic information regarding vanishing cycles, and only consider the underlying vector spaces in (\ref{uldef}).  The shift functor $\{1/2\}$ is then the usual cohomological shift functor $[1]$.
\item
$\mathcal{C}$ can be taken to be $\MHS$.  In this case we remember much of the mixed Hodge structure underlying (\ref{uldef}), but some care has to be taken when we try to turn (\ref{uldef}) into an algebra, as the Thom--Sebastiani isomorphism, which forms part of the definition of the multiplication map, does not necessarily respect mixed Hodge structures.  For a class of examples in which the Thom--Sebastiani isomorphism does respect the underlying mixed Hodge structures, see the appendix, and especially Proposition \ref{TScomm}.  As a further complication, the shift $\state{{-}}$ has no satisfactory definition in $\D{\MHS}$; this category has to be slightly extended as in \cite[Sec.3.4]{COHA}.
\item
$\mathcal{C}$ can be taken to be $\MMHS$.  In this case the shift functor $\state{{-1/2}}$ is given by tensoring with $\Ho_c(\mathbb{A}^1,\varphi_{x^2})$.
\item
Define $\mathbb{Z}_{\wt}:=\mathbb{Z}$.  Then $\mathcal{C}$ can be taken to be $\Vect_{\mathbb{Z}_{\wt}}$, the category of finite-dimensional graded vector spaces, but in this case we should replace $\mathcal{H}_{Q,W}^{\Sp}$ by its associated graded object with respect to the weight filtration on the underlying monodromic mixed Hodge structure.  The shift functor $\state{1/2}$ is given by tensoring with $\mathbb{Q}$, given bidegree $(-1,-1)$ with respect to the cohomological and weight grading.
\end{itemize}

\begin{remark}
\label{JacRem}
As an aside we explain the representation-theoretic origin of $\Ho_{c,\Gl_{\gamma}}(\RS^{\Sp}_{\gamma},\varphi_{\tr(W)_{\gamma}})$.  Given a quiver $Q$, let $S$ be the set of cyclic paths in $Q$.  For a potential $W=\sum_{m\in S} \lambda_m m\in \Cp Q/[\Cp Q,\Cp Q]$, where each of the $\lambda_m$ is a scalar, and all but finitely many of the $\lambda_m$ are zero, the noncommutative derivative $\partial W/\partial a$ is defined by setting
\[
\partial W/\partial a:=\sum_{m\in S}\lambda_m\sum_{m=uav|u,v\text{ paths in }Q}vu.
\]
The Jacobi algebra for the QP $(Q,W)$ is defined by
\[
\Jac(Q,W):=\Cp Q/\langle \partial W/\partial a|a\in Q_1\rangle.
\]
Representations of the Jacobi algebra form a Zariski closed subscheme $\mathcal{V}_{\gamma}$ of $\RS_{\gamma}$ in the natural way, and we have the equality of subschemes
\[
\mathcal{V}_{\gamma}=\crit(\tr(W)_{\gamma}).
\]
In addition, the vanishing cycles complex $\varphi_{\tr(W)_{\gamma}}$ is supported on the critical locus of $\tr(W)_{\gamma}$.  So the compactly supported equivariant cohomology $\Ho_{c,\Gl_{\gamma}}(\RS^{\Sp}_{\gamma},\varphi_{\tr(W)_{\gamma}})$ can be thought of as the compactly supported equivariant cohomology of the stack of $\gamma$-dimensional representations of $\Jac(Q,W)$ satisfying property $\Sp$, with coefficients in the vanishing cycles complex for $\tr(W)_{\gamma}$.
\end{remark}

\subsection{Multiplication for $\mathcal{H}^{\Sp}_{Q,W}$}  \label{MOsec}In this subsection we recall the associative product of \cite[Sec.7.6]{COHA}
\[
m\colon \mathcal{H}^{\Sp}_{Q,W,\gamma}\boxtimes^{\tw}_+\mathcal{H}^{\Sp}_{Q,W,\gamma}\rightarrow \mathcal{H}^{\Sp}_{Q,W,\gamma}.
\]
As above, all functors will be considered as derived functors.

Consider the affine fibrations 
\begin{equation}
\label{affpf}
p_{\gamma_1,\gamma_2,N}\colon \overline{(\RS_{\gamma_1,\gamma_2},\Gl_{\gamma_1}\times \Gl_{\gamma_2})}_N\rightarrow \overline{(\RS_{\gamma_1}\times \RS_{\gamma_2},\Gl_{\gamma_1}\times \Gl_{\gamma_2})}_N.
\end{equation}
These induce isomorphisms
\[
\phi_{(\tr(W)_{\gamma_1}\boxplus\tr(W)_{\gamma_2})_N}\left(\mathbb{Q}_{\overline{(\RS_{\gamma_1}\times \RS_{\gamma_2},\Gl_{\gamma_1}\times \Gl_{\gamma_2})}_N}\rightarrow p_{\gamma_1,\gamma_2,N,*}\mathbb{Q}_{\overline{(\RS_{\gamma_1,\gamma_2},\Gl_{\gamma_1}\times \Gl_{\gamma_2})}_N}\right),
\]
and via Verdier duality, isomorphisms
\begin{equation}
\phi_{(\tr(W)_{\gamma_1}\boxplus\tr(W)_{\gamma_2})_N}\left(p_{\gamma_1,\gamma_2,N,!}D\mathbb{Q}_{\overline{(\RS_{\gamma_1,\gamma_2},\Gl_{\gamma_1}\times \Gl_{\gamma_2})}_N}\rightarrow D\mathbb{Q}_{\overline{(\RS_{\gamma_1}\times \RS_{\gamma_2},\Gl_{\gamma_1}\times \Gl_{\gamma_2})}_N}\right),
\end{equation}
or
\begin{equation}
\phi_{(\tr(W)_{\gamma_1}\boxplus\tr(W)_{\gamma_2})_N}\left(p_{\gamma_1,\gamma_2,N,!}\mathbb{Q}_{\overline{(\RS_{\gamma_1,\gamma_2},\Gl_{\gamma_1}\times \Gl_{\gamma_2})}_N}\rightarrow \mathbb{Q}_{\overline{(\RS_{\gamma_1}\times \RS_{\gamma_2},\Gl_{\gamma_1}\times \Gl_{\gamma_2})}_N}\state{{-l_1(\gamma_2,\gamma_1)}}\right),
\end{equation}
using (\ref{CanDu}).  The function $l_1$ is as defined in equations (\ref{ldefs}).
By Corollary \ref{ppoutside} we obtain isomorphisms
\begin{align}
&p_{\gamma_1,\gamma_2,N,!}\phi_{\tr(W)_{\gamma_1,\gamma_2,N}}(\mathbb{Q}_{\overline{(\RS_{\gamma_1,\gamma_2},\Gl_{\gamma_1}\times \Gl_{\gamma_2})}_N})\rightarrow \\&\nonumber\phi_{(\tr(W)_{\gamma_1}\boxplus\tr(W)_{\gamma_2})_N}(\mathbb{Q}_{\overline{(\RS_{\gamma_1}\times \RS_{\gamma_2},\Gl_{\gamma_1}\times \Gl_{\gamma_2})}_N}\state{{-l_1(\gamma_2,\gamma_1)}}).
\end{align}
Applying the shift functor, passing to the limit and taking compactly supported cohomology, we arrive at an isomorphism
\begin{align}
(\alpha)^*\colon &\Ho_{c,\Gl_{\gamma_1}\times \Gl_{\gamma_2}}(\RS_{\gamma_1,\gamma_2}^{\Sp},\varphi_{\tr(W)_{\gamma_1,\gamma_2}})\state{{-\chi(\gamma,\gamma)/2+l_0(\gamma_2,\gamma_1)}}\rightarrow \\ \nonumber &\Ho_{c,\Gl_{\gamma_1}\times \Gl_{\gamma_2}}(\RS_{\gamma_1}^{\Sp}\times \RS_{\gamma_2}^{\Sp},\varphi_{\tr(W)_{\gamma_1}\boxplus\tr(W)_{\gamma_2}})\state{{-\chi(\gamma,\gamma)/2+\chi(\gamma_2,\gamma_1)}},
\end{align}
though we work with the dual map
\begin{align}
\alpha\colon &\Ho_{c,\Gl_{\gamma_1}\times \Gl_{\gamma_2}}(\RS_{\gamma_1}^{\Sp}\times \RS_{\gamma_2}^{\Sp},\varphi_{\tr(W)_{\gamma_1}\boxplus\tr(W)_{\gamma_2}})^{\vee}\state{{\chi(\gamma,\gamma)/2-\chi(\gamma_2,\gamma_1)}}\rightarrow \\ \nonumber &\Ho_{c,\Gl_{\gamma_1}\times \Gl_{\gamma_2}}(\RS_{\gamma_1,\gamma_2}^{\Sp},\varphi_{\tr(W)_{\gamma_1,\gamma_2}})^{\vee}\state{{\chi(\gamma,\gamma)/2-l_0(\gamma_2,\gamma_1)}},
\end{align}
as $\mathcal{H}_{Q,W}^{\Sp}$ is defined in terms of dual spaces of compactly supported equivariant cohomology.  In the terminology of Section \ref{umkehr_sec}, this map is the pullback map associated to the $\Gl_{\gamma_1}\times \Gl_{\gamma_2}$-equivariant map $\RS_{\gamma_1,\gamma_2}\rightarrow \RS_{\gamma_1}\times \RS_{\gamma_2}$.
\smallbreak
By the definitions of the functions $\chi,l_0,l_1$ in Equation (\ref{ldefs}) we have the equalities
\begin{align*}
\chi(\gamma,\gamma)/2-l_0(\gamma_2,\gamma_1)+l_1(\gamma_2,\gamma_1)=&\chi(\gamma,\gamma)/2-\chi(\gamma_2,\gamma_1)\\
=&\chi(\gamma_1,\gamma_1)/2+\chi(\gamma_2,\gamma_2)/2+\chi(\gamma_1,\gamma_2)/2-\chi(\gamma_2,\gamma_1)/2.
\end{align*}
This final quantity is the same as the twist in the definition of $\mathcal{H}^{\Sp}_{Q,W,\gamma_1}\boxtimes_{+}^{\tw}\mathcal{H}^{\Sp}_{Q,W,\gamma_2}$.
\smallbreak

Similarly, the affine fibrations 
\[
q_{\gamma_1,\gamma_2,N}\colon \overline{(\RS_{\gamma_1,\gamma_2},\Gl_{\gamma_1}\times \Gl_{\gamma_2})}_N\rightarrow \overline{(\RS_{\gamma_1,\gamma_2}, \Gl_{\gamma_1,\gamma_2})}_N
\]
induce maps
\begin{equation}
\label{af2}
q_{\gamma_1,\gamma_2,N,!}\phi_{\tr(W)_{\gamma_1,\gamma_2,N}}(\mathbb{Q}_{\overline{(\RS_{\gamma_1,\gamma_2},\Gl_{\gamma_1}\times \Gl_{\gamma_2})}_N})\rightarrow \phi_{\tr(W)_{\gamma_1,\gamma_2,N}}(\mathbb{Q}_{\overline{(\RS_{\gamma_1,\gamma_2},\Gl_{\gamma_1,\gamma_2})}_N})\state{{-l_0(\gamma_1,\gamma_2)}}.
\end{equation}
%where 
%\begin{equation}
%\label{lddef}
%l'=\sum_{a\in Q_1}\gamma_1(s(a))\gamma_2(t(a)).  
%\end{equation}
Taking shifted compactly supported cohomology and taking duals we obtain pullback isomorphisms
\begin{align}
\label{betaMap}
\beta\colon &\Ho_{c, \Gl_{\gamma_1,\gamma_2}}(\RS_{\gamma_1,\gamma_2}^{\Sp},\varphi_{\tr(W)_{\gamma_1,\gamma_2}})^{\vee}\state{{\chi(\gamma,\gamma)/2}}\rightarrow \\&\Ho_{c, \Gl_{\gamma_1}\times \Gl_{\gamma_2}}(\RS_{\gamma_1,\gamma_2}^{\Sp},\varphi_{\tr(W)_{\gamma_1,\gamma_2}})^{\vee}\state{{\chi(\gamma,\gamma)/2-l_0(\gamma_2,\gamma_1)}}.\nonumber
\end{align}
Consider the proper maps
\begin{equation}
\label{prdef}
pr_{\gamma_1,\gamma_2,N}\colon \overline{(\RS_{\gamma},\Gl_{\gamma_1,\gamma_2}})_N\rightarrow \overline{(\RS_{\gamma},\Gl_{\gamma}})_N.  
\end{equation}
The natural transformation of functors
\[
\phi_{\tr(W)_{\gamma,N}}\rightarrow pr_{\gamma_1,\gamma_2,N,!}\phi_{\tr(W)_{\gamma,N}}pr_{\gamma_1,\gamma_2,N}^*
\]
applied to $\mathbb{Q}_{\overline{(\RS_{\gamma},\Gl_{\gamma})}_N}$ and restricted to $\overline{(\RS^{\Sp}_{\gamma},\Gl_{\gamma}})_N$ induces the map
\begin{equation}
(\delta)^*\colon \Ho_{c,\Gl_{\gamma}}(\RS^{\Sp}_{\gamma},\varphi_{\tr(W)_{\gamma}})\rightarrow \Ho_{c,\Gl_{\gamma_1,\gamma_2}}(\RS^{\Sp}_{\gamma},\varphi_{\tr(W)_{\gamma}})
\end{equation}
or, taking shifted duals
\begin{align}
\label{pfdef}
\delta\colon &\Ho_{c,\Gl_{\gamma_1,\gamma_2}}(\RS^{\Sp}_{\gamma},\varphi_{\tr(W)_{\gamma}})^{\vee}\state{{\chi(\gamma,\gamma)/2-\chi(\gamma_2,\gamma_1)}}\rightarrow\\& \Ho_{c,\Gl_{\gamma}}(\RS^{\Sp}_{\gamma},\varphi_{\tr(W)_{\gamma}})^{\vee}\state{{\chi(\gamma,\gamma)/2-\chi(\gamma_2,\gamma_1)}},\nonumber
\end{align}
the pushforward map associated to the maps (\ref{prdef}).  Here, in passing to the limit, we have applied Proposition \ref{mixingprop}.

Next, consider the inclusions $i_{\gamma_1,\gamma_2,N}\colon \overline{(\RS_{\gamma_1,\gamma_2},\Gl_{\gamma_1,\gamma_2})}_N\rightarrow \overline{(\RS_{\gamma},\Gl_{\gamma_1,\gamma_2})}_N$.  These induce maps
\begin{equation}
\label{zetapd}
\varphi_{\tr(W)_{\gamma,N}}\rightarrow i_{\gamma_1,\gamma_2,N,*}i_{\gamma_1,\gamma_2,N}^*\varphi_{\tr(W)_{\gamma,N}}
\end{equation}
and maps
\[
\zeta^{\vee}\colon \Ho_{c,\Gl_{\gamma_1,\gamma_2}}(\RS^{\Sp}_{\gamma},\varphi_{\tr(W)_{\gamma}})\rightarrow\Ho_{c,\Gl_{\gamma_1,\gamma_2}}(\RS^{\Sp}_{\gamma_1,\gamma_2},\varphi_{\tr(W)_{\gamma}})
\]
or, again taking shifted duals,
\begin{align}
\label{zetdef}
\zeta\colon &\Ho_{c,\Gl_{\gamma_1,\gamma_2}}(\RS^{\Sp}_{\gamma_1,\gamma_2},\varphi_{\tr(W)_{\gamma}})^{\vee}\state{{\chi(\gamma,\gamma)/2-\chi(\gamma_2,\gamma_1)}}\rightarrow \\&\Ho_{c,\Gl_{\gamma_1,\gamma_2}}(\RS^{\Sp}_{\gamma},\varphi_{\tr(W)_{\gamma}})^{\vee}\state{{\chi(\gamma,\gamma)/2-\chi(\gamma_2,\gamma_1)}}.\nonumber
\end{align}

Finally, the natural transformation
\[
\phi_{\tr(W)_{\gamma,N}}\rightarrow i_{\gamma_1,\gamma_2,N,*}\phi_{\tr(W)_{\gamma_1,\gamma_2,N}}i_{\gamma_1,\gamma_2,N}^*
\]
induces maps
\begin{align}
\label{epsdef}
\epsilon\colon &\Ho_{c,\Gl_{\gamma_1,\gamma_2}}(\RS^{\Sp}_{\gamma_1,\gamma_2},\varphi_{\tr(W)_{\gamma_1,\gamma_2}})^{\vee}\state{{\chi(\gamma,\gamma)/2-\chi(\gamma_2,\gamma_1)}}\rightarrow \\&\Ho_{c,\Gl_{\gamma_1,\gamma_2}}(\RS^{\Sp}_{\gamma_1,\gamma_2},\varphi_{\tr(W)_{\gamma}})^{\vee}\state{{\chi(\gamma,\gamma)/2-\chi(\gamma_2,\gamma_1)}}\nonumber
\end{align}
as in the definition of $\delta$.

Throughout we have used that $pr_{\gamma_1,\gamma_2,N}$ and $i_{\gamma_1,\gamma_2,N}$ are proper in order to identify $pr_{\gamma_1,\gamma_2,N,*}$ with $pr_{\gamma_1,\gamma_2,N,!}$ and $i_{\gamma_1,\gamma_2,N,*}$ with $i_{\gamma_1,\gamma_2,N,!}$.  
The passage to the limit in the definition of $\zeta\epsilon$ is again an application of Proposition \ref{mixingprop}.

We define a product 
\[
m\colon \mathcal{H}^{\Sp}_{Q,W}\boxtimes_+^{\tw} \mathcal{H}^{\Sp}_{Q,W}\rightarrow \mathcal{H}^{\Sp}_{Q,W}
\]
to be the map which, when restricted to $\mathcal{H}_{Q,W,\gamma_1}^{\Sp}\boxtimes_+^{\tw}\mathcal{H}_{Q,W,\gamma_2}^{\Sp}$, is given by the composition of maps $\delta\zeta\epsilon\beta^{-1}\alpha\TS$, where $\TS$ is the shift of the Thom--Sebastiani isomorphism
\begin{align*}
\Ho_{c,\Gl_{\gamma_1}}(\RS^{\Sp}_{\gamma_1},\varphi_{\tr(W)_{\gamma_1}})^{\vee}\otimes \Ho_{c,\Gl_{\gamma_2}}(\RS^{\Sp}_{\gamma_2},\varphi_{\tr(W)_{\gamma_2}})^{\vee}\rightarrow&\Ho_{c,\Gl_{\gamma_1}\times \Gl_{\gamma_2}}(\RS^{\Sp}_{\gamma_1}\times \RS^{\Sp}_{\gamma_2},\varphi_{\tr(W)_{\gamma_1}\boxplus\tr(W)_{\gamma_2}})^{\vee}.
\end{align*}

\begin{proposition}[\cite{COHA} --- Section 7.8]
\label{IntId}
The map $\epsilon$ of Equation (\ref{epsdef}) is an isomorphism.
\end{proposition}
Proposition \ref{IntId} is an application of the Hodge theoretic version of the ``integral identity'' from \cite[Sec.4.4]{KS}, and is proved as Theorem 13 of \cite{COHA}.  In fact the result will not be used in this paper, its main purpose is to show that the integration map of \cite[Sec.6.3]{KS} is a ring homomorphism, but we will not be concerned with the integration map.
\begin{remark}
The map $\zeta\epsilon$ given by composing (\ref{zetdef}) and (\ref{epsdef}) is the shifted pushforward 
\begin{align*}
\underline{\zeta}\colon &\Ho_{c,\Gl_{\gamma_1,\gamma_2}}(\RS^{\Sp}_{\gamma_1,\gamma_2},\varphi_{\tr(W)_{\gamma_1,\gamma_2}})^{\vee}\state{{\chi(\gamma,\gamma)/2-\chi(\gamma_2,\gamma_1)}}\rightarrow \\&
\Ho_{c,\Gl_{\gamma_1,\gamma_2}}(\RS^{\Sp}_{\gamma},\varphi_{\tr(W)_{\gamma}})^{\vee}\state{{\chi(\gamma,\gamma)/2-\chi(\gamma_2,\gamma_1)}}
\end{align*}
associated to $i_{\gamma_1,\gamma_2}$.  We have split it into two maps as in \cite[Sec.7.6]{COHA}.  In subsequent sections we will just use $\underline{\zeta}$.
\end{remark}
For the benefit of the reader we represent the multiplication in a different way:
\[
\xymatrix{
\Ho_{c,\Gl_{\gamma_1+\gamma_2}}(\RS^{\Sp}_{\gamma_1+\gamma_2},\varphi_{\tr(W)_{\gamma_1+\gamma_2}})^{\vee}
&
\Ho_{\begin{picture}(7,9)
\put(1,5){\circle*{3}}\put(1,1){\circle*{3}}\put(5,1){\circle*{3}}\put(5,5){\circle*{3}}
\end{picture}}\bigl(\RS_{\begin{picture}(7,9)\put(1,5){\circle*{3}}\put(1,1){\circle*{3}}\put(5,1){\circle*{3}}\put(5,5){\circle*{3}}
\end{picture}},W_{\begin{picture}(7,9)\put(1,5){\circle*{3}}\put(1,1){\circle*{3}}\put(5,1){\circle*{3}}\put(5,5){\circle*{3}}
\end{picture}}\bigr)^{\vee}\ar[l]^-{=:}
\\
\Ho_{\begin{picture}(7,9)
\put(1,5){\circle*{3}}\put(5,1){\circle*{3}}\put(5,5){\circle*{3}}
\end{picture}}\bigl(\RS_{\begin{picture}(7,9)\put(1,5){\circle*{3}}\put(5,1){\circle*{3}}\put(5,5){\circle*{3}}
\end{picture}},W_{\begin{picture}(7,9)\put(1,1){\circle*{3}}\put(1,5){\circle*{3}}\put(5,1){\circle*{3}}\put(5,5){\circle*{3}}
\end{picture}}\bigr)^{\vee}\ar[r]^-{\zeta}
&
\Ho_{\begin{picture}(7,9)
\put(1,5){\circle*{3}}\put(5,1){\circle*{3}}\put(5,5){\circle*{3}}
\end{picture}}\bigl(\RS_{\begin{picture}(7,9)\put(1,1){\circle*{3}}\put(1,5){\circle*{3}}\put(5,1){\circle*{3}}\put(5,5){\circle*{3}}
\end{picture}},W_{\begin{picture}(7,9)\put(1,5){\circle*{3}}\put(1,1){\circle*{3}}\put(5,1){\circle*{3}}\put(5,5){\circle*{3}}
\end{picture}}\bigr)^{\vee}\ar[u]^-{\delta}
\\
\Ho_{\begin{picture}(7,9)
\put(1,5){\circle*{3}}\put(5,1){\circle*{3}}\put(5,5){\circle*{3}}
\end{picture}}\bigl(\RS_{\begin{picture}(7,9)\put(1,5){\circle*{3}}\put(5,1){\circle*{3}}\put(5,5){\circle*{3}}
\end{picture}},W_{\begin{picture}(7,9)\put(1,5){\circle*{3}}\put(5,1){\circle*{3}}\put(5,5){\circle*{3}}
\end{picture}}\bigr)^{\vee}\ar[u]^{\epsilon}\ar[ur]^{\underline{\zeta}}
&
\Ho_{\begin{picture}(7,9)
\put(1,5){\circle*{3}}\put(5,1){\circle*{3}}
\end{picture}}\bigl(\RS_{\begin{picture}(7,9)\put(1,5){\circle*{3}}\put(5,1){\circle*{3}}\put(5,5){\circle*{3}}
\end{picture}},W_{\begin{picture}(7,9)\put(1,5){\circle*{3}}\put(5,1){\circle*{3}}\put(5,5){\circle*{3}}
\end{picture}}\bigr)^{\vee}\ar[l]_-{\beta^{-1}}
\\
\left(\Ho_{c,\Gl_{\gamma_1}}(\RS^{\Sp}_{\gamma_1},\varphi_{\tr(W)_{\gamma_1}})\otimes\Ho_{c,\Gl_{\gamma_1}}(\RS^{\Sp}_{\gamma_2},\varphi_{\tr(W)_{\gamma_2}})\right)^{\vee}\ar[r]_-{\TS}
&
\Ho_{\begin{picture}(7,9)
\put(1,5){\circle*{3}}\put(5,1){\circle*{3}}
\end{picture}}\bigl(\RS_{\begin{picture}(7,9)\put(1,5){\circle*{3}}\put(5,1){\circle*{3}}
\end{picture}},W_{\begin{picture}(7,9)\put(1,5){\circle*{3}}\put(5,1){\circle*{3}}
\end{picture}}\bigr)^{\vee}\ar[u]_-{\alpha}
}
\]
where we have abbreviated the notation in the obvious ways, and left out the shifts.
\smallbreak
%For each $\gamma\in\mathbb{Z}^{Q_0}$ we have an inclusion $\Cp^*\rightarrow \Gl_{\gamma}$ given by $z\rightarrow (z\cdot \id_{\gamma(i)\times\gamma(i)})_{i\in Q_0}$.  It follows that each $\RS_{\gamma}^{\Sp}$, $\RS_{\gamma_1}^{\Sp}\times \RS_{\gamma_2}^{\Sp}$ and  $\RS_{\gamma_1,\gamma_2}^{\Sp}$ carries a $\Cp^*$-action, which can be seen to be trivial, and so each of the constituent terms of the above diagram are free $\Cp[y]$-modules.  Since the underlying maps of spaces are $\Cp^*$-equivariant, we deduce
%\begin{proposition}
%\label{misY}
%The multiplication $m$ is a $\Cp[y]$-module homomorphism between free $\Cp[y]$-modules.
%\end{proposition}
%Note that for $\tau\otimes \tau'\in\Ho_{c,\Gl_{\gamma_1}}(\RS^{\Sp}_{\gamma_1},\phi_{\tr(W)_{\gamma_1}})^*\otimes \Ho_{c,\Gl_{\gamma_2}}(\RS^{\Sp}_{\gamma_2},\phi_{\tr(W)_{\gamma_2}})^*$ we have 
%\[
%y(\tau\otimes \tau')=((y\tau)\otimes \tau')+(\tau\otimes (y\tau')).
%\]

We work with the nonsymmetric monoidal structure $\boxtimes_+^{\tw}$ in order for $m$ to preserve cohomological degree, and for $m$ to be an unshifted map in $\D{\mathcal{C}_Q}$, e.g. to avoid the appearance of any Tate twists in the definition of $m$, assuming that our background category $\mathcal{C}$ is set to be $\MMHS$, the category of mondromic mixed Hodge structures.
\begin{definition}
\label{degpres}
We say that $\Sp$ is \textit{degree preserving} if it satisfies the following condition: for all $\gamma_1$ and $\gamma_2$ such that $\RS^{\Sp}_{\gamma_1}$ and $\RS^{\Sp}_{\gamma_2}$ are non-empty, $\chi(\gamma_1,\gamma_2)=\chi(\gamma_2,\gamma_1)$.
\end{definition}
\begin{proposition}
Let $\Sp$ be degree-preserving.  Then $(\mathcal{H}^{\Sp}_{Q,W},m)$ is an algebra object in the category $\D{\mathcal{C}_Q}$ with the untwisted symmetric monoidal product $\boxtimes_+$.
\end{proposition}
\begin{example}
\label{symmEx}
Assume that $Q$ is symmetric.  Then $\chi(\bullet,\bullet)$ is symmetric, and so every property $\Sp$ is degree-preserving, and so $\mathcal{H}^{\Sp}_{Q,W}$ is an algebra in the category $\D{\mathcal{C}_Q}$ with the symmetric monoidal product $\boxtimes_+$ for every property $\Sp$.
\end{example}
\begin{example}
Assume that the stability condition $\zeta$ is generic, in the strong sense that $\arg(\zeta\cdot\gamma_1)=\arg(\zeta\cdot\gamma_2)$ implies that $\gamma_1=r\gamma_2$ for some $r\in\mathbb{R}$.  Then setting $\Sp$ to be the property that a representation $\rho$ of $Q$ is $\zeta$-semistable, and $\mu(\rho)=\theta$ for some fixed $\theta$, the property $\Sp$ is degree preserving.  In fact the requirement that $\Sp$, so defined, is degree-preserving, is often the requirement that one is most interested in guaranteeing by imposing genericity, so elsewhere (see for example \cite{MeRe14}), genericity for $\zeta$ is just defined by the weak requirement that $\mu(\rho)=\theta$ is a degree-preserving property for every $\theta$.  By Example \ref{symmEx}, all stability conditions on a symmetric quiver are generic in this sense.
\end{example}

\section{The critical CoHA as a shuffle algebra}
\label{t_sec}

\subsection{Localisation}  Let $\mathcal{H}^{\Sp}_{Q,W,\gamma_1}\boxtimes_+^{\tw}\ldots\boxtimes_+^{\tw}\mathcal{H}^{\Sp}_{Q,W,\gamma_r}$ be an arbitrary finite tensor product of $\mathbb{N}^{Q_0}$-graded pieces of the CoHA $\mathcal{H}_{Q,W}^{\Sp}$.  Via the Thom--Sebastiani isomorphism we may identify the vector spaces
\begin{equation}
\label{bigTens}
\mathcal{H}^{\Sp}_{Q,W,\gamma_1}\boxtimes_+^{\tw}\ldots\boxtimes_+^{\tw}\mathcal{H}^{\Sp}_{Q,W,\gamma_r}\cong \Ho_{c,\Gl_{\gamma_1}\times\ldots\times \Gl_{\gamma_r}}(\RS^{\Sp}_{\gamma_1}\times\ldots\times \RS^{\Sp}_{\gamma_r},\varphi_{\tr(W)_{\gamma_1}\boxplus\ldots\boxplus\tr(W)_{\gamma_r}})^{\vee},
\end{equation}
where we have left out the shift in cohomology.  As detailed in Section \ref{moduleStruc}, the left hand side of (\ref{bigTens}) carries an action of
\[
\SR_{\gamma_1,\ldots,\gamma_r}:=\bigotimes_{c=1,\ldots,r}\left(\bigotimes_{i\in Q_0}\mathbb{Q}\left[x^{(c)}_{i,1},\ldots,x^{(c)}_{i,\gamma_c(i)}\right]\right)^{\Sym_{\gamma_c}}.
\]
Let $M:=\RS^{\Sp}_{\gamma_1}\times\ldots\times \RS^{\Sp}_{\gamma_r}$ and $G:=\Gl_{\gamma_1}\times\ldots\times \Gl_{\gamma_r}$.  Let $c,d\in\{1,\ldots,r\}$ be a pair of distinct numbers, and let $i,i'\in Q_0$.  Let
\[
V=M\times\Hom(\Cp^{\gamma_c(i)},\Cp^{\gamma_d(i')}).
\]
Then $V$ is naturally a $G$-equivariant vector bundle over $M$, with projection $\pi\colon V\rightarrow M$, section $M\xrightarrow{z\mapsto (z,0)} M\times \Hom(\Cp^{\gamma_c(i)},\Cp^{\gamma_d(i')})$ and an Euler class defined as in Section \ref{umkehr_sec}, which we denote $\eu(i,i',c,d)$.
\begin{proposition}
\label{ABprop}
Multiplication by $\eu(i,i',c,d)$ is an injective endomorphism of $\mathcal{H}^{\Sp}_{Q,W,\gamma_1}\boxtimes_+^{\tw}\ldots\boxtimes_+^{\tw}\mathcal{H}^{\Sp}_{Q,W,\gamma_r}$.
\end{proposition}
\begin{proof}
This is a small variation of the Atiyah Bott lemma, and we may adapt the original proof of \cite{AtBo83}.  In detail, let $\Sph^1\rightarrow \Gl_{\gamma_c}$ be defined by $e^{\sqrt{-1}\theta}\mapsto e^{\sqrt{-1}\theta}\id$.  Via the inclusion of algebraic groups $\Gl_{\gamma_c}\rightarrow \Gl_{\gamma_1}\times\ldots\times \Gl_{\gamma_r}$ this defines an action on the total space $T$ such that the fixed point set is exactly $M$.  Let $G'=G/\Sph^1$.  Then $\Ho_{c,G}(M^{\Sp},\varphi_{\tr(W)_{\gamma_1}\boxplus\ldots\boxplus\tr(W)_{\gamma_r}})^{\vee}$ is filtered by 
\[
F^p\left(\Ho_{c,G}(M^{\Sp},\varphi_{\tr(W)_{\gamma_1}\boxplus\ldots\boxplus\tr(W)_{\gamma_r}})^{\vee}\right):=\Ho^{\geq p}_{c,G'}(M^{\Sp},\varphi_{\tr(W)_{\gamma_1}\boxplus\ldots\boxplus\tr(W)_{\gamma_r}})^{\vee}\otimes \Ho_{\Sph^1}(\pt,\mathbb{Q}), 
\]
and we denote by $N$ the associated graded object, which is acted on freely by $\Ho_{\Sph^1}(\pt,\mathbb{Q})$.  Let 
\[
\tilde{\mu}\in \Ho_{c,G}(M^{\Sp},\varphi_{\tr(W)_{\gamma_1}\boxplus\ldots\boxplus\tr(W)_{\gamma_r}})^{\vee}, 
\]
and let $\mu\in N$ be the associated homogeneous element.  Let $s$ equal the degree of $\mu$ with respect to the grading induced by $F$.  Then projecting $\eu(i,i',c,d)\mu$ onto its degree $s$ part, also with respect to the grading induced by $F$, it is given by $\eu_{\Sph^1}(i,i',c,d)\mu$, where now $\eu_{\Sph^1}(i,i',c,d)$ is the $\Sph^1$-equivariant Euler characteristic of $V$, which is nonzero since $M$ is the fixed locus of the $\Sph^1$-action on $V$.
\end{proof}
\begin{definition}
With notation as above, so that in particular $c,d\in\{1,\ldots,r\}$ remain distinct numbers, we define 
\[
\eue(Q_1,\gamma_c,\gamma_d)=\prod_{a\in Q_1}\prod_{m=1}^{\gamma_c(s(a))}\nolimits\prod_{m'=1}^{\gamma_d(t(a))}\nolimits(x_{t(a),m'}^{(d)}-x^{(c)}_{s(a),m})
\]
and
\[
\eue(Q_0,\gamma_c,\gamma_d)=\prod_{i\in Q_0}\prod_{m=1}^{\gamma_c(i)}\nolimits\prod_{m'=1}^{\gamma_d(i)}\nolimits(x_{i,m'}^{(d)}-x^{(c)}_{i,m}).
\]
\end{definition}
Each of these classes is a product of classes of the form $\eu(i,i',c,d)$, and so we deduce from Proposition \ref{ABprop} that multiplication by $\eue(Q_\iota,\gamma_c,\gamma_d)$ is an injective endomorphism of $\mathcal{H}^{\Sp}_{Q,W,\gamma_1}\boxtimes_+^{\tw}\ldots\boxtimes_+^{\tw}\mathcal{H}^{\Sp}_{Q,W,\gamma_r}$ for $\iota=0,1$.
%\begin{corollary}
%\label{backinj}
%For every pair $\gamma_1,\gamma_2\in \mathbb{Z}^{Q_0}$, the localisation maps
%\[
%\mathcal{H}_{Q,\gamma_1}^{\Sp}\otimes\mathcal{H}_{Q,\gamma_2}^{\Sp}\rightarrow \left(\mathcal{H}_{Q,\gamma_1}^{\Sp}\otimes\mathcal{H}_{Q,\gamma_2}^{\Sp}\right)\otimes_{\Ho_{\Gl_{\gamma_1}\times \Gl_{\gamma_2}}(\pt,\mathbb{Q})}\Ho_{\Gl_{\gamma_1}\times \Gl_{\gamma_2}}(\pt,\mathbb{Q})[\eue(Q_s,\gamma_1,\gamma_2)^{-1}]
%\]
%for $s=0,1$ are injective.
%  More generally, the maps
%\[
%\mathcal{H}_{Q,\gamma_1}^{\Sp}\otimes\mathcal{H}_{Q,\gamma_2}^{\Sp}\rightarrow \mathcal{H}_{Q,\gamma_1}^{\Sp}\otimes\mathcal{H}_{Q,\gamma_2}^{\Sp}\otimes_{A_{\gamma_1+\gamma_2}}A_{\gamma_1,\gamma_2}
%\]
%and
%\begin{align*}
%\mathcal{H}_{Q,\gamma_1}^{\Sp}\otimes\mathcal{H}_{Q,\gamma_2}^{\Sp}\otimes\mathcal{H}_{Q,\gamma_3}^{\Sp}\otimes\mathcal{H}_{Q,\gamma_4}^{\Sp}\rightarrow
%\mathcal{H}_{Q,\gamma_1}^{\Sp}\otimes\mathcal{H}_{Q,\gamma_2}^{\Sp}\otimes\mathcal{H}_{Q,\gamma_3}^{\Sp}\otimes\mathcal{H}_{Q,\gamma_4}^{\Sp}\otimes_{A_{\gamma}}A_{\gamma^0_1,\gamma^0_2}
%\end{align*}
%are too, where here we set 
%\begin{align*}
%\gamma=&\gamma_1+\gamma_2+\gamma_3+\gamma_4\\
%\gamma^0_1=&\gamma_1+\gamma_3\\
%\gamma^0_2=&\gamma_2+\gamma_4.
%\end{align*}
%\end{corollary}
%This injectivity is required in the definition of a localised bialgebra structure (see Definition \ref{lbs}).
\subsection{Definition of the $\Ts_{\gamma}$-equivariant multiplication}  
\label{TCOHA}
Let the finite group $H$ act freely on the topological space $X$, with $X\xrightarrow{p} X/H$ the Galois cover.  Then $p_*\mathbb{Q}_X\cong p_*p^*\mathbb{Q}_{X/H}$ is a sheaf of $\mathbb{Q}[H]$-modules, and we have the following commutative diagram, in which the vertical maps are canonical isomorphisms
\[
\xymatrix{
\ar[d](p_*\mathbb{Q}_X)^H\ar[r]&p_*\mathbb{Q}_X\ar[r]\ar[dr]_b& (p_*\mathbb{Q}_X)_H\ar[d]
\\ \mathbb{Q}_{X/H}\ar[rr]^-c\ar[ur]_a&&\mathbb{Q}_{X/H}
}
\]
the map $a$ is the adjunction map $\mathbb{Q}_{X/H}\rightarrow p_*p^*\mathbb{Q}_{X/H}$ and the map $b$ is obtained from it via Verdier duality.  The isomorphism $c$ is the multiplication by $|H|$ map.  Let $f$ be a regular function on $X/H$.  As usual, we apply $\phi_f$ to these maps, and then restrict to a given subspace $Z\subset X/H$ to obtain the pushforward map in the following diagram
\[
\xymatrix{
\Ho_c(p^{-1}(Z),\varphi_{fp})^{\vee}\ar[r]^-{p_*}\ar[dr]^{\cong}&  \Ho_c(Z,\varphi_f)^{\vee}\\
&(\Ho_c(Z,\varphi_f)_{H})^{\vee}\ar@{^{(}->}[u]
}
\]
where the vertical map is the dual of the quotient map, and the pullback map in the following diagram 
\[
\xymatrix{
\Ho_c(p^{-1}(Z),\varphi_{fp})^{\vee}&  \Ho_c(Z,\varphi_f)^{\vee}\ar[l]_-{p^*}\ar@{->>}[d]\\
&(\Ho_c(Z,\varphi_f)^{H})^{\vee}\ar[ul]_{\cong}
}
\]
Since $ba=\cdot |H|$, we deduce that $p_*p^*=\cdot |H|$.

Define $\Ts_{\gamma}:=\prod_{i\in Q_0} (\Cp^*)^{\gamma(i)}$.  After choosing an ordered basis for $\bigoplus_{i\in Q_0} \Cp^{\gamma(i)}$ respecting the direct sum decomposition, there is a natural inclusion $\Ts_{\gamma}\subset \Gl_{\gamma}$; we consider the natural ordered basis.  Let $\No(\Ts_{\gamma})$ be the normalizer of $\Ts_{\gamma}$ inside $\Gl_{\gamma}$.  For every natural number $N$ there are morphisms
\begin{equation}
\label{nin}
s_N\colon \overline{(\RS_{\gamma},\No(\Ts_{\gamma}))}_N\rightarrow \overline{(\RS_{\gamma},\Gl_{\gamma})}_N
\end{equation}
induced by the inclusion $\No(\Ts_{\gamma})\subset \Gl_{\gamma}$.  As with equation (\ref{phibs}), the maps (\ref{nin}) induce maps 
\begin{equation}
\label{aai}
\varphi_{\tr(W)_{\gamma,N}}\rightarrow s_{N,*}\varphi_{\tr(W)_{\gamma,N}}
\end{equation}
and
\[
s_{N,!}\varphi_{\tr(W)_{\gamma,N}}\state{{l(\gamma)}}\rightarrow \varphi_{\tr(W)_{\gamma,N}}
\]
where 
\begin{equation}
\label{bldef}
l(\gamma):=\sum_{i\in Q_0}(\gamma(i)^2-\gamma(i)), 
\end{equation}
and we abuse notation by denoting by $\tr(W)_{\gamma, N}$ the functions defined by $\tr(W)_{\gamma}$ on $\overline{(\RS_{\gamma},\No(\Ts_{\gamma}))}_N$ as well as on $\overline{(\RS_{\gamma},\Gl_{\gamma})}_N$.
\begin{proposition}
\label{invprop}
There are natural maps
\begin{equation}
\label{invpart}
\Theta\colon :\Ho_{c,\Ts_{\gamma}}(\RS^{\Sp}_{\gamma},\varphi_{\tr(W)_{\gamma}})_{\SG_{\gamma}}\state{{l(\gamma)}}\rightarrow\Ho_{c,\Gl_{\gamma}}(\RS^{\Sp}_{\gamma},\varphi_{\tr(W)_{\gamma}})
\end{equation}
which are isomorphisms.
\end{proposition}
Here the domain is the space of coinvariants with respect to the action --- after we dualize we will consider instead the space of invariants.  
\begin{proof}
For each $N$, there is a Galois cover
\[
w_N\colon \overline{(\RS^{\Sp}_{\gamma},\Ts_{\gamma})}_N\rightarrow \overline{(\RS^{\Sp}_{\gamma},\No(\Ts_{\gamma}))}_N
\]
with Galois group $\SG_{\gamma}$, from which we deduce that $\Ho_{c,\Ts_{\gamma}}(\RS^{\Sp}_{\gamma},w^*\varphi_{\tr(W)_{\gamma}})$ carries a $\SG_{\gamma}$-action and there is an isomorphism in compactly supported cohomology
\[
\Ho_{c,\Ts_{\gamma}}(\RS^{\Sp}_{\gamma},\varphi_{\tr(W)_{\gamma}})_{\SG_{\gamma}}\rightarrow \Ho_{c,\No(\Ts_{\gamma})}(\RS^{\Sp}_{\gamma},\varphi_{\tr(W)_{\gamma}}).
\]
It suffices to prove that 
\[
\Ho_{c,\No(\Ts_{\gamma})}(\RS^{\Sp}_{\gamma},\varphi_{\tr(W)_{\gamma}})\state{{l(\gamma)}}\rightarrow\Ho_{c,\Gl_{\gamma}}(\RS^{\Sp}_{\gamma},\varphi_{\tr(W)_{\gamma}})
\]
is an isomorphism.  This will follow from the claim that (\ref{aai}) is an isomorphism in $\Dbc{\overline{(\RS_{\gamma}^{\Sp},\Gl_{\gamma})}_N}$.  Since $s_N$ is smooth, it follows that 
\[
\varphi_{\tr(W)_{\gamma,N}\circ s_N}\cong s_N^*\varphi_{\tr(W)_{\gamma,N}},
\]
and so the claim follows from the stronger claim that if $\mathcal{F}$ is an object of $\Dbc{\overline{(\RS_{\gamma},\Gl_{\gamma})}_N}$, then the natural map $\mathcal{F}\rightarrow s_{N,*}s_N^*\mathcal{F}$ is an isomorphism.  But this follows from the classical fact that the fibres of $s_N$ have cohomology equal to $\mathbb{Q}$ in degree zero, and zero in other degrees.
\end{proof}
Similarly to \cite[Sec.6.3]{COHA} we describe cohomological Hall algebra operations on the underlying vector space
\[
\mathcal{T}^{\Sp}_{Q,W}:=\bigoplus_{\gamma\in\mathbb{N}^{Q_0}}\mathcal{T}^{\Sp}_{Q,W,\gamma},
\]
where
\[
\mathcal{T}^{\Sp}_{Q,W,\gamma}:=\left(\Ho_{c,\Ts_{\gamma}}(\RS^{\Sp}_{Q,\gamma},\varphi_{\tr(W)_{\gamma}})^{\vee}\right)^{\SG_{\gamma}}\state{{-l(\gamma)+\chi(\gamma,\gamma)/2}}.
\]
Firstly, define 
\[
\overline{\mathcal{T}}^{\Sp}_{Q,W,\gamma}:=\Ho_{c,\Ts_{\gamma}}(\RS^{\Sp}_{Q,\gamma},\varphi_{\tr(W)_{\gamma}})^{\vee}\state{{-l(\gamma)+\chi(\gamma,\gamma)/2}}.
\]
As always, we define the cohomological Hall algebra product as a composition of morphisms.  Let $\gamma_1,\gamma_2$ be dimension vectors in $\mathbb{N}^{Q_0}$, and set $\gamma=\gamma_1+\gamma_2$:
\begin{itemize}
\item
Define 
\begin{align*}
\overline{\delta}_{\Ts}\colon  &\Ho_{c,\Ts_{\gamma}}(\RS^{\Sp}_{\gamma},\varphi_{\tr(W)_{\gamma}})^{\vee}\state{{-l(\gamma)+l_0(\gamma_2,\gamma_1)+\chi(\gamma,\gamma)/2}}\rightarrow \\&\Ho_{c,\Ts_{\gamma}}(\RS^{\Sp}_{\gamma},\varphi_{\tr(W)_{\gamma}})^{\vee}[\eue(Q_0,\gamma_1,\gamma_2)^{-1}]\state{{-l(\gamma)+\chi(\gamma,\gamma)/2}}
\end{align*}
to be division by $\eue(Q_0,\gamma_1,\gamma_2)$.
\item
Define 
\begin{align*}
\underline{\overline{\zeta}}_{\Ts}\colon &\Ho_{c,\Ts_{\gamma}}(\RS^{\Sp}_{\gamma_1,\gamma_2},\varphi_{\tr(W)_{\gamma_1,\gamma_2}})^{\vee}\state{{-l(\gamma)+l_0(\gamma_2,\gamma_1)+\chi(\gamma,\gamma)/2}}\rightarrow \\&\Ho_{c,\Ts_{\gamma}}(\RS^{\Sp}_{\gamma},\varphi_{\tr(W)_{\gamma}})^{\vee}\state{{-l(\gamma)+l_0(\gamma_2,\gamma_1)+\chi(\gamma,\gamma)/2}}
\end{align*}
as the pushforward induced by the inclusion $\RS_{\gamma_1,\gamma_2}\rightarrow \RS_{\gamma}$.

\item
Define 
\begin{align*}
\overline{\alpha}_{\Ts}\colon &\Ho_{c,\Ts_{\gamma}}(\RS^{\Sp}_{\gamma_1}\times \RS^{\Sp}_{\gamma_2},\varphi_{\tr(W)_{\gamma_1}\boxplus\tr(W)_{\gamma_2}})^{\vee}\state{{-l(\gamma)+l_0(\gamma_2,\gamma_1)+l_1(\gamma_1,\gamma_2)+\chi(\gamma,\gamma)/2}}\rightarrow\\& \Ho_{c,\Ts_{\gamma}}(\RS^{\Sp}_{\gamma_1,\gamma_2},\varphi_{\tr(W)_{\gamma_1,\gamma_2}})^{\vee}\state{{-l(\gamma)+l_0(\gamma_2,\gamma_1)+\chi(\gamma,\gamma)/2}}
\end{align*}
as the pullback induced by the affine fibration $\RS_{\gamma_1,\gamma_2}\rightarrow \RS_{\gamma_1}\times \RS_{\gamma_2}$.
\end{itemize}
\begin{remark}
There is no $\overline{\beta}_{\Ts}$ in the above list.  This is because the $\beta$ map (\ref{betaMap}) is given by the pullback map induced by the passage from equivariant cohomology with respect to the $\Gl_{\gamma_1,\gamma_2}$-action to equivariant cohomology with respect to the $\Gl_{\gamma_1}\times \Gl_{\gamma_2}$-action.  When we work with the torus-equivariant CoHA, there is no analogue of this move, since the isomorphism between $\Gl_{\gamma}$-equivariant dual compactly supported cohomology and $\Ts_{\gamma}$-equivariant dual compactly supported cohomology already involves a pullback map.
\end{remark}
An easy calculation shows that
\begin{align*}
-l(\gamma)+l_0(\gamma_2,\gamma_1)+l_1(\gamma_2,\gamma_1)+\chi(\gamma,\gamma)/2=&-l(\gamma_1)-l(\gamma_2)+\chi(\gamma_1,\gamma_1)/2+\chi(\gamma_2,\gamma_2)/2-\\&-\chi(\gamma_2,\gamma_1)/2+\chi(\gamma_1,\gamma_2)/2
\end{align*}
and so the domain of $\overline{\alpha}_T$ is $\overline{\mathcal{T}}^{\Sp}_{Q,W,\gamma_1}\boxtimes^{\tw}_+\overline{\mathcal{T}}^{\Sp}_{Q,W,\gamma_2}$, after applying the Thom--Sebastiani isomorphism.
We define 
\begin{equation}
\label{mBarDef}
\overline{m}\colon \overline{\mathcal{T}}^{\Sp}_{Q,W,\gamma_1}\boxtimes_+^{\tw}\overline{\mathcal{T}}^{\Sp}_{Q,W,\gamma_2}\rightarrow\overline{\mathcal{T}}^{\Sp}_{Q,W,\gamma}[\eue(Q_0,\gamma_1,\gamma_2)^{-1}]
\end{equation}
by $\overline{m}=  \overline{\delta}_{\Ts}\underline{\overline{\zeta}}_{\Ts}\overline{\alpha}_{\Ts}\overline{\TS}_{\Ts}$.  
\begin{example}
Consider the example in which $W=0$, as in Section \ref{NoPot}.  Then 
\[
\overline{\mathcal{T}}_{Q,\gamma}=\bigotimes_{i\in Q_0} \mathbb{Q}[x_{i,1},\ldots,x_{i,\gamma(i)}]
\]
and the map (\ref{mBarDef}) is given by
\begin{align}
\label{unSymmProd}
&\overline{m}(f_1,f_2)(x_{1,1},\ldots,x_{1,\gamma_1(1)+\gamma_2(1)},\ldots,x_{n,1},\ldots,x_{n,\gamma_1(n)+\gamma_2(n)})=\\ \nonumber&f_1(x_{1,1},\ldots,x_{1,\gamma_1(1)},\ldots,x_{n,1},\ldots,x_{n,\gamma_1(n)})\cdot\\ \nonumber&f_2(x_{1,\gamma_1(1)+1},\ldots,x_{1,\gamma_1(1)+\gamma_2(1)},\ldots,x_{n,\gamma_1(n)+1},\ldots,x_{n,\gamma_1(n)+\gamma_2(n)})\cdot\\ \nonumber&\prod_{i,j\in Q_0}\prod_{\alpha=1}^{\gamma_1(i)}\nolimits\prod_{\beta=\gamma_1(j)+1}^{\gamma_1(j)+\gamma_2(j)}\nolimits(x_{j,\beta}-x_{i,\alpha})^{-b_{ij}}.
\end{align}
Note that (\ref{expform}) is obtained from (\ref{unSymmProd}) by restricting to $\SG_{\gamma_1}$-invariant $f_1$ and $\SG_{\gamma_2}$-invariant $f_2$, and summing over shuffles in order to get a $\SG_{\gamma}$-invariant $m(f_1,f_2)$.
\end{example}
Returning to the general case, the space $\overline{\mathcal{T}}^{\Sp}_{Q,W,\gamma}$ carries a $\SG_{\gamma}$-action, and by definition we have $\mathcal{T}^{\Sp}_{Q,W,\gamma}:=\overline{\mathcal{T}}_{Q,W,\gamma}^{\Sp,\SG_{\gamma}}$.  Each of the maps $\overline{\delta}_T,\underline{\overline{\zeta}}_T,\overline{\alpha}_T,\overline{\TS}_T$ are $\SG_{\gamma_1}\times\SG_{\gamma_2}$-equivariant, and restricting to invariant parts we define 
\begin{align*}
\TS_{\Ts}:=&\overline{\TS}_{\Ts}^{\SG_{\gamma_1}\times\SG_{\gamma_2}},\\
\alpha_{\Ts}:=&\overline{\alpha}_{\Ts}^{\SG_{\gamma_1}\times\SG_{\gamma_2}},\\
\underline{\zeta}_{\Ts}:=&\underline{\overline{\zeta}}_{\Ts}^{\SG_{\gamma_1}\times\SG_{\gamma_2}},\\
\delta_{\Ts}:=&\left(\sum_{\pi\in \mathcal{P}(\gamma_1,\gamma_2)}\pi\right)\overline{\delta}_{\Ts}^{\SG_{\gamma_1}\times\SG_{\gamma_2}}.
\end{align*}
Composing these maps we build a map
\[
m_{\Ts}:=\delta_{\Ts}\underline{\zeta}_{\Ts}\alpha_{\Ts}\TS_{\Ts}\colon \mathcal{T}^{\Sp}_{Q,W,\gamma_1}\otimes\mathcal{T}^{\Sp}_{Q,W,\gamma_2}\rightarrow (\Ho_{c,\Ts_{\gamma}}(\RS^{\Sp}_{\gamma},\varphi_{\tr(W)_{\gamma}})_L^{\vee})^{\SG_{\gamma}}
\]
where the subscript $L$ means we formally invert $\pi_*\eue(Q_0,\gamma_1,\gamma_2)$ for every $\pi\in\mathcal{P}(\gamma_1,\gamma_2)$.
\begin{proposition}
\label{TtoH}
Let $\gamma=\gamma_1+\gamma_2$.  Then the following diagram commutes:
\begin{equation}
\label{TtoHd}
\xymatrix@C=10pt{
(\Ho_{c,\Ts_{\gamma}}(\RS^{\Sp}_{\gamma},\varphi_{\tr(W)_{\gamma}})_L^{\vee})^{\SG_{\gamma}}\state{{\spadesuit}}&\Ho_{c,\Gl_{\gamma}}(\RS^{\Sp}_{\gamma},\varphi_{\tr(W)_{\gamma}})^{\vee}\state{\chi(\gamma,\gamma)/2}\ar[l]_-{\xi_1}
\\
(\Ho_{c,\Ts_{\gamma}}(\RS^{\Sp}_{\gamma},\varphi_{\tr(W)_{\gamma}})^{\vee})^{\SG_{\gamma_1}\times\SG_{\gamma_2}}\state{\spadesuit+l_0(\gamma_2,\gamma_1)}\ar[u]^{\delta_{\Ts}}&\Ho_{c,\Gl_{\gamma_1,\gamma_2}}(\RS^{\Sp}_{\gamma},\varphi_{\tr(W)_{\gamma}})^{\vee}\state{{\chi(\gamma,\gamma)/2}}\ar[u]^{\delta}\ar[l]_-{\xi_2}
\\
\Ho_{c,\Ts_{\gamma}}(\RS^{\Sp}_{\gamma_1,\gamma_2},\varphi_{\tr(W)_{\gamma_1,\gamma_2}})^{\vee})^{\SG_{\gamma_1}\times\SG_{\gamma_2}}\state{\diamondsuit+\heartsuit-l_1(\gamma_2,\gamma_1)}\ar[u]^{\underline{\zeta}_{\Ts}}&\Ho_{c,\Gl_{\gamma_1,\gamma_2}}(\RS_{\gamma_1,\gamma_2}^{\Sp},\varphi_{\tr(W)_{\gamma_1,\gamma_2}})^{\vee}\state{{\chi(\gamma,\gamma)/2}}\ar[u]^{\underline{\zeta}}\ar[l]_-{\xi_3}
\\
(\Ho_{c,\Ts_{\gamma}}(\RS^{\Sp}_{\gamma_1}\times \RS^{\Sp}_{\gamma_2},\varphi)^{\vee})^{\SG_{\gamma_1}\times\SG_{\gamma_2}}\state{\diamondsuit+\heartsuit}\ar[u]^{\alpha_{\Ts}}&\Ho_{c,\Gl_{\gamma_1}\times \Gl_{\gamma_2}}(\RS_{\gamma_1}^{\Sp}\times \RS_{\gamma_2}^{\Sp},\varphi)^{\vee}\state{{\heartsuit}}\ar[u]^{\beta^{-1}\alpha}\ar[l]_-{\xi_4}
}
\end{equation}

where in the last line $\varphi=\varphi_{\tr(W)_{\gamma_1}\boxplus \tr(W)_{\gamma_2}}$, all of the $\xi_t$ for $t\geq 2$ are isomorphisms, the shifts are defined by
\begin{align*}
\spadesuit=&\chi(\gamma,\gamma)-l(\gamma)
\\
\diamondsuit=&-l(\gamma_1)-l(\gamma_2),
\\
\heartsuit=&\chi(\gamma,\gamma)/2-\chi(\gamma_2,\gamma_1),
\end{align*}
and $l,l_0$ and $l_1$ are as defined in Equations (\ref{bldef}) and (\ref{ldefs}).
\end{proposition}
Note that the following identity follows from the definitions:
\[
\diamondsuit+\heartsuit-l_1(\gamma_2,\gamma_1)=\spadesuit+l_0(\gamma_2,\gamma_1),
\]
so that the shifts in the domain and the target of $\underline{\zeta}_{\Ts}$ agree --- this map does not involve a shift.
\begin{proof}
All of the $\xi_t$ are defined as the duals of the maps of Proposition \ref{invprop}.

We deal with the commutativity of the constituent squares one by one, working from top to bottom.  

\removelastskip\vskip.5\baselineskip\par\noindent{\bf (Square 1)} Consider the commutative diagram
\[
\xymatrix{
(\Ho_{c,\Ts_{\gamma}}(\RS^{\Sp}_{\gamma},\varphi_{\tr(W)_{\gamma}})^{\vee})^{\SG_{\gamma}}\ar[r]^{\cong}&(\Ho_{c,\Ts_{\gamma}}(\RS^{\Sp}_{\gamma},\varphi_{\tr(W)_{\gamma}})_{\SG_{\gamma}})^{\vee}
\\
(\Ho_{c,\Ts_{\gamma}}(\RS^{\Sp}_{\gamma},\varphi_{\tr(W)_{\gamma}})^{\vee})^{\SG_{\gamma_1}\times\SG_{\gamma_2}}\ar[r]^{\cong}\ar[u]^{\delta^+_{\Ts}}
&(\Ho_{c,\Ts_{\gamma}}(\RS^{\Sp}_{\gamma},\varphi_{\tr(W)_{\gamma}})_{\SG_{\gamma_1}\times \SG_{\gamma_2}})^{\vee}\ar[u]^{\delta'_{\Ts}}
}
\]
where $\delta'_{\Ts}$ is given by the pushforward (in dual compactly supported cohomology) along the maps 
\[
\overline{(\RS^{\Sp}_{\gamma},\No(\Ts_{\gamma_1})\times \No(\Ts_{\gamma_2}))}_N\xrightarrow{b}  \overline{(\RS^{\Sp}_{\gamma},\No(\Ts_{\gamma}))}_N
\]
and $\delta^+_{\Ts}=\left(\sum_{\pi\in \mathcal{P}(\gamma_1,\gamma_2)}\pi\right)$.  We also consider the following commutative diagram
\[
\xymatrix{
X_N\ar@/^1pc/[drr]^{d'}\ar@/_2pc/[rdd]_{c'}\\
&\ar[ul]^-i\overline{(\RS^{\Sp}_{\gamma},\No(\Ts_{\gamma_1})\times \No(\Ts_{\gamma_2}))}_N\ar[r]^-a\ar[d]^b&\overline{(\RS^{\Sp}_{\gamma},\Gl_{\gamma_1,\gamma_2})}_N\ar[d]^c\\
&\overline{(\RS^{\Sp}_{\gamma},\No(\Ts_{\gamma}))}_N\ar[r]^d&\overline{(\RS^{\Sp}_{\gamma},\Gl_{\gamma})}_N
}
\]
where the perimeter of the diagram is a Cartesian square, i.e. 
\[
X_N:=\overline{(\RS^{\Sp}_{\gamma},\No(\Ts_{\gamma}))}_N\times_{\overline{(\RS^{\Sp}_{\gamma},\Gl_{\gamma})}_N}\overline{(\RS^{\Sp}_{\gamma},\Gl_{\gamma_1,\gamma_2})}_N.
\]
Then by the proof of Proposition \ref{invprop} the maps 
\begin{align*}
\mathbb{Q}_{\overline{(\RS^{\Sp}_{\gamma},\Gl_{\gamma_1,\gamma_2})}_N}\rightarrow &a_*\mathbb{Q}_{\overline{(\RS^{\Sp}_{\gamma},\No(\Ts_{\gamma_1})\times \No(\Ts_{\gamma_2}))}_N}\\
\mathbb{Q}_{\overline{(\RS^{\Sp}_{\gamma},\Gl_{\gamma_1,\gamma_2})}_N}\rightarrow &d'_*\mathbb{Q}_{X_N}
\end{align*}
are isomorphisms, and so in turn the map
\[
d'_*\mathbb{Q}_{X_N}\rightarrow a_*\mathbb{Q}_{\overline{(\RS^{\Sp}_{\gamma},\No(\Ts_{\gamma_1})\times \No(\Ts_{\gamma_2}))}_N}
\]
is an isomorphism.  On the other hand, the map $i$ is a closed embedding, and the Euler characteristic of the normal bundle is $\eue(Q_0,\gamma_2,\gamma_1)$: if we pull back along the inclusion of a point $x\hookrightarrow \overline{(\RS^{\Sp}_{\gamma},\Gl_{\gamma})}_N$ we obtain the following diagram
\[
\xymatrix{
P(\gamma)\times \Gr(\gamma_1,\gamma)\ar[drr]\ar[ddr]\\
& \ar[ul]^-{i_x}P(\gamma_1,\gamma)\ar[d]\ar[r]^{a_x} &\Gr(\gamma_1,\gamma)\ar[d]\\
&P(\gamma)\ar[r]&x
}
\]
Here $\Gr(\gamma_1,\gamma)=\prod_{i\in Q_0} \Gr(\mathbb{C}^{\gamma_1},\mathbb{C}^{\gamma})$, and $P(\gamma)=\prod_{i\in Q_0} P(\gamma(i))$, where $P(n)$ is the space of $n$-tuples of unordered linearly independent lines in $\mathbb{C}^n$, and $P(\gamma_1,\gamma)=\prod_{i\in Q_0}P(\gamma_1(i),\gamma(i))$, where $P(n',n)$ is the space of pairs $(T',T)$ where $T\in P(n)$ and $T'\subset T$ has order $n'$.  The map $a_x$ is given by taking the span of the $T'$.  So the inclusion $i_x$ is the inclusion of the space of pairs $(\{T_i\}_{i\in Q_0},\{V_i\}_{i\in Q_0})\in P(\gamma)\times\Gr(\gamma_1,\gamma)$ such that for each $i\in Q_0$ the subspace $V_i\subset\mathbb{C}^{\gamma(i)}$ is the space spanned by the first $\gamma_1(i)$ elements of $T_i$.  The claim regarding the Euler characteristic of the normal bundle is then clear.

By Proposition \ref{maneq} the following diagram commutes
\[
\xymatrix{
\Ho_c(X_N,\varphi_{\tr(W)_N})^{\vee}\ar[r]^-{i^*}\ar[d]^{i^*}&\Ho_c(\overline{(\RS^{\Sp}_{\gamma},\No(\Ts_{\gamma_1})\times \No(\Ts_{\gamma_2}))}_N,\varphi_{\tr(W)_N})^{\vee}\ar[d]^{\id}\\
\Ho_c(\overline{(\RS^{\Sp}_{\gamma},\No(\Ts_{\gamma_1})\times \No(\Ts_{\gamma_2}))}_N,\varphi_{\tr(W)_N})^{\vee}\ar[r]^-{\id}\ar[d]^{i_*}&\Ho_c(\overline{(\RS^{\Sp}_{\gamma},\No(\Ts_{\gamma_1})\times \No(\Ts_{\gamma_2}))}_N,\varphi_{\tr(W)_N})^{\vee}\ar[d]^{\cdot \eue(Q_0,\gamma_1,\gamma_2)}\\
\Ho_c(X_N,\varphi_{\tr(W)_N})^{\vee}\ar[r]^-{i^*}&\Ho_c(\overline{(\RS^{\Sp}_{\gamma},\No(\Ts_{\gamma_1})\times \No(\Ts_{\gamma_2}))}_N,\varphi_{\tr(W)_N})^{\vee}.
}
\]
The top square trivially commutes, while commutativity of the bottom square is the content of Proposition \ref{maneq}.  The horizontal maps are isomorphisms as in Proposition \ref{invprop}, so we deduce that the composition of the leftmost vertical arrows is also given by multiplication by the Euler class $\eue(Q_0,\gamma_1,\gamma_2)$.  Putting everything together, we deduce that
\begin{align*}
b_*a^*=&c'_*i_*i^*d'^*\\
=&c'_*d'^*\circ (\cdot\eue(Q_0,\gamma_1,\gamma_2))\\
=&d^*c_*\circ (\cdot\eue(Q_0,\gamma_1,\gamma_2))
\end{align*}
where the equality $c'_*d'^*=d^*c_*$ is as in Proposition \ref{mixingprop}, so that we have the equality
\[
b_*a^*=d^*c_*\circ(\cdot \eue(Q_0,\gamma_1,\gamma_2))
\]
and so
\[
\delta\epsilon_1\circ (\cdot \eue(Q_0,\gamma_1,\gamma_2))=\epsilon_2\delta_T^+.
\]
It follows that after localising we have the equality
\[
\delta\epsilon_1=\epsilon_2\delta_T.
\]
as required.

\removelastskip\vskip.5\baselineskip\par\noindent{\bf (Square 2)} The following diagram of spaces is a commutative Cartesian diagram, in which the vertical maps are closed inclusions and the horizontal maps are smooth projections:
\[
\xymatrix{
\overline{(\RS^{\Sp}_{\gamma},\No(\Ts_{\gamma}))}_N\ar[r]&\overline{(\RS^{\Sp}_{\gamma}, \Gl_{\gamma_1,\gamma_2})}_N\\
\overline{(\RS^{\Sp}_{\gamma_1,\gamma_2},\No(\Ts_{\gamma}))}_N\ar[r]\ar[u]&\overline{(\RS^{\Sp}_{\gamma_1,\gamma_2}, \Gl_{\gamma_1,\gamma_2})}_N\ar[u].
}
\]
The maps $\zeta_T$, $\zeta$ are obtained by pushforward along the vertical maps, while $\xi_2$ and $\xi_3$ are obtained by pullback along the horizontal arrows.  Commutativity then follows by  Proposition \ref{mixingprop}.

\removelastskip\vskip.5\baselineskip\par\noindent{\bf (Square 3)}  The proof of the commutativity of the bottom square is as in the proof of the commutativity of the second square, using Proposition \ref{mixingprop} in the affine fibration case.
\end{proof}
Since the map $\epsilon_1\delta$ has image contained in $\Ho_{c,\Ts_{\gamma}}(\RS^{\Sp}_{Q,\gamma},\varphi_{\tr(W)_{\gamma}})^{\vee}$ we deduce the following corollaries.
\begin{corollary}
\label{Tfactor}
The map
\[
m_{\Ts}\colon \mathcal{T}^{\Sp}_{Q,W,\gamma_1}\boxtimes_+^{\tw}\mathcal{T}^{\Sp}_{Q,W,\gamma_2}\rightarrow(\Ho_{c,\Ts_{\gamma}}(\RS^{\Sp}_{\gamma},\varphi_{\tr(W)_{\gamma}})_L^{\vee})^{\SG_{\gamma}}
\]
factors through $\mathcal{T}^{\Sp}_{Q,W,\gamma}$, and induces an associative multiplication on $\mathcal{T}^{\Sp}_{Q,W}$, which we will also denote $m_{\Ts}$.  
\end{corollary}
\begin{corollary}
\label{PsiDef}
There is an isomorphism of algebras $\Psi\colon (\mathcal{T}^{\Sp}_{Q,W},m_{\Ts})\cong(\mathcal{H}^{\Sp}_{Q,W},m)$.  Equivalently, the composition
\[
\mathcal{T}^{\Sp}_{Q,W,\gamma_1}\boxtimes_{+}^{\tw}\mathcal{T}_{Q,W,\gamma_2}^{\Sp}\xrightarrow{\cong}\mathcal{H}^{\Sp}_{Q,W,\gamma_1}\boxtimes_{+}^{\tw}\mathcal{H}_{Q,W,\gamma_2}^{\Sp}\xrightarrow{m}\mathcal{H}^{\Sp}_{Q,W,\gamma}\rightarrow  (\Ho_{c,\Ts_{\gamma}}(\RS^{\Sp}_{\gamma},\varphi_{\tr(W)_{\gamma}})_L^{\vee})^{\SG_{\gamma}}
\]
is equal to the map $m_{\Ts}$, where the first and last morphisms are given by the dual of the map $\Theta$ defined in (\ref{invpart}).
\end{corollary}

\section{The comultiplication on $\mathcal{H}^{\Sp}_{Q,W}$}
\label{comult_section}

\subsection{$Q$-localised bialgebras}
\label{QLocBi}
Recall that a $\mathbb{Q}$-linear tensor category $\mathcal{C}$ is a symmetric monoidal $\mathbb{Q}$-linear Abelian category with a monoidal unit $\mathbf{1}_{\mathcal{C}}$, for which the monoidal product is exact in both arguments.  Let $\mathcal{C}$ be as above a $\mathbb{Q}$-linear tensor category, with an exact faithful $\mathbb{Q}$-linear tensor functor $\fib\colon \mathcal{C}\rightarrow\Vect$ to the tensor category of $\mathbb{Q}$-vector spaces, and let $\state{1/2}$ be the shift functor on $\D{\mathcal{C}}$ as in Section \ref{Vsec}.
\smallbreak
We define $\SR_{\gamma}:=\mathbb{Q}[x_{1,1},\ldots,x_{1,\gamma(1)},\ldots,x_{n,1},\ldots,x_{n,\gamma(n)}]^{\SG_{\gamma}}$.  Throughout we consider $\SR_{\gamma}$ as an object of $\D{\mathcal{C}}$ in the obvious way: we first form $\overline{\SR}_{\gamma}$, the free unital commutative algebra in $\D{\mathcal{C}}$ generated by a copy of $\mathbb{Q}\state{{-1}}$ for each generator $x_{h,j}$ in $\mathbb{Q}[x_{1,1},\ldots,x_{1,\gamma(1)},\ldots,x_{n,1},\ldots,x_{n,\gamma(n)}]$, and then define $\SR_{\gamma}$ to be the $\SG_{\gamma}$ invariant part of $\overline{\SR}_{\gamma}$, which we can define, abstractly, as a kernel in each cohomological degree.  For $\gamma_1,\ldots,\gamma_r\in\mathbb{N}^{Q_0}$ an $r$-tuple of dimension vectors we define
\[
\SR_{\gamma_1,\ldots,\gamma_r}=\bigotimes_{\tau=1,\ldots,r}\mathbb{Q}[x^{(\tau)}_{1,1},\ldots,x^{(\tau)}_{1,\gamma_{\tau}(1)},\ldots,x^{(\tau)}_{n,1},\ldots,x^{(\tau)}_{n,\gamma_{\tau}(n)}]^{\SG_{\gamma_{\tau}}}
\]
where, as ever, $n$ is the number of vertices in $Q$.  We define the localised monoidal product as follows:
\[
B'\tilde{\boxtimes}^{\tw}_+B''=\bigoplus_{\gamma_1,\gamma_2\in\mathbb{N}^{Q_0}}\left(B'_{\gamma_1}\boxtimes_+^{\tw} B''_{\gamma_2}\right)\otimes_{\SR_{\gamma_1,\gamma_2}} \SR_{\gamma_1,\gamma_2}\left[\prod_{i,j\in Q_0}\prod_{\substack{s'=1,\ldots,\gamma_1(j)\\s''=1,\ldots,\gamma_2(i)}}\left(x^{(1)}_{j,s'}-x^{(2)}_{i,s''}\right)^{-1}\right]
\]
where we consider the localised algebra $\SR_{\gamma_1,\gamma_2}[\prod_{i,j\in Q_0}\prod_{\substack{s'=1,\ldots,\gamma_1(j)\\s''=1,\ldots,\gamma_2(i)}}(x^{(1)}_{j,s'}-x^{(2)}_{i,s''})^{-1}]$ as an object of $\D{\mathcal{C}}$ in the same way as $\SR_{\gamma_1,\gamma_2}$.  More generally, if $\gamma_1,\ldots,\gamma_r\in\mathbb{N}^{Q_0}$ and 
\[
S\subset \{(\tau,\mu)\in\{1,\ldots,r\}^{\times 2}|\tau\neq \mu\}
\]
then we define
\begin{align*}
&[B_{1}\boxtimes^{\tw}_+\ldots\boxtimes^{\tw}_+ B_{r}]_{S}:=\\&\bigoplus_{\gamma_1,\ldots,\gamma_r\in\mathbb{N}^{Q_0}} B_{\gamma_1}\boxtimes_+^{\tw}\ldots\boxtimes_+^{\tw} B_{\gamma_r}\otimes_{\SR_{\gamma_1,\ldots,\gamma_r}}\SR_{\gamma_1,\ldots,\gamma_r}\left[\prod_{i,j\in Q_0}\prod_{(\tau,\mu)\in S}\prod_{\substack{s'=1,\ldots\gamma_{\tau}(j)\\s''=1,\ldots,\gamma_{\mu}(i)}}(x^{(\tau)}_{j,s'}-x^{(\mu)}_{i,s''})^{-1}\right].
\end{align*}

Recall that the monoidal product $\boxtimes_+^{\tw}$ is not braided.  In contrast, under the above assumption of $\SR_{\gamma}$-actions on the $B_{\gamma}$, there is a natural isomorphism 
\begin{equation}
\label{newsw}
\tilde{\sw}_{\gamma_1,\gamma_2}\colon B_{\gamma_1}\tilde{\boxtimes}^{\tw}_+ B_{\gamma_2}\rightarrow B_{\gamma_2}\tilde{\boxtimes}^{\tw}_+ B_{\gamma_1}
\end{equation}
given by the composition $\cdot\tilde{\eue}_{\gamma_1,\gamma_2} \circ\sw_{\boxtimes_+}$ where $\tilde{\eue}_{\gamma_1,\gamma_2}$ is defined by
\begin{equation}
\label{corr_term}
\tilde{\eue}_{\gamma_1,\gamma_2}:=(-1)^{\gamma_1\cdot\gamma_2}\prod_{i,j\in Q_0}\limits\left(\prod_{m=1}^{\gamma_1(j)}\nolimits\prod_{m'=1}^{\gamma_2(i)}\nolimits(x^{(1)}_{j,m}-x^{(2)}_{i,m'})^{-b_{ij}}(x^{(2)}_{j,m'}-x^{(1)}_{i,m})^{b_{ij}}\right).
\end{equation}
\begin{remark}
We may express $\tilde{\eue}_{\gamma_1,\gamma_2}$ in terms of the operators $\eue(Q_1,\bullet,\bullet)$ and $\eue(Q_0,\bullet,\bullet)$ as follows:
\begin{equation}
\label{lwdd}
\tilde{\eue}_{\gamma',\gamma''}=\eue(Q_1,\gamma_2,\gamma_1)^{-1}\eue(Q_0,\gamma_2,\gamma_1)\eue(Q_1,\gamma_1,\gamma_2)\eue(Q_0,\gamma_1,\gamma_2)^{-1}.
\end{equation}
Equality between (\ref{lwdd}) and (\ref{corr_term}) follows from the equality 
\[
\eue(Q_0,\gamma_2,\gamma_1)\eue(Q_0,\gamma_1,\gamma_2)^{-1}=(-1)^{\gamma_1\cdot\gamma_2}.
\]
\end{remark}
In particular, $\tilde{\sw}_{\gamma_2,\gamma_1}\circ\tilde{\sw}_{\gamma_1,\gamma_2}=\id$.  Taking the direct sum of (\ref{newsw}) over all pairs $\gamma_1,\gamma_2\in\mathbb{N}^{Q_0}$ we define an isomorphism 
\begin{align*}
\tilde{\sw}:=\bigoplus_{\gamma_1,\gamma_2\in\mathbb{N}^{Q_0}}\tilde{\sw}_{\gamma_1,\gamma_2}\colon &B\tilde{\boxtimes_+^{\tw}}B\rightarrow B\tilde{\boxtimes_+^{\tw}}B, \end{align*}
swapping pairs of dimension vectors, with $\tilde{\sw}^2=\id$.  More generally, for $\gamma_1,\ldots,\gamma_r\in\mathbb{N}^{Q_0}$ and $(\tau,\mu)\in S$, we define $\tilde{\sw}_{\tau\mu}\colon [B_{1}\boxtimes^{\tw}_+\ldots\boxtimes^{\tw}_+ B_{r}]_{S}\rightarrow [B_{1}\boxtimes^{\tw}_+\ldots\boxtimes^{\tw}_+ B_{r}]_{(\tau\mu)_*S}$ by swapping the $\tau$ and the $\mu$ factors and multiplying by 
\begin{equation}
(-1)^{\gamma_{\tau}\cdot\gamma_{\mu}}\prod_{i,j\in Q_0}\limits\left(\prod_{m=1}^{\gamma_{\tau}(j)}\nolimits\prod_{m'=1}^{\gamma_{\mu}(i)}\nolimits(x^{(\tau)}_{j,m}-x^{(\mu)}_{i,m'})^{-b_{ij}}(x^{(\mu)}_{j,m'}-x^{(\tau)}_{i,m})^{b_{ij}}\right).
\end{equation}
\begin{definition}
Let $(B,m)$ be an algebra in $\D{\mathcal{C}_Q}$ with respect to the monoidal structure $\boxtimes_+^{\tw}$, such that each $\mathbb{N}^{Q_0}$-graded piece $B_{\gamma}$ carries an $\SR_{\gamma}$-action.  We say that $B$ is crosslinear if for all $i,j\in Q_0$, the map
\[
B_{\gamma_1}\boxtimes_+^{\tw}B_{\gamma_2}\boxtimes_+^{\tw}B_{\gamma_3}\boxtimes_+^{\tw}B_{\gamma_4}\xrightarrow{m\boxtimes_+^{\tw}m}B_{\gamma_1+\gamma_2}\boxtimes_+^{\tw}B_{\gamma_3+\gamma_4}
\]
commutes with multiplication by 
\begin{equation}
\label{easa}
\prod_{i,j\in Q_0}\prod_{\substack{\tau\in\{1,2\}\\ \mu\in\{3,4\}}}\prod_{m=1}^{\gamma_{\tau}(i)}\prod_{m'=1}^{\gamma_{\mu}(j)}(x_{i,m}^{(\tau)}-x_{j,m'}^{(\mu)}).
\end{equation}
Here we use the natural isomorphisms, for $T$ any ordered finite set,
\begin{align*}
&\bigotimes_{\tau\in T}\mathbb{Q}[x^{(\tau)}_{1,1},\ldots,x^{(\tau)}_{1,\gamma_{\tau}(1)},\ldots,x^{(\tau)}_{n,1},\ldots,x^{(\tau)}_{n,\gamma_{\tau}(n)}]\cong \\& \mathbb{Q}[x_{1,1},\ldots,x_{1,\sum_{\tau\in T}\gamma_{\tau}(1)},\ldots,x_{n,1},\ldots,x_{n,\sum_{\tau\in T}\gamma_{\tau}(n)}]
\end{align*}
to realise (\ref{easa}) as an element of $\SR_{\gamma_1+\gamma_2,\gamma_3+\gamma_4}$.\end{definition}
If $B$ is a crosslinear algebra with respect to the monoidal structure $\boxtimes_+^{\tw}$, then we define the product $\tilde{m}^2$ on $B\tilde{\boxtimes}_+^{\tw}B$ via the composition
\begin{equation}
\label{mtildedef}
\xymatrix{
\left(B\tilde{\boxtimes}_+^{\tw}B\right)\boxtimes_+^{\tw}\left(B\tilde{\boxtimes}_+^{\tw}B\right)\ar[r]^-=&[B\boxtimes_+^{\tw}\ldots\boxtimes_+^{\tw} B]_{(1,2),(3,4)}\ar[r] &[B\boxtimes_+^{\tw}\ldots\boxtimes_+^{\tw} B]_{(1,2),(3,2),(3,4),(1,4)}\ar[dll]_{\tilde{\sw}_{(23)}}\\
[B\boxtimes_+^{\tw}\ldots\boxtimes_+^{\tw} B]_{(1,3),(2,3),(2,4),(1,4)}\ar[r]_-{m\tilde{\boxtimes}_+^{\tw}m}&B\tilde{\boxtimes}_+^{\tw}B
}
\end{equation}
where $m\tilde{\boxtimes}_+^{\tw}m$ is the unique extension of $m\boxtimes_+^{\tw}m$ to the localisation with respect to 
\[
\prod_{i,j\in Q_0}\prod_{\substack{\tau\in\{1,2\}\\ \mu\in\{3,4\}}}\prod_{m=1}^{\gamma_{\tau}(i)}\prod_{m'=1}^{\gamma_{\mu}(j)}(x_{i,m}^{(\tau)}-x_{j,m'}^{(\mu)})^{-1},
\]
which exists by the assumption of crosslinearity for $m$.
\begin{definition}
\label{lbs}
A $Q$-localised bialgebra in $\mathcal{C}$ is the data of a unital algebra object $(B,m,\nu\colon \mathbf{1}\rightarrow B)$ in $\D{\mathcal{C}_{Q}}$ and an $\SR_{\gamma}$-module structure for each $B_{\gamma}$, such that $B$ is crosslinear.  We require also the data of a morphism (the $Q$-localised coproduct)
\begin{equation}
\Delta\colon B\rightarrow B\tilde{\boxtimes}^{\tw}_+ B,
\end{equation}
such that each map
\[
\Delta_{\gamma_1,\gamma_2}\colon B_{\gamma_1+\gamma_2}\rightarrow B_{\gamma_1}\tilde{\boxtimes}_+^{\tw}B_{\gamma_2}
\]
is $\SR_{\gamma_1+\gamma_2}$-linear, and the diagram
\[
\xymatrix{
B\ar[d]^{\Delta}\ar[rr]^-{\Delta}&& B\tilde{\boxtimes}_+^{\tw}B\ar[d]^{\id\tilde{\boxtimes}_+^{\tw}\Delta}
\\
B\tilde{\boxtimes}_+^{\tw}B\ar[rr]^-{\Delta\tilde{\boxtimes}_+^{\tw}\id}&&[B\boxtimes_+^{\tw}B\boxtimes_+^{\tw}B]_{\{(1,2),(1,3),(2,3)\}}
}
\]
defined via this linearity commutes.  We require also that the following diagram commutes
\begin{equation}
\label{commdiagB}
\xymatrix{
B\boxtimes_+^{\tw} B\ar[d]^m\ar[r]^-{\Delta\boxtimes_+^{\tw}\Delta}&B\tilde{\boxtimes}_+^{\tw} B\boxtimes_+^{\tw} B\tilde{\boxtimes}_+^{\tw} B\ar[d]^{\tilde{m}^2}\\
B\ar[r]^-{\Delta}&B\tilde{\boxtimes}_+^{\tw} B.
}
\end{equation}
Finally, we require that for all $S\subset \{(\tau,\mu)\in\{1,\ldots,r\}^{\times 2}|\tau\neq \mu\}$ the localisation map
\begin{equation}
\label{reqinj}
B^{\boxtimes_+^{\tw} r}\rightarrow [B^{\boxtimes_+^{\tw} r}]_S
\end{equation}
is injective.
\end{definition}
It is sufficient to check the final condition in the special cases in which $|S|=2$.  Before introducing the coproduct on $\mathcal{H}^{\Sp}_{Q,W}$, we verify that the background assumptions in the definition of a $Q$-localised bialgebra on the algebra $\mathcal{H}_{Q,W}^{\Sp}$ apply.  The injectivity of the map (\ref{reqinj}) is just the statement of Proposition \ref{ABprop}, so we will concentrate on the crosslinearity condition.

\begin{proposition}
\label{coCross}
The cohomological Hall algebra $\mathcal{H}^{\Sp}_{Q,W}$ is crosslinear.
\end{proposition}
\begin{proof}
The class
\[
\bigstar=\prod_{i,j\in Q_0}\prod_{\substack{\tau\in\{1,2\}\\ \mu\in\{3,4\}}}\prod_{m=1}^{\gamma_{\tau}(i)}\prod_{m'=1}^{\gamma_{\mu}(j)}(x_{i,m}^{(\tau)}-x_{j,m'}^{(\mu)})
\]
considered as an element of $\Ho_{\Gl_{\gamma_1}}(\pt,\QQ)\otimes\ldots\otimes\Ho_{\Gl_{\gamma_4}}(\pt,\QQ)$ arises from the class $\bigstar$, considered as an element of $\Ho_{\Gl_{\gamma_1+\gamma_2}}(\pt,\QQ)\otimes\Ho_{\Gl_{\gamma_3+\gamma_4}}(\pt,\QQ)$ under the natural map
\[
\Ho_{\Gl_{\gamma_1+\gamma_2}}(\pt,\QQ)\otimes\Ho_{\Gl_{\gamma_3+\gamma_4}}(\pt,\QQ)\rightarrow\Ho_{\Gl_{\gamma_1}}(\pt,\QQ)\otimes\ldots\otimes\Ho_{\Gl_{\gamma_4}}(\pt,\QQ).
\]
In other words, there is a vector bundle $V_{\bigstar}$ on the stack
\[
\mathfrak{N}=\pt/(\Gl_{\gamma_1+\gamma_2}\times \Gl_{\gamma_3+\gamma_4})
\]
such that multiplication by $\bigstar$ is given by multiplication by the Euler class of $V_{\bigstar}$.  On the other hand, all of the moduli stacks in all of the constituent maps going into the definition of $m\boxtimes_+^{\tw} m$ map naturally to $\mathfrak{N}$.  The result then follows from Proposition \ref{maneq}.
\end{proof}

\subsection{Comultiplication operation on $\mathcal{H}^{\Sp}_{Q,W}$}

For each decomposition $\gamma=\gamma_1+\gamma_2$ we define a map 
\begin{equation}
\Delta_{\gamma_1,\gamma_2}\colon \mathcal{H}^{\Sp}_{Q,W,\gamma}\rightarrow \mathcal{H}^{\Sp}_{Q,W,\gamma_1}\tilde{\boxtimes_+^{\tw}} \mathcal{H}^{\Sp}_{Q,W,\gamma_2}
\end{equation}
that we will show defines the structure of a $Q$-localised coproduct for $\mathcal{H}_{Q,W}^{\Sp}$.

To begin, we work with the torus-equivariant critical cohomological Hall algebra $\mathcal{T}^{\Sp}_{Q,W}\cong\mathcal{H}^{\Sp}_{Q,W}$.  We define
\begin{align*}
\overline{\Delta}_{\Ts,\gamma_1,\gamma_2}\colon &\Ho_{c,\Ts_{\gamma}}(\RS^{\Sp}_{\gamma},\varphi_{\tr(W)_{\gamma}})^{\vee}\state{{-l(\gamma)+\chi(\gamma,\gamma)/2}}\rightarrow\\& \Ho_{c,\Ts_{\gamma_1}}(\RS^{\Sp}_{\gamma_1},\varphi_{\tr(W)_{\gamma_1}})^{\vee}\tilde{\boxtimes}_+^{\tw}\Ho_{c,\Ts_{\gamma_2}}(\RS^{\Sp}_{\gamma_2},\varphi_{\tr(W)_{\gamma_2}})^{\vee}\state{{-l(\gamma_1)-l(\gamma_2)+\chi(\gamma_1,\gamma_1)/2+\chi(\gamma_2,\gamma_2)/2}}
\end{align*}
as the composition of the following maps:
\begin{itemize}
\item
Define 
\begin{align*}
\overleftarrow{\alpha}_{\Ts}\colon &\Ho_{c,\Ts_{\gamma}}\left(\RS^{\Sp}_{\gamma_2,\gamma_1},\varphi_{\tr(W)_{\gamma_2,\gamma_1}}\right)^{\vee}\state{{-l(\gamma_1)-l(\gamma_2)+\chi(\gamma,\gamma)/2-\chi(\gamma_2,\gamma_1)}} \rightarrow \\&\Ho_{c,\Ts_{\gamma}}\left(\RS^{\Sp}_{\gamma_1}\times \RS^{\Sp}_{\gamma_2},\varphi_{\tr(W)_{\gamma_1}\boxplus\tr(W)_{\gamma_2}}\right)^{\vee}[\eue(Q_1,\gamma_1,\gamma_2)^{-1}]\state{{-l(\gamma_1)-l(\gamma_2)+\chi(\gamma,\gamma)/2-\chi(\gamma_2,\gamma_1)}}
\end{align*}
as the pushforward associated to the affine fibration\footnote{Note that this is not the same affine fibration we used to define $\alpha$.} $\RS_{\gamma_2,\gamma_1}\xrightarrow{\pi} \RS_{\gamma_1}\times \RS_{\gamma_2}$.  Note that
\begin{align*}
&-l(\gamma_1)-l(\gamma_2)+\chi(\gamma_1,\gamma_1)/2+\chi(\gamma_2,\gamma_2)/2-\chi(\gamma_2,\gamma_1)/2+\chi(\gamma_1,\gamma_2)/2=\\&-l(\gamma_1)-l(\gamma_2)+\chi(\gamma,\gamma)/2-\chi(\gamma_2,\gamma_1)
\end{align*}
so that the target of $\overleftarrow{\alpha}_{\Ts}$ is the same as the target of $\overline{\Delta}_{\Ts,\gamma_1,\gamma_2}$ after applying the inverse of the Thom--Sebastiani isomorphism.  Here we are using that $\eue(Q_1,\gamma_1,\gamma_2)$ is the Euler class of $\pi$.
%, which has an obvious $\Ts_{\gamma}$-equivariant section given by the inclusion of block diagonal matrices into block lower triangular matrices.
\item
Define 
\begin{align*}
&\overleftarrow{\beta^{-1}}_{\Ts}\colon \Ho_{c,\Ts_{\gamma}}\left(\RS^{\Sp}_{\gamma_2,\gamma_1},\varphi_{\tr(W)_{\gamma_2,\gamma_1}}\right)^{\vee}\state{{-l(\gamma_1)-l(\gamma_2)+\chi(\gamma,\gamma)/2-\chi(\gamma_2,\gamma_1)-l_0(\gamma_2,\gamma_1)}}\rightarrow \\&\Ho_{c,\Ts_{\gamma}}\left(\RS^{\Sp}_{\gamma_2,\gamma_1},\varphi_{\tr(W)_{\gamma_2,\gamma_1}}\right)^{\vee}\state{{-l(\gamma_1)-l(\gamma_2)+\chi(\gamma,\gamma)/2-\chi(\gamma_2,\gamma_1)}}
\end{align*}
to be multiplication by 
\[
\prod_{i\in Q_0}\prod_{m=1}^{\gamma_1(i)}\nolimits\prod_{m'=1}^{\gamma_2(i)}\nolimits(x_{i,m}-x_{i,m'+\gamma_1(i)})=\eue(Q_0,\gamma_2,\gamma_1),
\]
where here we have made the natural identification
\[
\SR_{\gamma_1,\gamma_2}=\left(\bigotimes_{i\in Q_0} \mathbb{Q}[x_{i,1},\ldots,x_{i,\gamma_1(i)+\gamma_2(i)}]\right)^{\Sym_{\gamma_1}\times\Sym_{\gamma_2}}.
\]

\item
Define 
\begin{align*}
\overleftarrow{\underline{\zeta}}_{\Ts}\colon  &\Ho_{c,\Ts_{\gamma}}\left(\RS^{\Sp}_{\gamma},\varphi_{\tr(W)_{\gamma}}\right)^{\vee}\state{{-l(\gamma)+\chi(\gamma,\gamma)/2}}\rightarrow \\&\Ho_{c,\Ts_{\gamma}}\left(\RS^{\Sp}_{\gamma_2,\gamma_1},\varphi_{\tr(W)_{\gamma_2,\gamma_1}}\right)^{\vee}\state{{-l(\gamma_1)-l(\gamma_2)+\chi(\gamma,\gamma)/2-\chi(\gamma_2,\gamma_1)-l_0(\gamma_2,\gamma_1)}}
\end{align*}
as the pullback induced by the inclusion $\RS_{\gamma_2,\gamma_1}\rightarrow \RS_{\gamma}$.
\item
Define 
\begin{align*}
\overleftarrow{\delta}_{\Ts}\colon  &\Ho_{c,\Ts_{\gamma}}\left(\RS^{\Sp}_{\gamma_2,\gamma_1},\varphi_{\tr(W)_{\gamma_2,\gamma_1}}\right)^{\vee}\state{{-l(\gamma_1)-l(\gamma_2)+\chi(\gamma,\gamma)/2-\chi(\gamma_2,\gamma_1)+l_0(\gamma_2,\gamma_1)}}\rightarrow\\ &\Ho_{c,\Ts_{\gamma}}\left(\RS^{\Sp}_{\gamma_2,\gamma_1},\varphi_{\tr(W)_{\gamma_2,\gamma_1}}\right)^{\vee}\state{{-l(\gamma_1)-l(\gamma_2)+\chi(\gamma,\gamma)/2-\chi(\gamma_2,\gamma_1)+l_0(\gamma_2,\gamma_1)}}
\end{align*}
to be the identity map.
\end{itemize}

\begin{definition}
The map 
\[
\overline{\Delta}_{\Ts,\gamma_1,\gamma_2}\colon \overline{\mathcal{T}}^{\Sp}_{Q,W,\gamma}\rightarrow \overline{\mathcal{T}}^{\Sp}_{Q,W,\gamma_1}\tilde{\boxtimes}_+^{\tw} \overline{\mathcal{T}}^{\Sp}_{Q,W,\gamma_2}
\]
is defined by $\overline{\Delta}_{\Ts,\gamma_1,\gamma_2}=\overline{\TS}^{-1}\overleftarrow{\alpha}_{\Ts}\overleftarrow{\beta^{-1}}_{\Ts}\overleftarrow{\underline{\zeta}}_{\Ts}\overleftarrow{\delta}_{\Ts}
$
and 
\[
\Delta_{\Ts,\gamma_1,\gamma_2}\colon \mathcal{T}^{\Sp}_{Q,W}\rightarrow \mathcal{T}^{\Sp}_{Q,W,\gamma_1}\tilde{\boxtimes}_+^{\tw} \mathcal{T}^{\Sp}_{Q,W,\gamma_2}
\]
is defined by restricting to the $\SG_{\gamma}$ invariant part of the domain.
The map
\[
\Delta_{\Ts}\colon \mathcal{T}^{\Sp}_{Q,W}\rightarrow \bigoplus_{\gamma_1,\gamma_2\in\mathbb{N}^{Q_0}}\mathcal{T}^{\Sp}_{Q,W,\gamma_1}\tilde{\boxtimes}_+^{\tw} \mathcal{T}^{\Sp}_{Q,W,\gamma_2}
\]
is defined to be the sum
\begin{equation}
\label{Deltadef}
\Delta_{\Ts}=\sum_{\gamma_1,\gamma_2\in\mathbb{N}^{Q_0}}\Delta_{\Ts,\gamma_1,\gamma_2}.
\end{equation}
\end{definition}
%\begin{proposition}
%Assume that $\Sp$ is grade-preserving.  Then $\Delta$ is a map in the category $\MMHS$ and preserves cohomological degree.
%\end{proposition}
\begin{definition}
We define the coproduct
\[
\Delta\colon \mathcal{H}_{Q,W}^{\Sp}\rightarrow \mathcal{H}_{Q,W}^{\Sp}\tilde{\boxtimes}^{\tw}_{+}\mathcal{H}_{Q,W}^{\Sp}
\]
via the formula
\[
\Delta:=(\Psi\tilde{\boxtimes}^{\tw}_{+}\Psi)\circ \Delta_{\Ts}\circ \Psi^{-1},
\]
where $\Psi$ is the isomorphism of Corollary \ref{PsiDef}.
\end{definition}
Since each of the maps defining $\overline{\Delta}_{\Ts}$ is $\Ho_{\Ts_{\gamma}}(\pt,\QQ)$-linear, and hence also $\Ho_{\Ts_{\gamma}}(\pt,\QQ)^{\SG_{\gamma}}$-linear, we deduce the following proposition.
\begin{proposition}
The map $\Delta$ is $\SR_{\gamma}:=\Ho_{\Gl_{\gamma}}(\pt,\QQ)$-linear.  
\end{proposition}
The following is proved in just the same way as the associativity of $m$, see \cite[Sec.2.3]{COHA}.
\begin{proposition}
The following diagram commutes
\[
\xymatrix{
\mathcal{H}_{Q,W}^{\Sp}\ar[rr]^-{\Delta}\ar[d]_{\Delta}&&\mathcal{H}_{Q,W}^{\Sp}\tilde{\boxtimes}_+^{\tw}\mathcal{H}_{Q,W}^{\Sp}\ar[d]^-{\id\tilde{\boxtimes}^{\tw}_+\Delta}
\\
\mathcal{H}_{Q,W}^{\Sp}\tilde{\boxtimes}_+^{\tw}\mathcal{H}_{Q,W}^{\Sp}
\ar[rr]^-{\Delta\tilde{\boxtimes}^{\tw}_+\id}
&&
[\mathcal{H}^{\Sp}_{Q,W}\boxtimes_+^{\tw}\mathcal{H}^{\Sp}_{Q,W}\boxtimes_+^{\tw}\mathcal{H}^{\Sp}_{Q,W}]_{\{(1,2),(1,3),(2,3)\}}.
}
\]
\end{proposition}

\subsection{Another description of the coproduct}
We now describe directly the coproduct 
\[
\Delta\colon \mathcal{H}^{\Sp}_{Q,W}\rightarrow \mathcal{H}^{\Sp}_{Q,W}\tilde{\boxtimes_+^{\tw}}\mathcal{H}_{Q,W}^{\Sp}
\]
without reference to $\mathcal{T}^{\Sp}_{Q,W}$.  The morphisms $\alpha,\beta,\delta,\underline{\zeta}$ are as defined in Section \ref{MOsec}.  By replacing pushforwards by pullbacks, and vice versa, we arrive at morphisms
\begin{align*}
\label{pfdef}
\overleftarrow{\delta}\colon  &\Ho_{c,\Gl_{\gamma}}(\RS^{\Sp}_{\gamma},\varphi_{\tr(W)_{\gamma}})^{\vee}\state{\chi(\gamma,\gamma)/2-\chi(\gamma_1,\gamma_2)-l_1(\gamma_1,\gamma_2)+l_0(\gamma_1,\gamma_2)}\rightarrow\\& \Ho_{c,\Gl_{\gamma_1,\gamma_2}}(\RS^{\Sp}_{\gamma},\varphi_{\tr(W)_{\gamma}})^{\vee}\state{\chi(\gamma,\gamma)/2-\chi(\gamma_1,\gamma_2)-l_1(\gamma_1,\gamma_2)}
\\\\
\overleftarrow{\underline{\zeta}}\colon  &\Ho_{c,\Gl_{\gamma_1,\gamma_2}}(\RS^{\Sp}_{\gamma},\varphi_{\tr(W)_{\gamma}})^{\vee}\state{\chi(\gamma,\gamma)/2-\chi(\gamma_1,\gamma_2)-l_1(\gamma_1,\gamma_2)}\rightarrow\\& \Ho_{c,\Gl_{\gamma_1,\gamma_2}}(\RS^{\Sp}_{\gamma_1,\gamma_2},\varphi_{\tr(W)_{\gamma_1,\gamma_2}})^{\vee}\state{\chi(\gamma,\gamma)/2-\chi(\gamma_1,\gamma_2)}
\\\\
\overleftarrow{\alpha}\colon  &\Ho_{c,\Gl_{\gamma_1}\times \Gl_{\gamma_2}}(\RS_{\gamma_1,\gamma_2}^{\Sp},\varphi_{\tr(W)_{\gamma_1,\gamma_2}})^{\vee}\state{\chi(\gamma,\gamma)/2-\chi(\gamma_1,\gamma_2)}\rightarrow \\&\Ho_{c,\Gl_{\gamma_1}\times \Gl_{\gamma_2}}(\RS_{\gamma_1}^{\Sp}\times \RS_{\gamma_2}^{\Sp},\varphi_{\tr(W)_{\gamma_1}\boxplus\tr(W)_{\gamma_2}})^{\vee}\state{{\chi(\gamma,\gamma)/2-\chi(\gamma_1,\gamma_2)}}.
\end{align*}
\begin{proposition}
\label{UnT}
The following diagram commutes:
\[
\xymatrix{
\Ho_{c,\Gl_{\gamma}}(\RS^{\Sp}_{\gamma},\varphi_{\tr(W)_{\gamma}})^{\vee}\state{{\chi(\gamma,\gamma)/2}}\ar[d]^{\overleftarrow{\delta}}\ar[r]&(\Ho_{c,\Ts_{\gamma}}(\RS^{\Sp}_{\gamma},\varphi_{\tr(W)_{\gamma}})^{\vee})^{\SG_{\gamma}}\state{{\spadesuit}}\ar[d]^{\overleftarrow{\delta}_{\Ts}}
\\
\Ho_{c,\Gl_{\gamma_1,\gamma_2}}(\RS^{\Sp}_{\gamma},\varphi_{\tr(W)_{\gamma}})^{\vee}\state{{\chi(\gamma,\gamma)/2-l_0(\gamma_1,\gamma_2)}}\ar[d]^{\overleftarrow{\underline{\zeta}}}\ar[r]&(\Ho_{c,\Ts_{\gamma}}(\RS^{\Sp}_{\gamma},\varphi_{\tr(W)_{\gamma}})^{\vee})^{\SG_{\gamma}}\state{{\spadesuit}}\ar[d]^{\overleftarrow{\underline{\zeta}}_{\Ts}}
\\
\Ho_{c,\Gl_{\gamma_1,\gamma_2}}(\RS^{\Sp}_{\gamma_1,\gamma_2},\varphi_{\tr(W)_{\gamma}})^{\vee}\state{{\heartsuit}}\ar[d]^{\beta\cdot \eue(Q_0,\gamma_2,\gamma_1)}\ar[r]&(\Ho_{c,\Ts_{\gamma}}(\RS^{\Sp}_{\gamma_1,\gamma_2},\varphi_{\tr(W)_{\gamma}})^{\vee})^{\SG_{\gamma}}\state{{\heartsuit+\diamondsuit-l_1(\gamma_1,\gamma_2)}}\ar[d]^{\overleftarrow{\beta^{-1}}_{\Ts}}
\\
\Ho_{c,\Gl_{\gamma_1}\times \Gl_{\gamma_2}}(\RS^{\Sp}_{\gamma_1,\gamma_2},\varphi_{\tr(W)_{\gamma_1,\gamma_2}})^{\vee}\state{{\heartsuit}}\ar[d]^{\overleftarrow{\alpha}}\ar[r]&(\Ho_{c,\Ts_{\gamma}}(\RS^{\Sp}_{\gamma_1,\gamma_2},\varphi_{\tr(W)_{\gamma_1,\gamma_2}})^{\vee})^{\SG_{\gamma}}\state{{\heartsuit+\diamondsuit}}\ar[d]^{\overleftarrow{\alpha}_{\Ts}}
\\
\mathcal{H}^{\Sp}_{\gamma_1}\tilde{\boxtimes_+^{\sw}}\mathcal{H}^{\Sp}_{\gamma_2}\ar[r]&\mathcal{T}^{\Sp}_{\gamma_1}\tilde{\boxtimes_+^{\sw}}\mathcal{T}^{\Sp}_{\gamma_2}
}
\]
where 
\[
\diamondsuit=-l(\gamma_1)-l(\gamma_2),
\]
\[
\heartsuit=\chi(\gamma,\gamma)/2-\chi(\gamma_1,\gamma_2),
\]
\begin{align*}
\spadesuit=&\heartsuit+\diamondsuit-l_0(\gamma_1,\gamma_2)-l_1(\gamma_1,\gamma_2)\\
=&\chi(\gamma,\gamma)/2-l(\gamma),
\end{align*}
and the horizontal maps are the isomorphisms of Proposition \ref{TtoH}.
\end{proposition}
\begin{proof}
The top two squares contain only pullback maps, since $\overleftarrow{\delta}_{\Ts}$ is the pullback map associated to the identity map.  It follows that the top two squares commute, since the underlying diagrams of spaces commute.  Similarly, up to multiplication by the factor $\eue(Q_0, \gamma_2,\gamma_1)$, the third square is composed entirely of pullbacks in a commutative square of spaces --- note that $\overleftarrow{\beta^{-1}}_{\Ts}$ is defined to just be multiplication by $\eue(Q_0,\gamma_2,\gamma_1)$.  The bottom quadrilateral commutes by Proposition \ref{mixingprop}.
\end{proof}
\begin{corollary}
\label{simplerComult}
There is an equality of morphisms
\[
\Delta=\eue(Q_0,\gamma_2,\gamma_1)\cdot\overleftarrow{\alpha}\beta\overleftarrow{\underline{\zeta}}\overleftarrow{\delta}.
\]
\end{corollary}
\begin{remark}
The map $\cdot\eue(Q_0,\gamma_2,\gamma_1)\beta$ is the natural definition of the umkehr map for $\beta^{-1}$.  So we can say that the coproduct is defined by taking all of the constituent maps in the definition of the product and replacing them with umkehr maps.
\end{remark}
\begin{remark}
Coming back to the case $W=0$ considered in the introduction, it is easy to check that the above formula for $\Delta$ recovers the naive guess for the definition of the coproduct of Section \ref{NoPot} in the symmetric case.
\end{remark}

\subsection{Proof of the main theorem}
Now we come to our main theorem regarding the operation $\Delta$.

\begin{theorem}
\label{comultalg}
Let $\gamma\in\mathbb{N}^{Q_0}$, and let $\gamma^0_1,\gamma^0_2,\gamma_1,\gamma_2$ satisfy
\begin{align}
\label{decompCond}
\gamma_1^0+\gamma_2^0=\gamma\\ \nonumber
\gamma_1+\gamma_2=\gamma.
\end{align}
Then the following diagram commutes, with $\tilde{m}^2$ defined as in (\ref{mtildedef}):
\begin{equation}
\label{commdiag}
\xymatrix{
\mathcal{H}^{\Sp}_{Q,W,\gamma^0_1}\boxtimes_+^{\tw} \mathcal{H}^{\Sp}_{Q,W,\gamma^0_2}\ar[dd]^-{\sum\limits_{\substack{\gamma'_1+\gamma'_2=\gamma^0_1\\\gamma'_3+\gamma'_4=\gamma^0_2}}\Delta_{\gamma'_1,\gamma'_2}\boxtimes_+^{\tw}\Delta_{\gamma'_3,\gamma'_4}}\ar[rrr]^-{m}
&&&
\mathcal{H}^{\Sp}_{Q,W,\gamma}\ar[dd]^{\Delta_{\gamma_1,\gamma_2}}
\\ \\
{\bigoplus\limits_{\substack{\gamma'_1+\gamma'_2=\gamma^0_1\\\gamma'_3+\gamma'_4=\gamma^0_2}}\left[
\begin{matrix}
\mathcal{H}^{\Sp}_{Q,W,\gamma'_1}\boxtimes_+^{\tw}\mathcal{H}^{\Sp}_{Q,W,\gamma_2'} \boxtimes_+^{\tw} \\ \boxtimes_+^{\tw}\mathcal{H}^{\Sp}_{Q,W,\gamma_3'}\boxtimes_+^{\tw}\mathcal{H}^{\Sp}_{Q,W,\gamma_4'} 
\end{matrix}\right]_{\left\{(1,2),(3,4)\right\}}
}
\ar[rrr]^-{\tilde{m}^{ 2}}
&&&
\mathcal{H}^{\Sp}_{Q,W,\gamma_1}\tilde{\boxtimes_+^{\tw}} \mathcal{H}^{\Sp}_{Q,W,\gamma_2}
}
\end{equation}
and the operation $\Delta$ defines a $Q$-localised bialgebra structure on $\mathcal{H}^{\Sp}_{Q,W}$ in the category $\mathcal{C}$. 
\end{theorem}
\begin{proof}
For $\gamma_1,\gamma_2,\gamma_1^0,\gamma_2^0$ as in (\ref{decompCond}), we define $\DData_{\gamma_1,\gamma_2,\gamma_1^0,\gamma_2^0}$ to be the set of $(\gamma'_1,\gamma'_2,\gamma'_3,\gamma'_4)\in(\mathbb{N}^{Q_0})^4$ such that
\begin{align*}
\gamma'_1+\gamma'_2=&\gamma_1^0\\
\gamma'_1+\gamma'_3=&\gamma_1,
\end{align*}
or if the data $(\gamma_1,\gamma_2,\gamma_1^0,\gamma_2^0)$ is fixed, we abbreviate $\DData_{\gamma_1,\gamma_2,\gamma_1^0,\gamma_2^0}$ to $\DData$.

\begin{lemma}
\label{Dround}
Let $X$ be a complex algebraic variety carrying an action of $\Gl_{\gamma}$, let $X^{\Sp}$ be a $\Gl_{\gamma}$-equivariant subvariety, and define $V=\Ho_{c,\Ts_{\gamma}}(X^{\Sp},\varphi_f)^{\vee}$.  For $\nabla=(\gamma'_1,\gamma'_2,\gamma'_3,\gamma'_4)\in\DData_{\gamma_1,\gamma_2,\gamma_1^0,\gamma_2^0}$ we consider the diagram
\[
\xymatrix{
V^{\SG_{\gamma_1^0}\times\SG_{\gamma_2^0}}\ar[r]^-a\ar[d]^-{b_{\nabla}}&V^{\SG_{\gamma}}\ar[d]^-c\\
V^{\SG_{\gamma'_1}\times\SG_{\gamma'_2}\times\SG_{\gamma'_3}\times\SG_{\gamma'_4}}\ar[r]^-{d_{\nabla}}&V^{\SG_{\gamma_1}\times\SG_{\gamma_2}}
}
\]
of underlying vector spaces (the maps we define below do not respect the cohomological grading).  We define $a$ by demanding the commutativity of the following diagram
\[
\xymatrix{
\Ho_{c,\Gl_{\gamma_1^0,\gamma_2^0}}(X^{\Sp},\varphi_f)^{\vee}\ar[d]^{\cong}\ar[r]&\Ho_{c,\Gl_{\gamma}}(X^{\Sp},\varphi_f)^{\vee}\ar[d]^{\cong}\\
V^{\SG_{\gamma}}\ar[r]^-a&V^{\SG_{\gamma_1^0}\times\SG_{\gamma_2^0}}
}
\]
where the top map is defined in the same way as $\underline{\zeta}$, and the vertical isomorphisms as in Proposition \ref{invprop}.  Similarly we define $c$ in the same way as $\overleftarrow{\underline{\zeta}}$.  Likewise we define $d_{\nabla}$ as the map making the following diagram commute
\[
\xymatrix{
\Ho_{c,\Gl_{\gamma'_1,\gamma'_3}\times \Gl_{\gamma'_2,\gamma'_4}}(X^{\Sp},\varphi_f)^{\vee}\ar[d]^{\cong}\ar[r]& \Ho_{c,\Gl_{\gamma_1}\times \Gl_{\gamma_2}}(X^{\Sp},\varphi_f)^{\vee}\ar[d]^{\cong}\\
V^{\SG_{\gamma'_1}\times\SG_{\gamma'_2}\times\SG_{\gamma'_3}\times\SG_{\gamma'_4}}\ar[r]^-{d_{\nabla}}&V^{\SG_{\gamma_1}\times\SG_{\gamma_2}}
}
\]
where the top map is the pushforward, and $b_{\nabla}$ similarly, using the pullback.  Inside 
\[
\Ho_{c,\Gl_{\gamma_1}\times \Gl_{\gamma_2}}(\pt,\QQ)\cong\mathbb{Q}[x^{(1)}_{1,1},\ldots,x^{(1)}_{n,\gamma_2(n)},x^{(2)}_{1,1},\ldots,x^{(2)}_{n,\gamma_2(n)}]^{\SG_{\gamma_2}},
\]
let $\mathcal{S}_{\gamma_1,\gamma_2}$ be the the multiplicative set $\{z,z^2,\ldots\}$ with $z=\sum_{i\in Q_0}(x^{(2)}_{i,1}+\ldots x^{(2)}_{i,\gamma_2(i)})$.  Then after localising the target at $\mathcal{S}_{\gamma_1,\gamma_2}$, we have the equality of maps
\[
(\cdot \eue(Q_0,\gamma_2,\gamma_1))\circ c a=\sum_{\nabla\in\DData}(-1)^{\gamma'_2\cdot\gamma'_3}d_{\nabla}\circ (\cdot \eue(Q_0,\gamma'_2,\gamma'_1)\eue(Q_0,\gamma'_4,\gamma'_3))\circ b_{\nabla}.
\]
\end{lemma}
\begin{proof}
Consider the Cartesian diagram
\[
\xymatrix{
\overline{(X^{\Sp},\Gl_{\gamma^0_1,\gamma^0_2})}_N\ar[d]^-{s_N}&Z_N\ar[d]^-{t_N}\ar[l]^-{r_N}\\
\overline{(X^{\Sp},\Gl_{\gamma})}_N& \overline{(X^{\Sp},\Gl_{\gamma_1,\gamma_2})}_N\ar[l]^-{u_N}
}
\]
where $Z_N$ is the fibre product.  We can explicitly describe $Z_N$ as the quotient of the space $\overline{Z}_N$ by the action of $\Gl_{\gamma_1,\gamma_2}$, where $\overline{Z}_N$ is the space of triples $(x,\mathbf{f},V)$, where $x\in X^{\Sp}$, 
\[
\mathbf{f}=\left(\mathbf{f}_{1,1},\ldots,\mathbf{f}_{1,\gamma(1)},\ldots,\mathbf{f}_{n,1},\ldots,\mathbf{f}_{n,\gamma(n)}\right)\in\Fr(\sum_{i\in Q_0}\gamma(i),N)
\]
and $V$ is an $i$-tuple of vector subspaces $V_i\subset \Span(\mathbf{f}_{i,1},\ldots,\mathbf{f}_{i,\gamma(i)})$ where $\dim(V_i)=\gamma^0_1(i)$, and we let $\Gl_{\gamma_1,\gamma_2}$ act trivially on the space of such $V$.  Via the inclusion defined by $\textbf{f}$ there is a $\Gl_{\gamma}$-equivariant isomophism $\overline{Z}_N\cong \overline{Z}'_N$ where $\overline{Z}'_N$ is the space of triples $(x,V',\mathbf{f})$, where $x$ and $\mathbf{f}$ are as before, but $V'_i\subset \mathbb{C}^{\gamma_i}$ is a $\gamma_1^0(i)$-dimensional subspace of a fixed $\gamma$-dimensional vector space, and $\Gl_{\gamma_1,\gamma_2}$ acts via translation of that vector space, i.e. 
\[
V'\in\Gr(\gamma_1^0,\gamma):=\prod_{i\in Q_0}\Gr(\gamma_1^0(i),\gamma(i)).  
\]
In other words we have $Z_N\cong \overline{(X^{\Sp}\times \Gr(\gamma_1^0,\gamma),\Gl_{\gamma_1,\gamma_2})}_N$.  The map $ca$ is defined by $u^*s_*$, and by Proposition \ref{ppCor} we can write $u^*s_*=t_*r^*$.

Consider the commutative diagram of spaces
\[
\xymatrix{
\overline{(X\times \Gr(\gamma_1^0,\gamma),\Gl_{\gamma_1,\gamma_2})}_N\ar[d]^{t_N}&\ar[l](X\times \overline{\Gr(\gamma_1^0,\gamma),\Gl_{\gamma_1}\times \Gl_{\gamma_2})}_N\ar[d]^{t'_N}\\
\overline{(X,\Gl_{\gamma_1,\gamma_2})}_N&\ar[l]\overline{(X,\Gl_{\gamma_1}\times \Gl_{\gamma_2})}_N.
}
\]
We denote by $\overline{f}$ the functions on each of these spaces induced by the projection to $X$.
 Then by Proposition \ref{mixingprop} we have a commutative diagram in cohomology
\[
\xymatrix{
\Ho_{c,\Gl_{\gamma_1,\gamma_2}}(X^{\Sp}\times \Gr(\gamma_1^0,\gamma),\varphi_{\overline{f}})^{\vee}\ar[d]^{t_*}\ar[r]&\Ho_{c,\Gl_{\gamma_1}\times \Gl_{\gamma_2}}(X^{\Sp}\times \Gr(\gamma_1^0,\gamma),\varphi_{\overline{f}})^{\vee}\ar[d]^{t'_*}\\
\Ho_{c, \Gl_{\gamma_1,\gamma_2}}(X^{\Sp},\varphi_f)^{\vee}\ar[r]&\Ho_{c,\Gl_{\gamma_1}\times \Gl_{\gamma_2}}(X^{\Sp},\varphi_{\overline{f}})^{\vee}
}
\]
where the horizontal maps are the ismorphisms induced by pullback along affine fibrations.  Define $(X^{\Sp}\times \Gr(\gamma_1^0,\gamma))^{\Sph^1}$ to be the fixed point locus of $X^{\Sp}\times \Gr(\gamma_1^0,\gamma)$ under the action of 
\[
\Sph^1=\{z\in\mathbb{C}\textrm{ such that }|z|=1\}\subset \Gl_{\gamma_2}, 
\]
embedded via $z\mapsto (z\cdot\id_{\mathbb{C}^{\gamma_2(i)}})_{i\in Q_0}$.  Then we have a decomposition
\[
(X^{\Sp}\times \Gr(\gamma_1^0,\gamma))^{\Sph^1}=\coprod_{(\gamma'_1,\ldots,\gamma'_4)\in\DData}X^{\Sp,\Sph^1}\times \Gr(\gamma'_1,\gamma_1)\times\Gr(\gamma'_2,\gamma_2)
\]
and so, in particular, an inclusion 
\[
(X^{\Sp}\times \Gr(\gamma_1^0,\gamma))^{\Sph^1}\subset X^{\Sp}\times \coprod_{(\gamma'_1,\ldots,\gamma'_4)\in\DData} \Gr(\gamma'_1,\gamma_1)\times\Gr(\gamma'_2,\gamma_2)
\]
and so the bottom map in the commutative diagram 
\[
\xymatrix{
\Ho_{c,\Gl_{\gamma_1}\times \Gl_{\gamma_2}}(X^{\Sp}\times \Gr(\gamma_1^0,\gamma),\varphi_{\overline{f}})^{\vee}_{\mathcal{S}_{\gamma_1,\gamma_2}}\ar[r]^-{\res'} \ar[d]^{\cong}&\Ho_{c,\Gl_{\gamma_1}\times \Gl_{\gamma_2}}(X^{\Sp}\times\Gr(\gamma_1^0,\gamma),\varphi_{\overline{f}}|_{X\times\Gr(\gamma'_1,\gamma_1)\times\Gr(\gamma'_2,\gamma_2)})^{\vee}_{\mathcal{S}_{\gamma_1,\gamma_2}}\ar[d]^{\cong}
\\
\Ho_{\Gl_{\gamma_1}\times \Gl_{\gamma_2}}(X^{\Sp}\times \Gr(\gamma_1^0,\gamma),E)_{\mathcal{S}_{\gamma_1,\gamma_2}}\ar[r] &\Ho_{\Gl_{\gamma_1}\times \Gl_{\gamma_2}}(X^{\Sp}\times\Gr(\gamma_1^0,\gamma),E|_{X\times\Gr(\gamma'_1,\gamma_1)\times\Gr(\gamma'_2,\gamma_2)})_{\mathcal{S}_{\gamma_1,\gamma_2}}
}
\]
is an isomorphism by \cite[Thm.6.2]{GKM97}.  Here $E:=D(\phi_{\overline{f}}\mathbb{Q}_{X\times \Gr(\gamma_1^0,\gamma)}|_{X^{\Sp}\times\Gr(\gamma_1^0,\gamma)})$.  By definition, the function $\overline{f}|_{X\times \Gr(\gamma_1^0,\gamma)}$ is given by $f\boxplus 0$ on $X\times \Gr(\gamma_1^0,\gamma)$.  So by the Thom--Sebastiani isomorphism there is a chain of isomorphisms
\begin{align*}
\varphi_{\overline{f}}|_{X\times\Gr(\gamma'_1,\gamma_1)\times\Gr(\gamma'_2,\gamma_2)}\cong&\varphi_f\boxtimes \left(\mathbb{Q}_{\Gr(\gamma^0_1,\gamma)}|_{\Gr(\gamma'_1,\gamma_1)\times\Gr(\gamma'_2,\gamma_2)}\right)\\
\cong&\varphi_f\boxtimes\mathbb{Q}_{\Gr(\gamma'_1,\gamma_1)\times\Gr(\gamma'_2,\gamma_2)}\\
\cong&\varphi_{\overline{f}|_{X\times\Gr(\gamma'_1,\gamma_1)\times\Gr(\gamma'_2,\gamma_2)}}
\end{align*}
and so we deduce that the map
\[
\Ho_{c,\Gl_{\gamma_1}\times \Gl_{\gamma_2}}(X^{\Sp}\times \Gr(\gamma_1^0,\gamma),\varphi_{\overline{f}})^{\vee}_{\mathcal{S}_{\gamma_1,\gamma_2}}\xrightarrow{\res} \bigoplus_{(\gamma_1',\gamma_2',\gamma_3',\gamma_4')\in\DData}\Ho_{c,\Gl_{\gamma_1}\times \Gl_{\gamma_2}}(X^{\Sp}\times\Gr(\gamma'_1,\gamma_1)\times\Gr(\gamma'_2,\gamma_2),\varphi_{\overline{f}})^{\vee}_{\mathcal{S}_{\gamma_1,\gamma_2}}
\]
is also an isomorphism.  From the commutative diagram
\[
\xymatrix{
(X,\Gl_{\gamma'_1,\gamma'_3}\times \Gl_{\gamma'_2,\gamma'_4})_N\ar[rdd]_{p_{\gamma'_1,\gamma'_2,\gamma'_2,\gamma'_4,N}}\ar[r]^-{\cong} &(X\times \Gr(\gamma'_1,\gamma_1)\times\Gr(\gamma'_2,\gamma_2),\Gl_{\gamma_1}\times \Gl_{\gamma_2})_N\ar[d]_{q_{\gamma'_1,\gamma'_2,\gamma'_3,\gamma'_4,N}}\\
& (X\times \Gr(\gamma^0_1,\gamma),\Gl_{\gamma_1}\times \Gl_{\gamma_2})_N\ar[d]_{t'_N}\\
&(X,\Gl_{\gamma_1}\times \Gl_{\gamma_2})_N
}
\]
we deduce that 
\begin{align*}
&t'_{*,\mathcal{S}_{\gamma_1,\gamma_2}}|_{\res^{-1}(\Ho_{c,\Gl_{\gamma_1}\times \Gl_{\gamma_2}}(X^{\Sp}\times\Gr(\gamma'_1,\gamma_1)\times\Gr(\gamma'_2,\gamma_2)))}\cdot \eue(Q_0,\gamma'_4,\gamma'_1)\eue(Q_0,\gamma'_2,\gamma'_3)=\\&p_{\gamma'_1,\gamma'_2,\gamma'_3,\gamma'_4,*,\mathcal{S}_{\gamma_1,\gamma_2}}\res|_{\res^{-1}(\Ho_{c,\Gl_{\gamma_1}\times \Gl_{\gamma_2}}(X^{\Sp}\times\Gr(\gamma'_1,\gamma_1)\times\Gr(\gamma'_2,\gamma_2)))}
\end{align*}
since $\eue(Q_0,\gamma'_4,\gamma'_1)\eue(Q_0,\gamma'_3,\gamma'_2)$ is the Euler characteristic of the normal bundle to $\Gr(\gamma'_1,\gamma_1)\times\Gr(\gamma'_2,\gamma_2)$ inside $\Gr(\gamma^0_1,\gamma)$.  It follows that 
\begin{align*}
t'_{*,\mathcal{S}_{\gamma_1,\gamma_2}}\circ(\cdot \eue(Q_0,\gamma_2,\gamma_1))=&\sum_{(\gamma'_1,\ldots,\gamma'_4)\in\DData} t'_{*,\mathcal{S}_{\gamma_1,\gamma_2}}|_{\res^{-1}(\Ho_{c,\Gl_{\gamma_1}\times \Gl_{\gamma_2}}(X\times\Gr(\gamma'_1,\gamma_1)\times\Gr(\gamma'_2,\gamma_2)))}\circ \\&(\cdot\eue(Q_0,\gamma'_2,\gamma'_1)\eue(Q_0,\gamma'_4,\gamma'_1)\eue(Q_0,\gamma'_2,\gamma'_3)\eue(Q_0,\gamma'_4,\gamma'_3))
\\=&(-1)^{\gamma'_2\cdot\gamma'_3}\sum_{(\gamma'_1,\ldots,\gamma'_4)\in\DData}p_{\gamma'_1,\gamma'_2,\gamma'_3,\gamma'_4,*,\mathcal{S}_{\gamma_1,\gamma_2}}\circ (\cdot \eue(Q_0,\gamma'_2,\gamma'_1)\eue(Q_0,\gamma'_4,\gamma'_3)).
\end{align*}
Here we use the identity
\[
\eue(Q_0,\gamma'_2,\gamma'_3)=(-1)^{\gamma'_2\cdot\gamma'_3}\eue(Q_0,\gamma'_3,\gamma'_2).
\]
Since $p_{\gamma'_1,\gamma'_2,\gamma'_3,\gamma'_4}$ is the map inducing $d_{\nabla}$, and $q_{\gamma'_1,\gamma'_2,\gamma'_3,\gamma'_4,N}$ is the map inducing $b_{\nabla}$, the desired equality follows.
\end{proof}

To finish the proof of Theorem \ref{comultalg} we need to introduce some more notation.  Given $\nabla=(\gamma'_1,\gamma'_2,\gamma'_3,\gamma'_4)\in\DData_{\gamma_1,\gamma_2,\gamma_1^0,\gamma_2^0}$ and $D$ an array of dots in a four by four grid, we define $\RS^{\Sp}_{\nabla,D}\subset \RS^{\Sp}_{\gamma}$ to be the subspace containing those $\rho\in \RS^{\Sp}_{\gamma}$ such that for every $a\in Q_1$ and $c\in\{1,\ldots,4\}$,
\[
\rho\left(\mathbb{C}^{\gamma'_c(s(a))}\right)\subset\bigoplus_{\substack{c'\textrm{ such that }D\textrm{ has}\\ \textrm{a dot in position }(c',c)}}\mathbb{C}^{\gamma'_{c'}(t(a))}.
\]
So for example
\[
\RS^{\Sp}_{\nabla,\begin{picture}(12,14)\put(1,10){\circle*{2}}\put(4,7){\circle*{2}}\put(1,7){\circle*{2}}\put(7,4){\circle*{2}}\put(7,10){\circle*{2}}\put(10,7){\circle*{2}}\put(7,7){\circle*{2}}\put(10,1){\circle*{2}}\put(7,1){\circle*{2}}\end{picture}}
\]
is the subspace of $\RS^{\Sp}_{\gamma}$ of representations $\rho$ such that 
\begin{align*}
\rho (\Cp^{\gamma'_1})\subset &\Cp^{\gamma'_1}\oplus \Cp^{\gamma'_2}\\
\rho(\Cp^{\gamma'_2})\subset &  \Cp^{\gamma'_2}\\
\rho(\Cp^{\gamma'_4})\subset &\Cp^{\gamma'_2}\oplus \Cp^{\gamma'_4}.
\end{align*}
We denote by $W_{\nabla,D}$ the restriction of $\tr(W)$ to $\RS_{\nabla,D}$, and define
\[
\VV_{\nabla,D}=\Ho_{c,T_{\gamma}}(\RS^{\Sp}_{\nabla,D},W_{\nabla,D})^{\vee}.
\]
Depending on the array $D$, the space $\VV_{\nabla,D}$ carries an action of a group $G\subset \SG_{\gamma}$ where $G$ contains $\SG_{\gamma'_1}\times\SG_{\gamma'_2}\times\SG_{\gamma'_3}\times\SG_{\gamma'_4}$.  When restricting to the invariant part of the $G$-action, for $G=\SG_{\gamma''_1}\times\ldots\times\SG_{\gamma''_c}$, we abbreviate
\[
\VV_{\nabla,D}^{\SG_{\gamma''_1}\times\ldots\times\SG_{\gamma''_c}}
\]
to 
\[
\VV_{\nabla,D}^{\gamma''_1,\ldots,\gamma''_c}.
\]
Finally, if the element $\nabla$ is missing from the notation, but elements $\gamma'_1,\ldots,\gamma'_4$ appear in superscripts, we take a direct sum over all choices of $\nabla\in\Dec_{\gamma_1,\gamma_2,\gamma^0_1\gamma^0_2}$.

\begin{lemma}
\label{LemMod}
Consider the commutative diagram
\[
\xymatrix{
\VV_{\nabla,\begin{picture}(12,14)\put(1,10){\circle*{2}}\put(1,7){\circle*{2}}\put(4,7){\circle*{2}}\put(4,10){\circle*{2}}\put(7,1){\circle*{2}}\put(7,4){\circle*{2}}\put(7,7){\circle*{2}}\put(7,10){\circle*{2}}\put(10,1){\circle*{2}}\put(10,4){\circle*{2}}\put(10,7){\circle*{2}}\put(10,10){\circle*{2}}\end{picture}}^{\gamma'_1,\gamma'_2,\gamma'_3,\gamma'_4}
\ar[r]^a\ar[d]^b&
\VV_{\nabla,\begin{picture}(12,14)\put(1,10){\circle*{2}}\put(1,7){\circle*{2}}\put(4,7){\circle*{2}}\put(4,10){\circle*{2}}\put(4,4){\circle*{2}}\put(7,1){\circle*{2}}\put(7,4){\circle*{2}}\put(7,7){\circle*{2}}\put(7,10){\circle*{2}}\put(10,1){\circle*{2}}\put(10,4){\circle*{2}}\put(10,7){\circle*{2}}\put(10,10){\circle*{2}}\put(1,4){\circle*{2}}\put(1,1){\circle*{2}}\put(4,1){\circle*{2}}\end{picture}}^{\gamma'_1,\gamma'_2,\gamma'_3,\gamma'_4}\ar[d]^c
\\
\VV_{\nabla,\begin{picture}(12,14)\put(1,10){\circle*{2}}\put(4,7){\circle*{2}}\put(1,7){\circle*{2}}\put(7,4){\circle*{2}}\put(7,10){\circle*{2}}\put(10,1){\circle*{2}}\put(7,1){\circle*{2}}\put(10,7){\circle*{2}}\put(7,7){\circle*{2}}\end{picture}}^{\gamma'_1,\gamma'_2,\gamma'_3,\gamma'_4}
\ar[r]^d&
\VV_{\nabla,\begin{picture}(12,14)\put(1,10){\circle*{2}}\put(4,7){\circle*{2}}\put(1,7){\circle*{2}}\put(1,1){\circle*{2}}\put(7,4){\circle*{2}}\put(7,10){\circle*{2}}\put(10,1){\circle*{2}}\put(7,1){\circle*{2}}\put(10,7){\circle*{2}}\put(7,7){\circle*{2}}\put(1,4){\circle*{2}}\put(4,1){\circle*{2}}\end{picture}}^{\gamma'_1,\gamma'_2,\gamma'_3,\gamma'_4}
}
\]
where $a$ and $d$ are given by pushforward along inclusions and $b$ and $c$ are given by pullbacks along inclusions.  Then $c a=d b(\cdot \eue(Q_1,\gamma'_2,\gamma'_3))$.
\end{lemma}
\begin{proof}
We can extend the diagram as follows.
\[
\xymatrix{
&&\VV_{\nabla,\begin{picture}(12,14)\put(1,10){\circle*{2}}\put(1,7){\circle*{2}}\put(4,7){\circle*{2}}\put(4,10){\circle*{2}}\put(4,4){\circle*{2}}\put(7,1){\circle*{2}}\put(7,4){\circle*{2}}\put(7,7){\circle*{2}}\put(7,10){\circle*{2}}\put(10,1){\circle*{2}}\put(10,4){\circle*{2}}\put(10,7){\circle*{2}}\put(10,10){\circle*{2}}\put(1,4){\circle*{2}}\put(1,1){\circle*{2}}\put(4,1){\circle*{2}}\end{picture}}^{\gamma'_1,\gamma'_2,\gamma'_3,\gamma'_4}\ar@/^2pc/[ddl]^c\ar@/^1pc/[dl]_{i^*}
\\
\VV_{\nabla,\begin{picture}(12,14)\put(1,10){\circle*{2}}\put(1,7){\circle*{2}}\put(4,7){\circle*{2}}\put(4,10){\circle*{2}}\put(7,1){\circle*{2}}\put(7,4){\circle*{2}}\put(7,7){\circle*{2}}\put(7,10){\circle*{2}}\put(10,1){\circle*{2}}\put(10,4){\circle*{2}}\put(10,7){\circle*{2}}\put(10,10){\circle*{2}}\end{picture}}^{\gamma'_1,\gamma'_2,\gamma'_3,\gamma'_4}
\ar@/^2pc/[urr]^a\ar[d]^{b}\ar[r]^{a'}&
\VV_{\nabla,\begin{picture}(12,14)\put(1,10){\circle*{2}}\put(1,7){\circle*{2}}\put(1,4){\circle*{2}}\put(4,10){\circle*{2}}\put(4,7){\circle*{2}}\put(1,1){\circle*{2}}\put(4,1){\circle*{2}}\put(7,1){\circle*{2}}\put(7,4){\circle*{2}}\put(7,7){\circle*{2}}\put(7,10){\circle*{2}}\put(10,1){\circle*{2}}\put(10,4){\circle*{2}}\put(10,7){\circle*{2}}\put(10,10){\circle*{2}}\end{picture}}^{\gamma'_1,\gamma'_2,\gamma'_3,\gamma'_4}\ar[d]^{c'}\ar@/^1pc/[ur]^{i_*}
\\
\VV_{\nabla,\begin{picture}(12,14)\put(1,10){\circle*{2}}\put(4,7){\circle*{2}}\put(1,7){\circle*{2}}\put(7,4){\circle*{2}}\put(7,10){\circle*{2}}\put(10,1){\circle*{2}}\put(7,1){\circle*{2}}\put(10,7){\circle*{2}}\put(7,7){\circle*{2}}\end{picture}}^{\gamma'_1,\gamma'_2,\gamma'_3,\gamma'_4}
\ar[r]^d&
\VV_{\nabla,\begin{picture}(12,14)\put(1,10){\circle*{2}}\put(4,7){\circle*{2}}\put(1,7){\circle*{2}}\put(1,1){\circle*{2}}\put(7,4){\circle*{2}}\put(7,10){\circle*{2}}\put(10,1){\circle*{2}}\put(7,1){\circle*{2}}\put(10,7){\circle*{2}}\put(7,7){\circle*{2}}\put(1,4){\circle*{2}}\put(4,1){\circle*{2}}\end{picture}}^{\gamma'_1,\gamma'_2,\gamma'_3,\gamma'_4}
}
\]
where all maps are still either pullback maps or pushforward maps along the obvious inclusions, and the inner square arises from a transversal intersection of manifolds.  Then by Corollary \ref{ppCor},
\[
c' a'=d b,
\]
and by Proposition \ref{maneq}, 
\[
i^*i_*=\cdot\eue(Q_1,\gamma'_2,\gamma'_3).
\]
Since $a,i_*,a'$ are all pushforward maps, it follows that $a=i_*a'$.  Similarly we deduce that $c=c'i^*$, and the lemma follows.

\end{proof}
\begin{lemma}
Consider the diagram
\label{LemMod2}
\[
\xymatrix{
\VV_{\nabla,\begin{picture}(12,14)\put(1,10){\circle*{2}}\put(4,7){\circle*{2}}\put(1,7){\circle*{2}}\put(7,4){\circle*{2}}\put(10,1){\circle*{2}}\put(7,1){\circle*{2}}\end{picture}}^{\gamma'_1,\gamma'_2,\gamma'_3,\gamma'_4}
\ar[r]^-a\ar[d]^b&
\VV_{\nabla,\begin{picture}(12,14)\put(1,10){\circle*{2}}\put(4,7){\circle*{2}}\put(1,7){\circle*{2}}\put(7,4){\circle*{2}}\put(7,10){\circle*{2}}\put(10,7){\circle*{2}}\put(7,7){\circle*{2}}\put(10,1){\circle*{2}}\put(7,1){\circle*{2}}\end{picture}}^{\gamma'_1,\gamma'_2,\gamma'_3,\gamma'_4}\ar[d]^c
\\
\VV_{\nabla,\begin{picture}(12,14)\put(1,10){\circle*{2}}\put(4,7){\circle*{2}}\put(7,4){\circle*{2}}\put(10,1){\circle*{2}}\end{picture},(1,2),(3,4)}^{\gamma'_1,\gamma'_2,\gamma'_3,\gamma'_4}
\ar[r]^-d&
\VV_{\nabla,\begin{picture}(12,14)\put(1,10){\circle*{2}}\put(4,7){\circle*{2}}\put(7,4){\circle*{2}}\put(7,10){\circle*{2}}\put(10,7){\circle*{2}}\put(10,1){\circle*{2}}\end{picture},(1,2),(3,4),(2,3)}^{\gamma'_1,\gamma'_2,\gamma'_3,\gamma'_4}
}
\]
where all the maps are induced by pushforward or pullback along affine fibrations, so we have localised the targets of the vertical maps.  Then $ca=db\circ(\cdot \eue(Q_1,\gamma'_2,\gamma'_3)$
\end{lemma}
\begin{proof}
The proof of the lemma is the same as the proof of Lemma \ref{LemMod}, using the other half of Proposition \ref{maneq}, and using that in the following diagram of affine fibrations
\[
\xymatrix{
&&
\RS_{\nabla,\begin{picture}(12,14)\put(1,10){\circle*{2}}\put(4,7){\circle*{2}}\put(1,7){\circle*{2}}\put(7,4){\circle*{2}}\put(7,10){\circle*{2}}\put(10,7){\circle*{2}}\put(7,7){\circle*{2}}\put(10,1){\circle*{2}}\put(7,1){\circle*{2}}\end{picture}}\ar[dl]\ar[dll]\ar[ddl]
\\
\RS_{\nabla,\begin{picture}(12,14)\put(1,10){\circle*{2}}\put(4,7){\circle*{2}}\put(1,7){\circle*{2}}\put(7,4){\circle*{2}}\put(10,1){\circle*{2}}\put(7,1){\circle*{2}}\end{picture}}\ar[d]
&
\RS_{\nabla,\begin{picture}(12,14)\put(1,10){\circle*{2}}\put(4,7){\circle*{2}}\put(1,7){\circle*{2}}\put(7,4){\circle*{2}}\put(7,10){\circle*{2}}\put(10,7){\circle*{2}}\put(10,1){\circle*{2}}\put(7,1){\circle*{2}}\end{picture}}\ar[l]\ar[d]
\\
\RS_{\nabla,\begin{picture}(12,14)\put(1,10){\circle*{2}}\put(4,7){\circle*{2}}\put(7,4){\circle*{2}}\put(10,1){\circle*{2}}\end{picture}}
&
\RS_{\nabla,\begin{picture}(12,14)\put(1,10){\circle*{2}}\put(4,7){\circle*{2}}\put(7,4){\circle*{2}}\put(7,10){\circle*{2}}\put(10,7){\circle*{2}}\put(10,1){\circle*{2}}\end{picture}}\ar[l]
}
\]
the inner square is Cartesian, and the maps $a,b,c,d$ are given by pullbacks or pushforwards induced by the morphisms in the outer square.
\end{proof}
\begin{lemma}
\label{LemCatch}
The square formed by the perimeter of the diagram 
\[
\xymatrix{
\VV_{\begin{picture}(12,14)\put(1,10){\circle*{2}}\put(1,7){\circle*{2}}\put(4,7){\circle*{2}}\put(4,10){\circle*{2}}\put(4,4){\circle*{2}}\put(7,1){\circle*{2}}\put(7,4){\circle*{2}}\put(7,7){\circle*{2}}\put(7,10){\circle*{2}}\put(10,1){\circle*{2}}\put(10,4){\circle*{2}}\put(10,7){\circle*{2}}\put(10,10){\circle*{2}}\put(1,4){\circle*{2}}\put(1,1){\circle*{2}}\put(4,1){\circle*{2}}\end{picture}}^{\gamma^0_1,\gamma^0_2}\ar[r]\ar[d]\
&
\VV_{\begin{picture}(12,14)\put(1,10){\circle*{2}}\put(1,7){\circle*{2}}\put(4,7){\circle*{2}}\put(4,10){\circle*{2}}\put(4,4){\circle*{2}}\put(7,1){\circle*{2}}\put(7,4){\circle*{2}}\put(7,7){\circle*{2}}\put(7,10){\circle*{2}}\put(10,1){\circle*{2}}\put(10,4){\circle*{2}}\put(10,7){\circle*{2}}\put(10,10){\circle*{2}}\put(1,4){\circle*{2}}\put(1,1){\circle*{2}}\put(4,1){\circle*{2}}\end{picture}}^{\gamma}\ar[d]
\\
\VV_{\begin{picture}(12,14)\put(1,10){\circle*{2}}\put(1,7){\circle*{2}}\put(4,7){\circle*{2}}\put(4,10){\circle*{2}}\put(4,4){\circle*{2}}\put(7,1){\circle*{2}}\put(7,4){\circle*{2}}\put(7,7){\circle*{2}}\put(7,10){\circle*{2}}\put(10,1){\circle*{2}}\put(10,4){\circle*{2}}\put(10,7){\circle*{2}}\put(10,10){\circle*{2}}\put(1,4){\circle*{2}}\put(1,1){\circle*{2}}\put(4,1){\circle*{2}}\end{picture}}^{\gamma'_1,\gamma'_2,\gamma'_3,\gamma'_4}\ar[r]\ar[d]
&
\VV_{\begin{picture}(12,14)\put(1,10){\circle*{2}}\put(1,7){\circle*{2}}\put(4,7){\circle*{2}}\put(4,10){\circle*{2}}\put(4,4){\circle*{2}}\put(7,1){\circle*{2}}\put(7,4){\circle*{2}}\put(7,7){\circle*{2}}\put(7,10){\circle*{2}}\put(10,1){\circle*{2}}\put(10,4){\circle*{2}}\put(10,7){\circle*{2}}\put(10,10){\circle*{2}}\put(1,4){\circle*{2}}\put(1,1){\circle*{2}}\put(4,1){\circle*{2}}\end{picture}}^{\gamma_1,\gamma_2}\ar[d]
\\
\VV^{\gamma'_1,\gamma'_2,\gamma'_3,\gamma'_4}_{\begin{picture}(12,14)\put(1,10){\circle*{2}}\put(4,7){\circle*{2}}\put(1,7){\circle*{2}}\put(1,1){\circle*{2}}\put(7,4){\circle*{2}}\put(7,10){\circle*{2}}\put(10,1){\circle*{2}}\put(7,1){\circle*{2}}\put(10,7){\circle*{2}}\put(7,7){\circle*{2}}\put(1,4){\circle*{2}}\put(4,1){\circle*{2}}\end{picture}}\ar[r]
&
\VV^{\gamma_1,\gamma_2}_{\begin{picture}(12,14)\put(1,10){\circle*{2}}\put(4,7){\circle*{2}}\put(1,7){\circle*{2}}\put(1,1){\circle*{2}}\put(7,4){\circle*{2}}\put(7,10){\circle*{2}}\put(10,1){\circle*{2}}\put(7,1){\circle*{2}}\put(10,7){\circle*{2}}\put(7,7){\circle*{2}}\put(1,4){\circle*{2}}\put(4,1){\circle*{2}}\end{picture}}
}
\]
commutes.
\end{lemma}
\begin{proof}
The bottom of the two squares commutes by Corollary \ref{ppCor}, since the horizontal maps are given by pushforward along smooth submersions.  By Lemma \ref{Dround}, the top square commutes after localising with respect to $\mathcal{S}_{\gamma_1,\gamma_2}$, and so the lemma will follow from the claim that the localisation map
%Let $z_2=\sum_{i\in Q_0}\sum_{r=1}^{\gamma_2(i)}x_{i,r}^{(2)}\in\Ho_{\Gl_{\gamma_1}\times \Gl_{\gamma_2}}(\pt,\QQ)$.  Then we have $\Ho_{\Gl_{\gamma_1}\times \Gl_{\gamma_2}}(\pt,\QQ)=\mathcal{S}_{\nabla}\oplus \mathbb{Q}[z_2]$, and since the domain is a free $\mathbb{Q}[z_2]$-module, it follows that the localisation
\[
\VV^{\gamma_1,\gamma_2}_{\nabla,\begin{picture}(12,14)\put(1,10){\circle*{2}}\put(4,7){\circle*{2}}\put(1,7){\circle*{2}}\put(1,1){\circle*{2}}\put(7,4){\circle*{2}}\put(7,10){\circle*{2}}\put(10,1){\circle*{2}}\put(7,1){\circle*{2}}\put(10,7){\circle*{2}}\put(7,7){\circle*{2}}\put(1,4){\circle*{2}}\put(4,1){\circle*{2}}\end{picture}}\rightarrow \VV^{\gamma_1,\gamma_2}_{\nabla,\begin{picture}(12,14)\put(1,10){\circle*{2}}\put(4,7){\circle*{2}}\put(1,7){\circle*{2}}\put(1,1){\circle*{2}}\put(7,4){\circle*{2}}\put(7,10){\circle*{2}}\put(10,1){\circle*{2}}\put(7,1){\circle*{2}}\put(10,7){\circle*{2}}\put(7,7){\circle*{2}}\put(1,4){\circle*{2}}\put(4,1){\circle*{2}}\end{picture},\mathcal{S}_{\gamma_1,\gamma_2}}
\]
is an injection.  For this, it suffices to show that if 
\[
\alpha\in \VV^{\gamma_1,\gamma_2}_{\nabla,\begin{picture}(12,14)\put(1,10){\circle*{2}}\put(4,7){\circle*{2}}\put(1,7){\circle*{2}}\put(1,1){\circle*{2}}\put(7,4){\circle*{2}}\put(7,10){\circle*{2}}\put(10,1){\circle*{2}}\put(7,1){\circle*{2}}\put(10,7){\circle*{2}}\put(7,7){\circle*{2}}\put(1,4){\circle*{2}}\put(4,1){\circle*{2}}\end{picture}}
\]
and 
\[
w\in  \mathcal{S}_{\gamma_1,\gamma_2}, 
\]
then $w\cdot \alpha\neq 0$.  This follows from the isomorphism (not respecting cohomological shifts) $\VV^{\gamma_1,\gamma_2}_{\nabla,\begin{picture}(12,14)\put(1,10){\circle*{2}}\put(4,7){\circle*{2}}\put(1,7){\circle*{2}}\put(1,1){\circle*{2}}\put(7,4){\circle*{2}}\put(7,10){\circle*{2}}\put(10,1){\circle*{2}}\put(7,1){\circle*{2}}\put(10,7){\circle*{2}}\put(7,7){\circle*{2}}\put(1,4){\circle*{2}}\put(4,1){\circle*{2}}\end{picture}}\cong \mathcal{H}^{\Sp}_{Q,W,\gamma_1}\otimes\mathcal{H}^{\Sp}_{Q,W,\gamma_2}$ and the freeness of the $\mathbb{Q}[z]$-action on $\mathcal{H}^{\Sp}_{Q,W,\gamma_2}$, where $z$ is as in the definition of $\mathcal{S}_{\gamma_1,\gamma_2}$.
\end{proof}
We can summarise Lemmas \ref{Dround}, \ref{LemMod}, \ref{LemMod2} and \ref{LemCatch} in the following diagram
\begin{equation}
\label{bigDiag}
\xymatrix{
\VV_{\begin{picture}(12,14)\put(1,10){\circle*{2}}\put(1,7){\circle*{2}}\put(4,7){\circle*{2}}\put(4,10){\circle*{2}}\put(7,1){\circle*{2}}\put(7,4){\circle*{2}}\put(10,1){\circle*{2}}\put(10,4){\circle*{2}}\end{picture}}^{\gamma^0_1,\gamma^0_2}\ar[r]\ar[d]
&
\VV_{\begin{picture}(12,14)\put(1,10){\circle*{2}}\put(1,7){\circle*{2}}\put(4,7){\circle*{2}}\put(4,10){\circle*{2}}\put(7,1){\circle*{2}}\put(7,4){\circle*{2}}\put(7,7){\circle*{2}}\put(7,10){\circle*{2}}\put(10,1){\circle*{2}}\put(10,4){\circle*{2}}\put(10,7){\circle*{2}}\put(10,10){\circle*{2}}\end{picture}}^{\gamma^0_1,\gamma_2^0}\ar[r]\ar[d]
&
\VV_{\begin{picture}(12,14)\put(1,10){\circle*{2}}\put(1,7){\circle*{2}}\put(4,7){\circle*{2}}\put(4,10){\circle*{2}}\put(4,4){\circle*{2}}\put(7,1){\circle*{2}}\put(7,4){\circle*{2}}\put(7,7){\circle*{2}}\put(7,10){\circle*{2}}\put(10,1){\circle*{2}}\put(10,4){\circle*{2}}\put(10,7){\circle*{2}}\put(10,10){\circle*{2}}\put(1,4){\circle*{2}}\put(1,1){\circle*{2}}\put(4,1){\circle*{2}}\end{picture}}^{\gamma^0_1,\gamma^0_2}\ar[r]\ar[d]\save[]+<1.5cm,-.9cm>*\txt{Lemma \ref{Dround}}\restore
&
\VV_{\begin{picture}(12,14)\put(1,10){\circle*{2}}\put(1,7){\circle*{2}}\put(4,7){\circle*{2}}\put(4,10){\circle*{2}}\put(4,4){\circle*{2}}\put(7,1){\circle*{2}}\put(7,4){\circle*{2}}\put(7,7){\circle*{2}}\put(7,10){\circle*{2}}\put(10,1){\circle*{2}}\put(10,4){\circle*{2}}\put(10,7){\circle*{2}}\put(10,10){\circle*{2}}\put(1,4){\circle*{2}}\put(1,1){\circle*{2}}\put(4,1){\circle*{2}}\end{picture}}^{\gamma}\ar[d]
\\
\VV_{\begin{picture}(12,14)\put(1,10){\circle*{2}}\put(1,7){\circle*{2}}\put(4,7){\circle*{2}}\put(4,10){\circle*{2}}\put(7,1){\circle*{2}}\put(7,4){\circle*{2}}\put(10,1){\circle*{2}}\put(10,4){\circle*{2}}\end{picture}}^{\gamma'_1,\gamma'_2,\gamma'_3,\gamma'_4}\ar[r]\ar[d]
&
\VV_{\begin{picture}(12,14)\put(1,10){\circle*{2}}\put(1,7){\circle*{2}}\put(4,7){\circle*{2}}\put(4,10){\circle*{2}}\put(7,1){\circle*{2}}\put(7,4){\circle*{2}}\put(7,7){\circle*{2}}\put(7,10){\circle*{2}}\put(10,1){\circle*{2}}\put(10,4){\circle*{2}}\put(10,7){\circle*{2}}\put(10,10){\circle*{2}}\end{picture}}^{\gamma'_1,\gamma'_2,\gamma'_3,\gamma'_4}\ar[r]\ar[d]\save[]+<1.7cm,-.9cm>*\txt{Lemma \ref{LemMod}}\restore
&
\VV_{\begin{picture}(12,14)\put(1,10){\circle*{2}}\put(1,7){\circle*{2}}\put(4,7){\circle*{2}}\put(4,10){\circle*{2}}\put(4,4){\circle*{2}}\put(7,1){\circle*{2}}\put(7,4){\circle*{2}}\put(7,7){\circle*{2}}\put(7,10){\circle*{2}}\put(10,1){\circle*{2}}\put(10,4){\circle*{2}}\put(10,7){\circle*{2}}\put(10,10){\circle*{2}}\put(1,4){\circle*{2}}\put(1,1){\circle*{2}}\put(4,1){\circle*{2}}\end{picture}}^{\gamma'_1,\gamma'_2,\gamma'_3,\gamma'_4}\save[]+<1.5cm,-.9cm>*\txt{Lemma \ref{LemCatch}}\restore\ar[r]\ar[d]
&
\VV_{\begin{picture}(12,14)\put(1,10){\circle*{2}}\put(1,7){\circle*{2}}\put(4,7){\circle*{2}}\put(4,10){\circle*{2}}\put(4,4){\circle*{2}}\put(7,1){\circle*{2}}\put(7,4){\circle*{2}}\put(7,7){\circle*{2}}\put(7,10){\circle*{2}}\put(10,1){\circle*{2}}\put(10,4){\circle*{2}}\put(10,7){\circle*{2}}\put(10,10){\circle*{2}}\put(1,4){\circle*{2}}\put(1,1){\circle*{2}}\put(4,1){\circle*{2}}\end{picture}}^{\gamma_1,\gamma_2}\ar[d]
\\
\VV_{\begin{picture}(12,14)\put(1,10){\circle*{2}}\put(1,7){\circle*{2}}\put(4,7){\circle*{2}}\put(7,1){\circle*{2}}\put(7,4){\circle*{2}}\put(10,1){\circle*{2}}\end{picture}}^{\gamma'_1,\gamma'_2,\gamma'_3,\gamma'_4}\ar[r]\ar[d]\save[]+<1.7cm,-.9cm>*\txt{Lemma \ref{LemMod2}}\restore
&
\VV_{\begin{picture}(12,14)\put(1,10){\circle*{2}}\put(1,7){\circle*{2}}\put(4,7){\circle*{2}}\put(7,1){\circle*{2}}\put(7,4){\circle*{2}}\put(7,7){\circle*{2}}\put(7,10){\circle*{2}}\put(10,1){\circle*{2}}\put(10,7){\circle*{2}}\end{picture}}^{\gamma'_1,\gamma'_2,\gamma'_3,\gamma'_4}\ar[r]\ar[d]
&
\VV^{\gamma'_1,\gamma'_2,\gamma'_3,\gamma'_4}_{\begin{picture}(12,14)\put(1,10){\circle*{2}}\put(4,7){\circle*{2}}\put(1,7){\circle*{2}}\put(1,1){\circle*{2}}\put(7,4){\circle*{2}}\put(7,10){\circle*{2}}\put(10,1){\circle*{2}}\put(7,1){\circle*{2}}\put(10,7){\circle*{2}}\put(7,7){\circle*{2}}\put(1,4){\circle*{2}}\put(4,1){\circle*{2}}\end{picture}}\ar[r]\ar[d]
&
\VV^{\gamma_1,\gamma_2}_{\begin{picture}(12,14)\put(1,10){\circle*{2}}\put(4,7){\circle*{2}}\put(1,7){\circle*{2}}\put(1,1){\circle*{2}}\put(7,4){\circle*{2}}\put(7,10){\circle*{2}}\put(10,1){\circle*{2}}\put(7,1){\circle*{2}}\put(10,7){\circle*{2}}\put(7,7){\circle*{2}}\put(1,4){\circle*{2}}\put(4,1){\circle*{2}}\end{picture}}\ar[d]
\\
\VV^{\gamma'_1,\gamma'_2,\gamma'_3,\gamma'_4}_{\begin{picture}(12,14)\put(1,10){\circle*{2}}\put(4,7){\circle*{2}}\put(7,4){\circle*{2}}\put(10,1){\circle*{2}}\end{picture},(1,2),(3,4)}\ar[r]
&
\VV^{\gamma'_1,\gamma'_2,\gamma'_3,\gamma'_4}_{\begin{picture}(12,14)\put(1,10){\circle*{2}}\put(4,7){\circle*{2}}\put(7,4){\circle*{2}}\put(7,10){\circle*{2}}\put(10,7){\circle*{2}}\put(10,1){\circle*{2}}\end{picture},(1,2),(3,4),(2,3)}\ar[r]
&
\VV^{\gamma'_1,\gamma'_2,\gamma'_3,\gamma'_4}_{\begin{picture}(12,14)\put(1,10){\circle*{2}}\put(4,7){\circle*{2}}\put(7,4){\circle*{2}}\put(7,10){\circle*{2}}\put(10,1){\circle*{2}}\put(10,7){\circle*{2}}\put(1,4){\circle*{2}}\put(4,1){\circle*{2}}\end{picture},\substack{(1,2),(3,4),\\(2,3),(1,4)}}\ar[r]
&
\VV^{\gamma_1,\gamma_2}_{\begin{picture}(12,14)\put(1,10){\circle*{2}}\put(4,7){\circle*{2}}\put(7,4){\circle*{2}}\put(7,10){\circle*{2}}\put(10,1){\circle*{2}}\put(10,7){\circle*{2}}\put(1,4){\circle*{2}}\put(4,1){\circle*{2}}\end{picture},(1,2)}.
}
\end{equation}
\begin{lemma}
All of the small squares in diagram (\ref{bigDiag}) that are not marked with lemmas commute.
\end{lemma}
\begin{proof}
Each of the squares is either entirely composed of pullback maps or pushforward maps along a commutative diagram, in which case commutativity is trivial, or else commutativity is an application of Proposition \ref{mixingprop}.
\end{proof}

We now have all of the components in place to prove Theorem \ref{comultalg}.  If we let $e_{\nabla}$ be the composition of the leftmost vertical maps, for a fixed summand corresponding to $\nabla\in\DData$, then by the definition of $\Delta$ and Corollary \ref{simplerComult}, we have the identity
\[
\Delta_{\gamma'_1,\gamma'_2}\otimes\Delta_{\gamma'_3,\gamma'_4}=(\cdot \eue(Q_0,\gamma'_2,\gamma'_1)\eue(Q_0,\gamma'_4,\gamma'_3))\circ e_{\nabla}
\]
and similarly, if we let $f$ be the composition of the rightmost vertical maps in (\ref{bigDiag}) we have
\[
\Delta_{\gamma_1,\gamma_2}=(\cdot\eue(Q_0,\gamma_2,\gamma_1))\circ f.
\]
On the other hand, by Lemma \ref{LemCatch}, up to the sign $(-1)^{\gamma'_2\cdot\gamma'_3}$ these are exactly the factors introduced by the noncommutativity of the top two squares in the rightmost column of (\ref{bigDiag}).  The remaining failure of the diagram to commute is described by Lemmas \ref{LemMod} and \ref{LemMod2}, so that the overall noncommutativity of the diagram is given by the factor $\tilde{\eue}_{\gamma'_2,\gamma'_3}$, as defined in (\ref{corr_term}) in the definition of the symmetrising morphism $\tilde{\sw}_{\gamma'_2,\gamma'_3}$ and so the compatibility condition of diagram (\ref{commdiagB}) is satisfied.

The injectivity of localisation condition, and the crosslinearity condition, contained in the definition of a $Q$-localised bialgebra, are the content of Propositions \ref{ABprop} and \ref{coCross} respectively, and we are done.
\end{proof}
\section{Examples}
\subsection{The Jordan quiver with potential}
Let $Q_{\Jor}$ be the quiver with one vertex and one loop.  We label this loop $X$.  Let $d\in\mathbb{Z}_{\geq 1}$, we set $W_d=X^{d+1}$.  The motivic Donaldson--Thomas theory of the pair $(Q_{\Jor},W_d)$ was studied in the paper \cite{DM11}, and the motivic Donaldson--Thomas invariants were calculated there: we have
\[
\Omega_{Q_{\Jor},W_d,n}=\begin{cases} [\Ho(\mathbb{A}^1,\varphi_{x^{d+1}})\state{1}]&\textrm{if }n=1\\0&\textrm{otherwise.}\end{cases}
\]
By Theorem \ref{comultalg}, the algebra $\mathcal{H}_{Q_{\Jor},W_d}$ carries a localised coproduct
\[
\Delta\colon \mathcal{H}_{Q_{\Jor},W_d}\rightarrow\mathcal{H}_{Q_{\Jor},W_d}\tilde{\boxtimes}^{\sw}_{+}\mathcal{H}_{Q_{\Jor},W_d}.
\]
However in this special case things are a little simpler than in general.  In fact the map $\Delta$ lifts to a map
\[
\Delta\colon \mathcal{H}_{Q_{\Jor},W_d}\rightarrow\mathcal{H}_{Q_{\Jor},W_d}\otimes\mathcal{H}_{Q_{\Jor},W_d},
\]
as the division factor in the definition of $\overleftarrow{\alpha}$ is exactly the factor by which we multiply $\beta$ in the definition of $\Delta$.  In addition, the multiplication factor in the symmetrising morphism $\tilde{\sw}$ is equal to one, i.e. we consider the target of $\Delta$ in the derived category of the usual symmetric monoidal category $\MMHS_{\mathbb{Z}}$.  In other words, $\mathcal{H}_{Q_{\Jor},W_d}$ is a Hopf algebra (the construction of a unique antipode is formal, using that $\mathcal{H}_{Q_{\Jor},W_d}$ is a connected algebra).  Let $\mathfrak{g}\subset\mathcal{H}_{Q_{\Jor},W_d}$ be the subspace of primitive elements, then by general theory $\mathfrak{g}$ is a Lie algebra, with Lie bracket given by the commutator of the multiplication in $\mathcal{H}_{Q_{\Jor},W_d}$, and the natural map
\[
\iota\colon \mathcal{U}(\mathfrak{g})\rightarrow\mathcal{H}_{Q_{\Jor},W_d}
\] 
is an injection.
\begin{proposition}
The following are equivalent:
\begin{enumerate}
\item
The element $\mathcal{H}_{Q_{\Jor},W_d}\in\Ob(\D{\MMHS_{\mathbb{Z}}})$ is pure, in the sense that the $i$th cohomologically graded piece is pure of weight $i$.
\item
The Lie algebra $\mathfrak{g}$ is concentrated entirely in degree $1$, with respect to the $\mathbb{Z}^{Q_0}$-grading, and $\iota$ is an isomorphism $\Sym(\mathcal{H}_{Q_{\Jor},W,1})\rightarrow\mathcal{H}_{Q_{\Jor},W_d}$.
\end{enumerate}
\end{proposition}
\begin{proof}
\begin{itemize}
\item[$1\rightarrow 2$:] If $\mathcal{H}_{Q_{\Jor},W_d}$ is pure, it follows that the sub-object $\mathfrak{g}$ is too, and hence so is $\mathcal{U}(\mathfrak{g})$.  Clearly the whole of $\mathcal{H}_{Q_{\Jor},W_d,1}$ is primitive, i.e. $\mathcal{H}_{Q_{\Jor},W_d,1}\subset \mathfrak{g}$.  By the main result of \cite{DM11}, the free supercommutative algebra $A\subset \mathcal{U}(\mathfrak{g})$ generated by $\mathcal{H}_{Q_{\Jor},W_d,1}$ has the same class in $\KK(\Dlb{\MMHS_{\mathbb{Z}}})$ as the target of $\iota$, and so it follows from purity that $\iota$ is an isomorphism when restricted to $A$, and also $A=\mathcal{U}(\mathfrak{g})$ since $\iota$ is injective.  Since $\mathfrak{g}$ is concentrated entirely in $\mathbb{Z}^{Q_0}$-degree one and the Lie bracket respects the $\mathbb{Z}^{Q_0}$-grading, the Lie bracket is zero, and so $\iota$ becomes the isomorphism $\Sym(\mathcal{H}_{Q_{\Jor},W_d,1})\rightarrow\mathcal{H}_{Q_{\Jor},W_d}$.
\item[$2\rightarrow 1$:]  The monodromic mixed Hodge structure on 
\begin{align*}
\mathcal{H}_{Q_{\Jor},W_d,1}=&\Ho_{c,\mathbb{C}^*}(\mathbb{A}^1,\varphi_{x^{d+1}})^{\vee}\\
\cong&\Ho_{c}(\mathbb{A}^1,\varphi_{x^{d+1}})^{\vee}\otimes\Ho_{\mathbb{C}^*}(\pt,\QQ)
\end{align*} is pure, and so the symmetric algebra generated by it is pure too.
\end{itemize}
\end{proof}
Purity of $\mathcal{H}_{Q_{\Jor},W_d}$ is quite easy to prove, using the representation theory of $\mathcal{H}^{\Sp}_{Q,W}$, for general triples $(Q,W,\Sp)$ --- this representation theory is worked out in some detail in \cite{DaMe15b}.  For our special choice of quiver with potential, the situation is easier to describe.  For $n,f\in\mathbb{N}$, define $V_{n,f}=\Hom(\mathbb{C}^n,\mathbb{C}^n)\times\Hom(\mathbb{C}^n,\mathbb{C}^f)$.  We think of $V_{n,f}$ as a space of representations of the coframed Jordan quiver $Q^{\mathrm{cfr}}_{\Jor}$, of dimension vector $(n,f)$.  The space $V_{n,f}$ carries the action of $\Gl_{n}$ via change of basis, as always.  Here we are using the notation of Equations (\ref{RSdef}) and (\ref{Gdef}), with $n$ considered as a dimension vector since our quiver has only one vertex.  Let $V_{n,f}^{\st}\subset V_{n,f}$ be the subset of representations $\rho$ of $Q^{\textrm{cfr}}_{\Jor}$ such that the there is no nonzero subrepresentation $\rho'\subset\rho$ supported on the original quiver $Q_{\Jor}$.  We have a chain of $\Gl_n$-equivariant inclusions
\[
\RS_n\times\Fr(n,f)\subset V^{\st}_{n,f}\subset V_{n,f}.
\]
As we let $f$ become very large, the codimension of $\left(V^{\st}_{n,f}/\Gl_n\setminus (\RS_n\times\Fr(n,f))/\Gl_n\right)$ becomes very large too, and it follows from the fact that $\varphi_{\tr(W)}[fn]$ is a perverse sheaf on $V^{\st}_{n,f}/\Gl_n$, with cohomology in the constructible t structure that is concentrated in an interval that is bounded independently of $f$, that the map in compactly supported cohomology 
\[
\Ho^{i+2fn}_{c}( (\RS_n,\Gl_n)_f,\varphi_{\tr(W_d)_f})\rightarrow \Ho^{i+2fn}_c(V^{\st}_{n,f}/\Gl_n,\varphi_{\tr(W_d)_f})
\]
is an isomorphism for fixed $i$ and large $f$, where $\tr(W_d)_f$ is the function on $V_{n,f}^{\st}/\Gl_n$ induced by $\tr(W_d)$.  By definition (see Equation (\ref{eqvcdef})), we deduce
\begin{equation}
\label{feqdef}
\Ho_{c,\Gl_n}(\RS_n,\varphi_{\tr(W_d)}):=\varinjlim\left(\Ho_c(V_{n,f}^{\st}/\Gl_n,\varphi_{\tr(W_d)_f})\state{fn}\right)
\end{equation}
and it is enough to demonstrate the purity of the mixed Hodge structures on the right hand side of (\ref{feqdef}).  Now the GIT quotient map $p_f\colon V^{\st}_{n,f}/\Gl_n\rightarrow \mathbb{A}^n$ to the coarse moduli space of $Q^{\mathrm{cfr}}$-representations is proper, and so there is a natural isomorphism
\[
p_{f,*}\phi_{\tr(W_d)_f}\mathbb{Q}_{V^{\st}_{n,f}/\Gl_n}\cong\phi_{g}p_{f,*}\mathbb{Q}_{V^{\st}_{n,f}/\Gl_n}
\]
where $g\colon \mathbb{A}^n\rightarrow\mathbb{A}^1$ is the map induced by $\tr(W_d)$.  In addition, the complex of mixed Hodge modules $p_{f,*}\mathbb{Q}_{V^{\st}_{n,f}/\Gl_n}$ is pure by Saito's version of the decomposition theorem \cite{Sa88}, in the sense that $\mathcal{H}^i(p_{f,*}\mathbb{Q}_{V^{\st}_{n,f}/\Gl_n})$ is pure of weight $i$.  Furthermore the support of $\phi_gp_{f,*}\mathbb{Q}_{V^{\st}_{n,f}/\Gl_n}$ is the origin $0\in\mathbb{A}^1$; this follows from the description of the support of $\phi_{\tr(W_d)_f}\mathbb{Q}_{V^{\st}_{n,f}/\Gl_n}$ in terms of the Jacobi algebra for the pair $(Q_{\Jor},W_d)$ (see Remark \ref{JacRem}) since all representations $\rho$ of the Jacobi algebra send $X$ to a nilpotent endomorphism, and by \cite[Thm.1]{LBP90} the functions on the affinization $\mathbb{A}^n$ are given by linear combinations of the maps $\rho\mapsto \tr(\rho(X)^i)$ for $i\in\mathbb{N}$.  In particular the support is proper, and it follows as in \cite[Cor.3.2]{DMSS12}, after passing to the intermediate extension of $p_{f,*}\mathbb{Q}_{V^{\st}_{n,f}/\Gl_n}$ to a $\mathbb{C}^*$-equivariant relative compactification of the morphism $\mathbb{A}^n\xrightarrow{g}\mathbb{A}^1$, that $\Ho_c(V^{\st}_{n,f}/\Gl_n,\varphi_{\tr(W)_f})\state{{fn}}$ is pure as a monodromic mixed Hodge structure.  We have proved the following result, which one may already find without proof in \cite[Sec.2.4]{So14}.
\begin{theorem}
The algebra $\mathcal{H}_{Q_{\Jor},W_d}$ is supercommutative, and there is an isomorphism of algebras
\begin{align*}
\mathcal{H}_{Q_{\Jor},W_d}\cong&\bigwedge \left(\Ho(\mathbb{A}^1,\varphi_{x^{d+1}})\otimes\Ho_{\mathbb{C}^*}(\pt,\QQ)\right)
\\
\cong& \bigwedge_{r=1}^d\bigwedge[u_{r,1},u_{r,2},\ldots],
\end{align*}
where $u_{r,n}$ is of cohomological degree $2n-1$.
\end{theorem}

\subsection{Twisted and untwisted character varieties}
\label{charSec}
Let $\Sigma_g$ denote a genus $g$ topological Riemann surface.  We define
\[
\Rep_m(\Sigma_g):=\{A_1,\ldots,A_g,B_1,\ldots,B_g\in\GL_{\Cp}(m)|\prod_{i=1}^g (A_i,B_i)=\id_{m\times m}\},
\]
where $(A_i,B_i)=A_iB_iA_i^{-1}B_i^{-1}$.  This variety is a space of homomorphisms $\pi_1(\Sigma_g)\rightarrow \GL_{\Cp}(m)$, or representations of $\pi_1(\Sigma_g)$, and the stack theoretic quotient $[\Rep_m(\Sigma_g)/\GL_{\Cp}(m)]$ is the stack of $m$ dimensional $\pi_1(\Sigma_g)$ representations.  Let $\zeta_m$ be a primitive $m$th root of unity.  We will also consider the twisted counterpart of this variety
\[
\Rep^{\zeta_m}_m(\Sigma_g):=\{A_1,\ldots,A_g,B_1,\ldots,B_g\in\GL_{\Cp}(m)|\prod(A_i,B_i)=\zeta_m\id_{m\times m}\}.
\]
By \cite[Cor.2.2.4]{HLRV13}, up to isomorphisms in cohomology lifting to isomorphisms of Hodge structure, it does not matter which primitive $m$th root of unity we pick.  We will build a CoHA $\mathcal{H}^{\Sp}_{Q_{\Sigma_g},W_{\Sigma_g}}$ for a quiver $Q_{\Sigma_g}$ with four vertices such that $\mathcal{H}^{\Sp}_{Q_{\Sigma_g},W_{\Sigma_g},\gamma}=0$ unless $\gamma$ is a constant dimension vector, and such that there are natural isomorphisms
\[
\mathcal{H}^{\Sp}_{Q_{\Sigma_g},W_{\Sigma_g},(m,m,m,m)}\cong \Ho_{c,\GL_{\Cp}(m)}(\Rep_m(\Sigma_g),\mathbb{Q})^{\vee}\state{(1-g)m^2},
\]
and conjectural isomorphisms
\[
\mathcal{H}^{\Sp}_{Q_{\Sigma_g},W_{\Sigma_g}}\cong\Sym\left(\bigoplus_{m\geq 1}\Ho_{c,\GL_{\Cp}(m)}(\Rep^{\zeta_m}_m(\Sigma_g),\mathbb{Q})^{\vee}\state{(1-g)m^2}\right)
\]
at the level of $\mathbb{N}$-graded mixed Hodge structures, where $\mathbb{N}$ keeps track of the dimension of the $\mathbb{C}[\pi_1(\Sigma_g)]$-reprsentation.  By \cite[Cor.2.2.7]{HLRV13} $\Rep^{\zeta_m}_m(\Sigma_g)$ is acted on freely by $\PGL_{\Cp}(m)$, so there is an isomorphism 
\[
\Ho_{c,\GL_{\Cp}(m)}(\Rep^{\zeta_m}_m(\Sigma_g),\mathbb{Q})^{\vee}\state{(1-g)m^2}\cong\Ho_c(\Rep^{\zeta_m}_m(\Sigma_g)/\PGL_{\Cp}(m),\mathbb{Q})^{\vee}\state{(1-g)m^2}[u]
\]
where $u$ is the degree 2 generator of $\Ho_{\Cp^*}(\pt,\mathbb{Q})$.  In the language of quantum enveloping algebras, we conjecture that there are isomorphisms
\[
\mathfrak{g}_{\prim,(m,m,m,m)}\cong\Ho_c(\Rep^{\zeta_m}_m(\Sigma_g)/\PGL_{\Cp}(m),\mathbb{Q})^{\vee}\state{(1-g)m^2}
\]
between the space of primitive generators of a PBW basis for a cocommutative deformation of $\mathcal{H}_{Q_{\Sigma_g},W_{\Sigma_g}}$ and the dual compactly supported cohomology of the twisted character varieties for $\Sigma_g$, where the primitive generators are defined to be those that generate the other generators under the action of multiplication by $u$.  Since by \cite[Thm.2.2.5]{HLRV13} the twisted character varieties are smooth, we may alternatively restate this as an isomorphism between the primitive generators of $\mathcal{H}^{\Sp}_{Q_{\Sigma_g},W_{\Sigma_g}}$ and the shifted cohomology of the twisted character varieties.

\begin{example}
For ease of exposition we consider only the Riemann surface of genus 2, but everything generalises in a way that is hopefully obvious.  We break the surface $\Sigma_2$ into 4 tiles.  The front two tiles are as drawn in black in Figure \ref{g2front}, they are formed from the upper half of the figure glued to the bottom half along the dashed line. 
\begin{figure}
\caption{Tiling of $\Sigma_2$ seen from the front}
\includegraphics{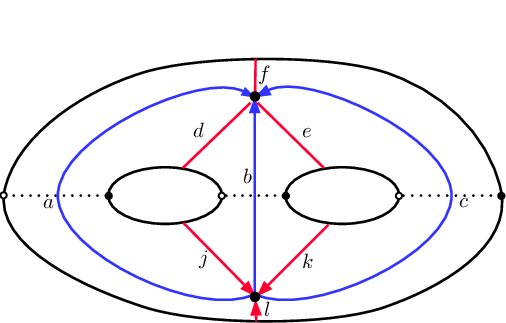}
\label{g2front}
\end{figure}

\begin{figure}
\caption{Tiling of $\Sigma_2$ seen from the back}
\includegraphics{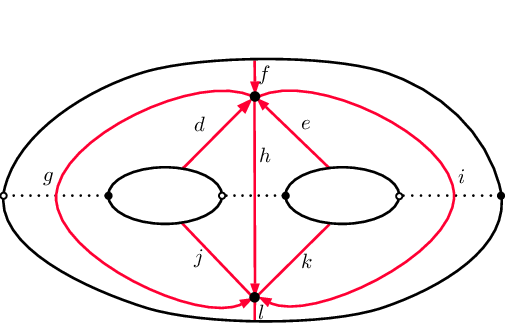}
\label{g2back}
\end{figure}
The back two tiles are as drawn in Figure \ref{g2back}, they are again formed from the top half of the figure glued to the bottom half along the dashed line.  The back two tiles are glued to the front two by identifying the solid black lines of Figure \ref{g2front} with the solid black lines of Figure \ref{g2back}.  We have drawn, in red and blue, the dual quiver to this tiling, this will be our quiver $Q_{\Sigma_2}$, which we reproduce below:
\[
\xymatrix{\\
\bullet\ar@/_1.5pc/@[blue][rrr]_a\ar@[blue][rrr]_b\ar@/^1.5pc/@[blue][rrr]_c &&&\bullet \ar@/_1.5pc/@[red][ddd]^d\ar@/^1.5pc/@[red][ddd]^f \ar@[red][ddd]^e
\\ \\
\\\bullet \ar@/_1.5pc/@[red][uuu]^l\ar@/^1.5pc/@[red][uuu]^j \ar@[red][uuu]^k&&&\bullet \ar@/_1.5pc/@[red][lll]^g\ar@[red][lll]^h\ar@/^1.5pc/@[red][lll]^i
}
\]
We consider the following element of $\mathbb{C}Q_{\Sigma_2}/[\mathbb{C}Q_{\Sigma_2},\mathbb{C}Q_{\Sigma_2}]$:
\[
W_{\Sigma_2}:=lgfa-jgda+jhdb-kheb+kiec-lifc.
\]
The recipe for this quiver with potential is as follows --- in the literature it is called the QP associated to a brane tiling of a surface, see e.g. \cite{longout} for a detailed reference, \cite{Rhombi} for the Physics background, or \cite{MR} or \cite{Dav08} for the Mathematics background.  To a tiling $\Delta$ of a Riemann surface, the 1-skeleton of which is given the structure of a bipartite graph, the associated quiver is just the dual quiver, as above, oriented so that the arrows go clockwise around the black vertices.  The potential is given by taking the alternating sum
\[
W_{\Delta}:=\sum_{v\in\Delta_0|v\text{ is white}}l_v-\sum_{v\in\Delta_0|v\text{ is black}}l_v,
\]
where $l_v$ is the shortest cycle going around the vertex $v$ in the dual quiver.

Returning to our special case, we define $\RS^{\Sp}_{Q_{\Sigma_2},\gamma}\subset \RS_{Q_{\Sigma_2},\gamma}$ by the condition that every red arrow is sent to an isomorphism, and since an upper block triangular matrix is invertible if and only if its diagonal blocks are, we deduce that these $\RS^{\Sp}_{Q_{\Sigma_2},\gamma}$ satisfy Assumption \ref{closed_under}.

The QP $(W_{\Sigma_2},W_{\Sigma_2})$ admits a cut in the sense of Section \ref{2dCOHA}, given by setting $S=\{a,b,c\}$, and the moduli spaces $\RS^{\Sp}_{Q_{\Sigma_2},\gamma}$ satisfy Assumption \ref{ass1}, so that we have an isomorphism in cohomology
\[
\Ho_{c,\Gl_{\gamma}}(\RS^{\Sp}_{Q_{\Sigma_2},\gamma},\varphi_{\tr(W_{\Sigma_2})_{\gamma}})\cong\Ho_{c,\Gl_{\gamma}}(\overline{Z}_{\gamma},\mathbb{Q})
\]
where $\overline{Z}_{\gamma}$ is the space of representations of $Q_{\Sigma_2}$ such that all red arrows are sent to isomorphisms, and the relations
\begin{align}
\label{g2rels}
\partial W_{\Sigma_2}/\partial a=lgf-jgd=0\\
\nonumber
\partial W_{\Sigma_2}/\partial b=jhd-khe=0\\
\nonumber
\partial W_{\Sigma_2}/\partial c=kie-lif=0
\end{align}
are satisfied.  Let $Z_{\gamma}$ be the space of representations of the quiver $Q'$, obtained by deleting arrows $a$, $b$ and $c$, still satisfying the relations (\ref{g2rels}).  Up to gauge transformation we may assume that $d$, $g$ and $j$ are all the identity matrix, and consider $[Z_{\gamma}/\Gl_{\gamma}]$ as the stack of representations of the 6 loop quiver algebra with loops labelled by $e$, $f$, $h$, $i$, $k$ and $l$, satisfying the relations
\begin{align*}
fl=1\\
h=khe\\
lif=kie
\end{align*}
such that all arrows are sent to invertible matrices.  Substituting $l=f^{-1}$ and $k=he^{-1}h^{-1}$, we deduce that $[Z_{\gamma}/\Gl_{\gamma}]$ is isomorphic to the stack of $m$-dimensional representations of the 4 loop quiver, with loops labelled $f,h,e,i$, satisfying the one relation
\[
he^{-1}h^{-1}ie=f^{-1}if,
\]
which becomes
\[
he^{-1}h^{-1}e=i^{-1}f^{-1}if
\]
after the substitution $h\mapsto ih$.  In other words, $[Z_{\gamma}/\Gl_{\gamma}]$ is isomorphic to the stack $[\Rep_m(\Sigma_2)/\GL_{\Cp}(m)]$ of $m$-dimensional representations of $\pi_1(\Sigma_2)$.

Set $\gamma=(m,m,m,m)$.  Via the affine fibration $\overline{Z}_{\gamma}\rightarrow Z_{\gamma}$ we obtain, by Theorem \ref{eqred}, an isomorphism of mixed Hodge structures
\[
\Ho_{c,\Gl_{\gamma}}(\RS_{Q_{\Sigma_2},W_{\Sigma_2}}^{\Sp},\varphi_{\tr(W_{\Sigma_2})_{\gamma}})\cong \Ho_{c,\GL_{\Cp}(m)}(\Rep_m(\Sigma_2),\mathbb{Q})\state{{-3m^2}}.
\]
Since $\chi((m,m,m,m),(m,m,m,m))=-8m^2$ we deduce that
\[
\mathcal{H}^{\Sp}_{Q_{\Sigma_2},W_{\Sigma_2},(m,m,m,m)}\cong \Ho_{c,\GL_{\Cp}(m)}(\Rep_m(\Sigma_2),\mathbb{Q})^{\vee}\state{{-m^2}}.
\]
\end{example}
\bigbreak
Generalising the above construction we have 
\[
\mathcal{H}^{\Sp}_{Q_{\Sigma_g},W_{\Sigma_g},(m,m,m,m)}\cong \Ho_{c,\GL_{\Cp}(m)}(\Rep_m(\Sigma_g),\mathbb{Q})^{\vee}\state{{(1-g)m^2}}.
\]
From Theorem \ref{comultalg} we deduce the following theorem.
\begin{theorem}
The graded mixed Hodge structure
\[
\bigoplus_{n\in\mathbb{N}}\Ho_{c,\GL_{\Cp}(m)}(\Rep_m(\Sigma_g),\mathbb{Q})^{\vee}\state{{(1-g)m^2}}
\]
carries the structure of a $Q$-localised bialgebra in the category of mixed Hodge structures.
\end{theorem}

We finish by returning to the conjectural form for the generators of $\mathcal{H}^{\Sp}_{Q_{\Sigma_g},W_{\Sigma_g}}$.  
\begin{conjecture}
There is a filtration $F$ on $\mathcal{H}^{\Sp}_{Q_{\Sigma_g},W_{\Sigma_g}}$ such that $\Gr_F(\mathcal{H}^{\Sp}_{Q_{\Sigma_g},W_{\Sigma_g}})\cong \mathcal{U}(\mathfrak{g}[u])$ for $\mathfrak{g}$ a $\mathbb{Z}^{Q_{\Sigma_g,0}}$-graded Lie algebra, and there are isomorphisms in $\Db{\MHS}$
\[
\mathfrak{g}_{(m,m,m,m)}\cong\Ho_{c}(\Rep_m^{\zeta_m}(\Sigma_g)/\PGL,\mathbb{Q})^{\vee}\state{(1-g)m^2}.
\]
\end{conjecture}

A candidate filtration $F$ will be constructed for general triples $(Q,W,\Sp)$ in the paper \cite{DaMe15b}.  The evidence for this conjecture comes from taking weight polynomials.  In \cite{HLRV13} the following calculation is made:
\[
\chi_q(\mathcal{H}^{\Sp}_{Q_{\Sigma_g},W_{\Sigma_g}})=\prod(1-x_m)^{\chi_q\big(\Ho_{c,\Gl_{\Cp}(m)}(\Rep_m^{\zeta_m}(\Sigma_g),\mathbb{Q})^{\vee}\state{(1-g)m^2}\big)},
\]
which means that if we have $\mathcal{H}^{\Sp}_{Q_{\Sigma_g},W_{\Sigma_g}}\cong \mathcal{U}(\mathfrak{g}[u])$ then the $(m,m,m,m)$th graded piece of the Lie algebra $\mathfrak{g}$ has the same weight polynomial as the cohomology of the twisted character variety $\Rep^{\zeta_m}_m(\Sigma_g)$, up to the correct Tate twist.

\appendix \label{dimredap}
\section{Dimensional reduction for quivers with potential and a cut}
\subsection{Relating critical cohomology to ordinary cohomology}
Let $Y:=X\times \Aff^n$ be the total space of the trivial vector bundle, carrying the $\Cp^*$-action that acts trivially on $X$ and with weight one on $\Aff^n$ (the rescaling action).  Let $f\colon Y\rightarrow \Aff^1$ be $\Cp^*$-equivariant, where $\Cp^*$ acts with weight one on the target.  We may express $f$ as
\[
f=\sum_{s=1}^{n}f_s x_s,
\]
where $\{x_s\}_{s\in\{1,\ldots,n\}}$ is a linear coordinate system on $\Aff^n$, i.e. each $x_s$ is acted on with weight one by the $\Cp^*$-action, and the $f_s$ are functions on $X$.  We define $Z\subset X$ as the space of closed points where all of the $f_s$ vanish.  Clearly $Z$ does not depend on the linear coordinate system we pick for $\Aff^n$ --- it can be defined without reference to it, as the space of closed points $z\in X$ such that $\pi^{-1}(z)\subset f^{-1}(0)$, where $\pi\colon Y\rightarrow X$ is the projection.

\begin{theorem}
\label{dim_red_prop}
In the above situation, let $i\colon  Z\rightarrow X$ be the closed inclusion.  There is a natural isomorphism of functors $\Dbc{X}\rightarrow \Dbc{X}$
\begin{equation}
\label{dimred}
\pi_!\phi_f\pi^*[-1]\cong \pi_!\pi^*i_*i^*,
\end{equation}
so that in particular:
\begin{equation}
\label{Hdimred}
\Ho_c^*(Y,\varphi_f)\cong\Ho_c^*(Z\times \Aff^n,\mathbb{Q})\cong \Ho_c^{*-2n}(Z,\mathbb{Q}).
\end{equation}
\end{theorem}

The proof of Theorem \ref{dim_red_prop} is a little complicated --- we first prove the special case $n=1$ and then reduce the general case to this.  Below we use the notation 
\begin{align*}
Y_+:=&f^{-1}(\mathbb{R}_{>0}), \\
Y_0:=&f^{-1}(0)
\end{align*}
and the definition
\[
\psi_f:=(Y_0\rightarrow Y)_*(Y_0\rightarrow Y)^*(Y_+\rightarrow Y)_*(Y_+\rightarrow Y)^*\mathcal{F}.
\]
For $\mathcal{F}\in\Dbc{Y}$ there is a distinguished triangle 
\begin{equation}
\label{cantri}
\phi_f[-1]\mathcal{F}\rightarrow \mathcal{F}|_{Y_0} \rightarrow \psi_f\mathcal{F}.
\end{equation}
\begin{lemma}
\label{lcco}
Let $X$, $Y$ and $f$ be as above, let $\mathcal{F}\in\Dbc{X}$, and assume $n=1$.  Then $\psi_f\pi^*\mathcal{F}$ has locally constant cohomology on the fibres of the projection $X\times(\mathbb{A}\setminus \{0\})\rightarrow X$, i.e. for each point $x\in X$, if we let $p'$ be the inclusion of the punctured fibre $(\mathbb{A}^1\setminus\{0\})$ over $x$ into $X\times \mathbb{A}^1$, $p'^*\mathcal{L}$ has locally constant cohomology sheaves.
\end{lemma}
\begin{proof}
We have $f=f_1u$ for $u$ a coordinate on $\mathbb{A}^1$.  The lemma, restated in terms of $f_1$ , is the statement that 
\[
\left(\big(f_1^{-1}(u^{-1}\mathbb{R}_{> 0})\rightarrow X\big)_*\big(f_1^{-1}(u^{-1}\mathbb{R}_{> 0})\rightarrow X\big)^*\mathcal{F}\right)|_{f_1^{-1}(0)}
\]
gives a family of objects in the derived category of constructible sheaves on $\mathbb{A}^1\setminus\{0\}\times X$ which is locally trivial on each fibre of the projection $(\mathbb{A}^1\setminus \{0\})\times X\rightarrow X$.  This follows from basic properties of the nearby cycles functor: the monodromy around $0\in\mathbb{A}^1$ is precisely the monodromy operator on $\psi_{f_1}\mathcal{F}$.
\end{proof}
\begin{lemma}
\label{mmm}
Let $\mathcal{L}\in\Dbc{X\times(\mathbb{A}^1\setminus \{0\})}$ have locally constant cohomology along the fibres of the projection $X\times(\mathbb{A}\setminus \{0\})\rightarrow X$.  Let $r\colon X\times(\mathbb{A}^1\setminus \{0\})\rightarrow Y=X\times\mathbb{A}^1$ be the inclusion, and let $\pi\colon Y\rightarrow X$ be the projection.  Then $\pi_!r_*\mathcal{L}=0$.
\end{lemma}
\begin{proof}
Since the condition on $\mathcal{L}$ is stable under replacing $\mathcal{L}$ with its Verdier dual, it is enough to prove that $\pi_*r_!\mathcal{L}=0$.  Let $\overline{Y}=X\times \overline{B}(0,1)\subset Y$ be the product of $X$ with the closed unit ball.  By the condition on $\mathcal{L}$, the map 
\[
\pi_*r_!\mathcal{L}\rightarrow \pi_*(\overline{Y}\rightarrow Y)_*(\overline{Y}\rightarrow Y)^*r_!\mathcal{L}
\]
is an isomorphism, and so it suffices to show that 
\[
\overline{\pi}_*\overline{r}_!\overline{\mathcal{L}}=0
\]
where $\overline{\pi}\colon \overline{Y}\rightarrow X$ is the projection, $\overline{r}\colon X\times \overline{B}'(0,1)\rightarrow X\times\overline{B}(0,1)$ is the inclusion, and 
\[
\overline{\mathcal{L}}=\left(X\times \overline{B}'(0,1)\rightarrow X\times (\mathbb{A}^1\setminus\{0\})\right)^*\mathcal{L}.  
\]
Here $\overline{B}'(0,1)$ is the closed unit ball in $\mathbb{C}$, with zero removed.

If $p\colon x\rightarrow X$ is the inclusion of the point $x$, it is enough to prove that 
\begin{equation}
\label{discNow}
p^*\overline{\pi}_*\overline{r}_!\overline{\mathcal{L}}=0
\end{equation}
for all choices of $x$.  By base change, 
\[
p^*\overline{\pi}_*\overline{r}_!\overline{\mathcal{L}}\cong \left(x\times\overline{B}(0,1)\rightarrow x\right)_*\left(x\times\overline{B}'(0,1)\rightarrow x\times\overline{B}(0,1)\right)_!\left(x\times \overline{B}'(0,1)\rightarrow X\times\overline{B}'(0,1)\right)^*\overline{\mathcal{L}}
\]
and so we have reduced to the case in which $X$ is a point, and we will now assume that $X=x$.  By the assumption on $\mathcal{L}$, the cohomology of $\overline{\mathcal{L}}$ admits a filtration by simple local systems $\mathcal{S}$ on $\overline{B}'(0,1)$, which in turn are given by simple finite-dimensional representations of $\mathbb{Z}$.  So we may prove the statement under the assumption that $\overline{\mathcal{L}}=\eta^*f_*\mathbb{Q}_{\overline{B}(0,1)}$, where $f\colon \overline{B}(0,1)\rightarrow \overline{B}(0,1)$ is the map $x\mapsto x^d$, and $\eta\colon \overline{B}'(0,1)\rightarrow\overline{B}(0,1)$ is the inclusion, i.e. we must show that $\Ho(\overline{B}(0,1),\eta_!\eta^*f_*\mathbb{Q}_{\overline{B}(0,1)})=0$.  Let $i\colon \{0\}\rightarrow\overline{B}(0,1)$ be the inclusion.  We have the commutative diagram
\[
\xymatrix{
\Ho(\overline{B}(0,1),\eta_!\eta^*f_*\mathbb{Q}_{\overline{B}(0,1)})\ar[r] &\Ho(\overline{B}(0,1),f_*\mathbb{Q}_{\overline{B}(0,1)})\ar[r]\ar[d]^{=} &\Ho(\overline{B}(0,1),i_*i^*f_*\mathbb{Q}_{\overline{B}(0,1)})\ar[d]^{\cong}\\
&\Ho(\overline{B}(0,1),\mathbb{Q})\ar[r]^{\cong}&\Ho(\{0\},\mathbb{Q})
}
\]
in which the top row is a distinguished triangle, from which the required vanishing follows.
\end{proof}

We return briefly to the case of general $n$.
\begin{lemma}
\label{suppLemma}
Let $\mathcal{F}\in\Dbc{X}$.  Then $\supp(\phi_f\pi^*\mathcal{F})\subset Z\times\mathbb{A}^n$.
\end{lemma}
\begin{proof}
Let $\tilde{i}\colon Z\times\AA^n\rightarrow Y$ be the inclusion.  Let $j$ be the inclusion of the open complement to $Z\times\AA^n$.  Then we have a distinguished triangle in the derived category of constructible sheaves on $Y$
\[
j_!j^*\phi_f\pi^*\mathcal{F}\rightarrow \phi_f\pi^*\mathcal{F}\rightarrow \tilde{i}_*\tilde{i}^*\phi_f\pi^*\mathcal{F}
\]
and we need to prove that $j_!j^*\phi_f\pi^*\mathcal{F}=0$.  For this we may replace $X$ by $J$, the open complement to $Z$, and show that $\phi_{f|_{J\times\mathbb{A}^n}}\pi|_J^*\mathcal{F}_J=0$, since there is an isomorphism
\[
j^*\phi_f\pi^*\mathcal{F}\cong\phi_{f|_{J\times\mathbb{A}^n}}\pi|_{J}^*\mathcal{F}_{J},
\]
as $j$ is an open inclusion.  The open sets $J_i:=J\setminus f_i^{-1}(0)$ form an open cover of $J$, and so it is enough to prove instead that that $\phi_{f|_{J_i\times\mathbb{A}^n}}\pi|_{J_i}^*\mathcal{F}_{J_i}=0$ for each $i$.  Via the isomorphism
\begin{align*}
\left(f^{-1}(0)\cap (J_i\times\mathbb{A}^n)\right)\times\mathbb{A}^1\rightarrow &J_i\times\mathbb{A}^n\\
(z,u)\mapsto z+ue_i,
\end{align*}
where $e_i$ is the $i$th basis element of $\mathbb{A}^n$, considered as a vector space, we reduce to the situation in which $n=1$, so for the rest of the proof we make this assumption.  

Consider $J\subset J\times \mathbb{A}^1$ as a subspace via the inclusion of the zero section.  Then under our simplifying assumptions on $n$ and $f$, $J=(J\times\mathbb{A}^1)_0:=f|_{J\times\mathbb{A}^1}^{-1}(0)$.  Finally, from the distinguished triangle (\ref{cantri}) we deduce that it is enough to check that the natural map that is the horizontal map in the commuting diagram
\[
\xymatrix{
(\pi_J^*\mathcal{F}_J)_J\ar[r]\ar[dr]_{\cong}& \left((J\times \mathbb{R}_{>0}\rightarrow J\times\mathbb{A}^1)_*(J\times \mathbb{R}_{>0}\rightarrow J\times\mathbb{A}^1)^*\pi_J^*\mathcal{F}_J\right)_J\ar[d]^{\cong}\\
&\mathcal{F}_J
}
\]
is an isomorphism, which is clear.
\end{proof}
\begin{remark}
Lemma \ref{suppLemma} is trivial in the case that $\mathcal{F}=\mathbb{Q}_X$, since $\crit(f)\subset Z\times\mathbb{A}^n$, and it is a general fact that $\supp(\varphi_f)=\crit(f)$ as long as $X$ is smooth (here we use our standing abuse of notation, whereby $\varphi_f=\phi_f\mathbb{Q}_Y[-1]$).  But for general $\mathcal{G}\in\Dbc{Y}$ it is not true that $\supp(\phi_f\mathcal{G})\subset\crit(f)$.  For example $\supp\left(\phi_f\mathbb{Q}_{f^{-1}(0)}\right)=f^{-1}(0)\nsubseteq Z\times\mathbb{A}^n$.  So it is important that we restrict to $\mathcal{G}$ of the form $\pi^*\mathcal{F}$ for $\mathcal{F}\in\Dbc{X}$.
\end{remark}

\begin{lemma}
\label{n1ver}
Theorem \ref{dim_red_prop} is true in the case $n=1$.
\end{lemma}
\begin{proof}
Let $\tilde{i}\colon Z\times\AA^1\rightarrow Y$ be the natural inclusion as in Lemma \ref{suppLemma}.  Then by Lemma \ref{suppLemma} the natural map $\phi_f\pi^*\rightarrow \tilde{i}_*\tilde{i}^*\phi_f\pi^*$, is an isomorphism.  By (\ref{cantri}) we have a distinguished triangle
\[
\pi_!\tilde{i}_*\tilde{i}^*\phi_f\pi^*[-1]\rightarrow \pi_!\tilde{i}_*\tilde{i}^*\pi^*\rightarrow \pi_!\tilde{i}_*\tilde{i}^*\psi_f\pi^*
\]
and the proposition will follow from the claim that $\pi_!\tilde{i}_*\tilde{i}^*\psi_f\pi^*=0$.  In other words we must show that
\[
\pi_!\tilde{i}_*\tilde{i}^*(Y_+\rightarrow Y)_*(Y_+\rightarrow Y)^*\pi^*=0.
\]
There is an inclusion $Y_+\subset X\times(\mathbb{A}^1\setminus \{0\})$ and so we may equivalently prove that
\[
\pi_!\tilde{i}_*\tilde{i}^*(X\times(\mathbb{A}^1\setminus \{0\})\rightarrow Y)_*(Y_+\rightarrow X\times(\mathbb{A}^1\setminus \{0\}))_*(Y_+\rightarrow Y)^*\pi^*=0,
\]
or
\[
\pi_!\tilde{i}_*\left(Z\times(\mathbb{A}^1\setminus \{0\})\rightarrow Z\times\mathbb{A}^1\right)_*\left(Z\times(\mathbb{A}^1\setminus \{0\})\rightarrow Y\right)^*(Y_+\rightarrow X\times(\mathbb{A}^1\setminus \{0\}))_*(Y_+\rightarrow Y)^*\pi^*=0,
\]
since for $\mathcal{F}$ a $\mathbb{R}_{>0}$-equivariant sheaf on $X\times(\AA^1\setminus \{0\})$, the base change map
\[
\tilde{i}^*(X\times(\mathbb{A}^1\setminus \{0\})\rightarrow Y)_*\mathcal{F}\rightarrow \left(Z\times(\mathbb{A}^1\setminus \{0\})\rightarrow Z\times\mathbb{A}^1\right)_*\left(Z\times(\mathbb{A}^1\setminus \{0\})\rightarrow Y\right)^*\mathcal{F}
\]
is an isomorphism.  By Lemma \ref{lcco} for $\mathcal{F}\in\Dbc{X}$,
\[
\left(Z\times(\mathbb{A}^1\setminus \{0\})\rightarrow Y\right)^*(Y_+\rightarrow X\times(\mathbb{A}^1\setminus \{0\}))_*(Y_+\rightarrow Y)^*\pi^*\mathcal{F}
\]
has locally constant cohomology along each restriction $x\times(\AA^1\setminus\{0\})$, and so the vanishing follows by Lemma \ref{mmm}.
\end{proof}
\begin{proof}[Proof of Theorem \ref{dim_red_prop}]
Let $g\colon \tilde{Y}\rightarrow Y$ be the blowup of $Y$ along $X\times \{0\}$, with exceptional divisor $E$, and with $\tilde{Z}:=(Z\times\mathbb{A}^n)\times_Y \tilde{Y}$.  As in the proof of Lemma \ref{n1ver}, we need to show that $\pi_!\tilde{i}_*\tilde{i}^*\psi_f\pi^*=0$.  The inclusion $Y_+\rightarrow Y$ factors through $\tilde{Y}\rightarrow Y$, and so we may write
\begin{align*}
\pi_!\tilde{i}_*\tilde{i}^*\psi_f\pi^*\cong&(Y\rightarrow X)_!(Z\times \AA^n\rightarrow Y)_*(Z\times\AA^n\rightarrow Y)^*(Y_+\rightarrow Y)_*(Y_+\rightarrow Y)^*(Y\rightarrow X)^*\\ \cong&
(Y\rightarrow X)_!(Z\times \AA^n\rightarrow Y)_*(Z\times\AA^n\rightarrow Y)^*(\tilde{Y}\rightarrow Y)_*(Y_+\rightarrow \tilde{Y})_*\\&(Y_+\rightarrow \tilde{Y})^*(\tilde{Y}\rightarrow Y)^*(Y\rightarrow X)^*\\ \cong&
(Y\rightarrow X)_!(\tilde{Y}\rightarrow Y)_!(\tilde{Z}\rightarrow \tilde{Y})_*(\tilde{Z}\rightarrow\tilde{Y})^*(Y_+\rightarrow \tilde{Y})_*(Y_+\rightarrow \tilde{Y})^*(\tilde{Y}\rightarrow E)^*(E\rightarrow X)^*\\ \cong&
(E\rightarrow X)_!(\tilde{Y}\rightarrow E)_!(\tilde{Z}\rightarrow \tilde{Y})_*(\tilde{Z}\rightarrow\tilde{Y})^*(Y_+\rightarrow \tilde{Y})_*(Y_+\rightarrow \tilde{Y})^*(\tilde{Y}\rightarrow E)^*(E\rightarrow X)^*
\end{align*}
and the result will follow from the claim that 
\begin{equation}
\label{RedVan}
(\tilde{Y}\rightarrow E)_!(\tilde{Z}\rightarrow \tilde{Y})_*(\tilde{Z}\rightarrow\tilde{Y})^*(Y_+\rightarrow \tilde{Y})_*(Y_+\rightarrow \tilde{Y})^*(\tilde{Y}\rightarrow E)^*=0.
\end{equation}
If we let $U\subset \tilde{Y}$ be defined by the condition $x_j\neq 0$ for some fixed $j\leq n$, we may write
\[
fg|_U=L\sum_{i=1}^n f_i x_i/x_j
\]
from which we deduce that we are in the situation of Lemma \ref{n1ver}, with $L$ our coordinate on $\mathbb{A}^1$ and $f'_1=\sum f_i x_i/x_j$.  The space $\tilde{Z}\cap U$ is contained in the vanishing locus of the function $f'_1$ since on an open dense subset of $\tilde{Z}\cap U$ all the $f_i$ vanish, from which we deduce (\ref{RedVan}) from Lemma \ref{n1ver}.  
\end{proof}
\subsection{The Thom--Sebastiani isomorphism and other corollaries}
Theorem \ref{dim_red_prop} can be expressed by saying that the natural transformation of derived functors of sheaves 
\[
\pi_!(Z\times\mathbb{A}^n\rightarrow Y)_*(Z\times\mathbb{A}^n\rightarrow Y)^*(\Gamma_{\{\real(f)\leq 0\}}\rightarrow\id)\pi^*
\]
is an isomorphism.  Since this lifts to a natural transformation of derived functors of mixed Hodge modules, and the forgetful functor from mixed Hodge modules to constructible sheaves is faithful, we deduce the following corollary.
\begin{corollary}
Let $X^{\Sp}\subset X$ be a subvariety of a variety $X$, and define 
\begin{align*}
Y^{\Sp}=&X^{\Sp}\times\mathbb{A}^n\\
Z^{\Sp}=&Z\cap X^{\Sp}\\
\tilde{Z}^{\Sp}=&Z^{\Sp}\times\mathbb{A}^n.
\end{align*}
Let $f\colon Y\rightarrow \Cp$ be a regular $\Cp^*$-equivariant function, where the target carries the weight one action and $Y$ carries the action given by the product of the trivial action on $X$ and the scaling action on $\mathbb{A}^n$.  There is a natural isomorphism in $\Db{\MMHS}$
\[
\Ho_c(Y^{\Sp},\varphi_f)\cong\Ho_c(\tilde{Z}^{\Sp},\QQ).
\]
In particular, the element $\Ho_c(Y^{\Sp},\varphi_f)\in\Ob(\Db{\MMHS})$ is in the essential image of the functor $h^*s_*\colon \Db{\MHS}\rightarrow\Db{\MMHS}$.\end{corollary}
\begin{proof}
Let $u$ be the coordinate on $\Cp^*$.  We extend the $\mathbb{C}^*$-action on $Y$ to a $\mathbb{C}^*$-action on $Y\times\mathbb{C}^*$ by letting $\mathbb{C}^*$ act trivially on itself.  Let $\pi'\colon Y\times\mathbb{C}^*\rightarrow X\times \mathbb{C}^*$ be the projection and let $i'\colon Z\times\mathbb{C}^*\rightarrow X\times \mathbb{C}^*$ be the inclusion.  By Theorem \ref{dim_red_prop} there is an isomorphism
\[
(X^{\Sp}\times \Cp^*\rightarrow X\times\Cp^*)^*(\pi'_!\phi_{f/u}\pi'^*\mathbb{Q}_{X\times\Cp^*}\rightarrow\pi'_!\pi'^*i'_*i'^*\mathbb{Q}_{X\times\Cp^*}[1])
\]
which induces the desired isomorphism after taking direct image with compact supports along the projection to $\Cp^*$.
\end{proof}

At this point it becomes quite straightforward to prove the Thom--Sebastiani isomorphism at the level of mixed Hodge structures, assuming that we have the required $\mathbb{C}^*$-equivariance for our functions.  Let $Y=Y_1\times Y_2$, and let $f=f_1\boxplus f_2$.  Assume furthermore that both $Y_1$ and $Y_2$ admit product decompositions $Y_i\cong X_i\times \mathbb{A}^{n_i}$ such that the $f_i$ are $\Cp^*$-equivariant, where $\Cp^*$ acts with weight one on $\mathbb{A}^{n_i}$, trivially on $X_i$, and with weight one on the target $\Cp$.  From the commutativity of the square of underived functors
\[
\xymatrix{
\Gamma_{\{\real(f_1)\leq 0\}}(\bullet)\boxtimes\Gamma_{\{\real(f_2)\leq 0\}}(\bullet)\ar[d]\ar[r]&\Gamma_{\{\real(f_1\boxplus f_2)\leq 0\}}(\bullet\boxtimes\bullet)\ar[d]\\\id(\bullet)\boxtimes\id(\bullet)
\ar[r]&\id(\bullet\boxplus\bullet)
}
\]
and bi-exactness of $\boxtimes$ we deduce the following proposition.
\begin{proposition}
\label{TScomm}
The following diagram of isomorphisms of mixed Hodge structures commutes
\[
\xymatrix{
\Ho_c(Y_1^{\Sp},\varphi_{f_1})\otimes \Ho_c^*(Y_2^{\Sp},\varphi_{f_2})\ar[d]_{\cong}^{\substack{\mathrm{dimensional}\\\mathrm{reduction}}}\ar[r]^-{\TS}&\Ho_c(f_1^{-1}(0)^{\Sp}\times f_2^{-1}(0)^{\Sp},\varphi_{f_1\boxplus f_2})\ar[d]_{\cong}^{\substack{\mathrm{dimensional}\\\mathrm{reduction}}}\\
\Ho_c(\tilde{Z}^{\Sp}_1,\mathbb{Q})\otimes\Ho_c^*(\tilde{Z}^{\Sp}_2,\mathbb{Q})\ar[r]^-{\Kunn}_-{\cong}&\Ho_c(\tilde{Z}^{\Sp}_1\times \tilde{Z}^{\Sp}_2,\mathbb{Q}),
}
\]
where $\TS$ is the Thom-Sebastiani map, $\Kunn$ is the Kunneth isomorphism, and the vertical isomorphisms are as in Theorem \ref{dim_red_prop}.  In particular, $\TS$ is an isomorphism.
\end{proposition}
We have again used the appendix of \cite{Br12} to identify the morphism at the level of perverse sheaves underlying the Saito Thom--Sebastiani isomorphism with Massey's Thom--Sebastiani isomorphism \cite{Ma01}.
\begin{corollary}
\label{csdr}
With $Y$ as in the statement of Theorem \ref{dim_red_prop}, assume also that $X$ is a $G$-equivariant complex algebraic variety, where $G$ is a complex algebraic group, and assume also that $Y$ is a $G$-equivariant vector bundle over $X$, and $Y^{\Sp}=X^{\Sp}\times\mathbb{A}^n$ is the total space of a sub $G$-bundle.  Then there is an isomorphism in $\D{\MHS}$
\[
\Ho_{c,G}(Y^{\Sp},\varphi_f)\rightarrow\Ho_{c,G}(\tilde{Z}^{\Sp},\mathbb{Q}).
\]
\end{corollary}
\begin{proof}
Let $h\colon \overline{(Y,G)}_N\rightarrow \overline{(Y,G)}_{N+1}$ be the natural inclusion and set $T=\overline{(f^{-1}(0),G)}_{N+1}$.  The diagram
\begin{equation}
\label{abovd}
\xymatrix{
\mathbb{Q}_{\overline{(Y,G)}_{N+1}}\state{{\dim\left(\overline{(Y,G)}_{N+1}\right)}}|_T&\ar[l]\phi_{f_{N+1}}\mathbb{Q}_{\overline{(Y,G)}_{N+1}}\state{{\dim\left(\overline{(Y,G)}_{N+1}\right)}}[-1]
\\
h_*\mathbb{Q}_{\overline{(Y,G)}_N}\state{{\dim\left(\overline{(Y,G)}_{N}\right)}}|_T\ar[u]&\ar[l] h_*\phi_{f_N}\mathbb{Q}_{\overline{(Y,G)}_N}\state{{\dim\left(\overline{(Y,G)}_{N}\right)}}[-1]\ar[u]
}
\end{equation}
commutes as it is obtained by applying the natural transformation $\phi_{f_{N+1}}[-1]\rightarrow \id|_T$ to the map $h_*\mathbb{Q}_{\overline{(Y,G)}_N}\state{{\dim\left(\overline{(Y,G)}_{N}\right)}}\rightarrow\mathbb{Q}_{\overline{(Y,G)}_{N+1}}\state{{\dim\left(\overline{(Y,G)}_{N+1}\right)}}$.  Applying 
\[
\left(\overline{(\tilde{Z}^{\Sp},G)}_{N+1}\rightarrow \overline{(Y,G)}_{N+1}\right)^*
\]
and taking compactly supported cohomology, we deduce that the following diagram commutes
\[
\xymatrix{
\Ho_{c}\left(\overline{(\tilde{Z}^{\Sp},G)}_{N+1},\mathbb{Q}\right)\state{t}& \ar[l]\Ho_{c}\left(\overline{(Y^{\Sp},G)}_{N+1},\varphi_{f_{N+1}}\right)\state{t}\\
\Ho_{c}\left(\overline{(\tilde{Z}^{\Sp},G)}_{N},\mathbb{Q}\right)\ar[u]&\Ho_{c}\left(\overline{(Y^{\Sp},G)}_{N},\varphi_{f_{N}}\right)\ar[l]\ar[u]
}
\]
where $t=\dim\left(\overline{(Y,G)}_{N+1}\right)-\dim\left(\overline{(Y,G)}_{N}\right)$.  Letting $N$ go to infinity, the corollary follows.
\end{proof}

\subsection{The critical cohomology of a quiver with potential and a cut}
\label{2dCOHA}
Let $Q$ be a quiver with potential $W\in \Cp Q/[\Cp Q,\Cp Q]$.  We assume that the QP $(Q,W)$ admits a \textit{cut}.  That is, there is a grading $\nu\colon Q_1\rightarrow \mathbb{N}$ of the edges of $Q$ with zeroes and ones such that $W$ is homogeneous of degree one.  Equivalently, there is a subset $S\subset Q_1$ such that every cyclic word in $W$ contains exactly one instance of exactly one arrow of $S$.  We pass from one description to the other by setting $S=\nu^{-1}(1)$.  For $\gamma\in\mathbb{N}^{Q_0}$ denote by $\RS_{\gamma}$, as before, the affine space
\[
\bigoplus_{a\in Q_1}\Hom(\Cp^{\gamma(s(a))},\Cp^{\gamma(t(a))}).
\]

The space $\RS_{\gamma}$ admits a decomposition $\RS_{\gamma}=(\RS_{\gamma})_0\oplus(\RS_{\gamma})_1$, where 
\[
(\RS_{\gamma})_i:=\bigoplus_{a\in\nu^{-1}(i)}\Hom(\Cp^{\gamma(s(a))},\Cp^{\gamma(t(a))}).
\]
This is the weight space decomposition of the natural $\Cp^*$-action on $\RS_{\gamma}$, defined via $\nu$.  

Up to cyclic permutation of words in $Q$, we may write $W=\sum_{a\in\nu^{-1}(1)}ap_a$, where $p_a$ are linear combinations of paths in $Q$ containing only arrows in $\nu^{-1}(0)$.  We define 
\[
\tilde{i}_{\gamma}\colon Z_{\gamma}\times\AA^{\sum_{a\in\nu^{-1}(1)}\gamma(s(a))\gamma(t(a))}\rightarrow \RS_{\gamma} 
\]
to be the inclusion of subscheme cut out by the matrix valued equations $\{p_a=0|a\in\nu^{-1}(1)\}$.

As in the rest of the paper we assume that we are given subspaces $\RS^{\Sp}_{\gamma}\subset \RS_{\gamma}$ satisfying Assumption \ref{closed_under}, as well as the the following additional assumption.
\begin{assumption}
\label{ass1}
The spaces $\RS^{\Sp}_{\gamma}$ are in turn given by pullbacks of subspaces $\pi_{\gamma,0}^{-1}((\RS_{\gamma})_0^{\Sp})$ for $(\RS_{\gamma})_0^{\Sp}\subset(\RS_{\gamma})_0$ and $\pi_{\gamma,0}\colon \RS_{\gamma}\rightarrow (\RS_{\gamma})_0$ the natural projection.  
\end{assumption}

In words, we assume that the subspace $\RS^{\Sp}_{\gamma}$ is the space of $Q$-representations $\rho$ satisfying some property that defines a Serre subcategory of the category of $Q$-representations, and is independent of the values of $\rho(a)$ for $a\in S$.  We define $Z_{\gamma}^{\Sp}=Z_{\gamma}\cap (M_{\gamma})_0^{\Sp}$.  The following is then a direct application of Corollary \ref{csdr}.
\begin{theorem}
\label{eqred}
Let $(Q,W)$ admit a cut, and assume that the spaces $\RS^{\Sp}_{\gamma}$ satisfy Assumptions \ref{closed_under} and \ref{ass1}.  Then there are canonical isomorphisms
\begin{align*}
\Ho_{c,\Gl_{\gamma}}\left(\RS_{\gamma}^{\Sp},\varphi_{\tr(W)_{\gamma}}\right)\cong&\Ho_{c,\Gl_{\gamma}}\left(\pi^{-1}_{\gamma,0}\left((\RS_{\gamma})_0^{\Sp}\cap Z_{\gamma}\right),\mathbb{Q}\right)\\
\cong&\Ho_{c,\Gl_{\gamma}}(Z_{\gamma}^{\Sp},\QQ)\state{{-\sum_{a\in\nu^{-1}(1)}\gamma(s(a))\gamma(t(a))}}.
\end{align*}
\end{theorem}

\bibliographystyle{amsplain}
\bibliography{SPP}

\providecommand{\bysame}{\leavevmode\hbox to3em{\hrulefill}\thinspace}
\providecommand{\MR}{\relax\ifhmode\unskip\space\fi MR }
\providecommand{\MRhref}[2]{%
  \href{http://www.ams.org/mathscinet-getitem?mr=#1}{#2}
}
\providecommand{\href}[2]{#2}
\begin{thebibliography}{10}

\bibitem{AtBo83}
M.~Atiyah and R.~Bott, \emph{{The Yang-Mills equations over Riemann surfaces}},
  Philos. Trans. Roy. Soc. London \textbf{308} (1983), 523--615.

\bibitem{BBS}
K.~Behrend, J.~Bryan, and B.~Szendr\H{o}i, \emph{Motivic degree zero
  {D}onaldson-{T}homas invariants}, Invent. Math. \textbf{192} (2013), no.~1,
  111--160.

\bibitem{Br12}
C.~Brav, V.~Bussi, D.~Dupont, D.~Joyce, and B.~Szendr\H{o}i, \emph{{Symmetries
  and stabilization for sheaves of vanishing cycles}}, arXiv preprint
  arXiv:1211.3259 (2012).

\bibitem{longout}
N.~Broomhead, \emph{{Dimer models and Calabi-Yau algebras}}, vol. 211,
  {American Mathematical Soc.}, 2012.

\bibitem{LBP90}
L.~Le Bruyn and C.~Procesi, \emph{Semisimple representations of quivers},
  {Transactions of the American Mathematical Society} (1990), 585--598.

\bibitem{Dav08}
B.~Davison, \emph{Consistency conditions for brane tilings}, J. Alg
  \textbf{338} (2011), 1--23.

\bibitem{Chicago3}
\bysame, \emph{{Purity of critical cohomology and Kac's conjecture}},
  \url{http://arxiv.org/abs/1311.6989}, 2013.

\bibitem{DMSS12}
B.~Davison, D.~Maulik, J.~Sch{\"u}rmann, and B.~Szendr\H{o}i, \emph{Purity for
  graded potentials and quantum cluster positivity}, Comp. Math. \textbf{151},
  1913--1944.

\bibitem{DM11}
B.~Davison and S.~Meinhardt, \emph{Motivic {DT}--invariants for the one loop
  quiver with potential}, Geom. \& Top. \textbf{19} (2015), 2535--2555.

\bibitem{DaMe15b}
\bysame, \emph{{Cohomological Donaldson--Thomas theory of a quiver with
  potential and quantum enveloping algebras}},
  \url{http://arxiv.org/abs/math/1601.02479}, Jan 2016.

\bibitem{EdGr98}
D.~Edidin and W.~Graham, \emph{{Equivariant intersection theory (With an
  appendix by Angelo Vistoli: The Chow ring of M2)}}, Inventiones mathematicae
  \textbf{131} (1998), no.~3, 595--634.

\bibitem{Ef12}
A.~Efimov, \emph{{Cohomological Hall algebra of a symmetric quiver}}, Comp.
  Math. \textbf{148} (2012), no.~4, 1133--1146.

\bibitem{ginz}
V.~Ginzburg, \emph{{C}alabi-{Y}au algebras},
  \url{http://arxiv.org/abs/math/0612139}, 2006.

\bibitem{GKM97}
M.~Goresky, R.~Kottwitz, and R.~MacPherson, \emph{{Equivariant cohomology,
  Koszul duality, and the localization theorem}}, {Inventiones Mathematicae}
  \textbf{131} (1997), no.~1, 25--83.

\bibitem{Rhombi}
A.~Hanany and D.~Vegh, \emph{Quivers, tilings, branes and rhombi}, Journal of
  High Energy Physics \textbf{2007} (2007), no.~10, 029.

\bibitem{HLRV13}
T.~Hausel, E.~Letellier, and F.~Rodriguez-Villegas, \emph{Positivity for kac
  polynomials and dt-invariants of quivers}, Ann. Math. \textbf{177} (2013),
  1147--1168.

\bibitem{KS90}
M.~Kashiwara and P.~Schapira, \emph{{Sheaves on Manifolds}}, {Grund. der Math.
  Wiss.}, vol. 292, {Springer -- Verlag}, 1990.

\bibitem{KS}
M.~Kontsevich and Y.~Soibelman, \emph{{Stability structures, motivic
  Donaldson-Thomas invariants and cluster transformations}},
  \url{http://arxiv.org/abs/0811.2435}, 2008.

\bibitem{COHA}
\bysame, \emph{{Cohomological Hall algebra, exponential Hodge structures and
  motivic Donaldson-Thomas invariants}}, {Comm. Num. Th. and Phys.} \textbf{5}
  (2011), no.~2, 231--252.

\bibitem{Ma01}
D.~B. Massey, \emph{{The Sebastiani--Thom isomorphism in the Derived
  Category}}, Comp. Math. \textbf{125} (2001), no.~3, 353--362.

\bibitem{Ma09}
David~B Massey, \emph{{Natural commuting of vanishing cycles and the Verdier
  dual}}, arXiv preprint arXiv:0908.2799 (2009).

\bibitem{MaOk12}
D.~Maulik and A.~Okounkov, \emph{{Quantum Groups and Quantum Cohomology}},
  \url{http://arxiv.org/abs/1211.1287}.

\bibitem{Meb89}
Z.~Mebkhout, \emph{{Le formalisme des six op\'erations de Grothendieck pour les
  $\mathcal{D}_X$-modules coh\'erents: syst\`emes diff\'erentiels}}, Hermann,
  Paris, 1989.

\bibitem{MeRe14}
S.~Meinhardt and M.~Reineke, \emph{{Donaldson-Thomas invariants versus
  intersection cohomology of quiver moduli}},
  \url{http://arxiv.org/abs/1411.4062}.

\bibitem{MR}
S.~Mozgovoy and M.~Reineke, \emph{On the noncommutative {D}onaldson-{T}homas
  invariants arising from brane tilings}, Adv. Math. \textbf{223} (2010),
  no.~5, 1521--1544.

\bibitem{Sai89}
M.~Saito, \emph{{Introduction to mixed Hodge modules}}, Ast\`{e}risque
  \textbf{179--180} (1989), 145--162.

\bibitem{Sai90}
\bysame, \emph{{Mixed Hodge modules}}, Publ. RIMS \textbf{26} (1990), 221--333.

\bibitem{Sa88}
Morihiko Saito, \emph{Modules de hodge polarisables}, Publications of the
  Research Institute for Mathematical Sciences \textbf{24} (1988), no.~6,
  849--995.

\bibitem{So14}
Yan Soibelman, \emph{{Remarks on Cohomological Hall algebras and their
  representations}}, \url{http://arxiv.org/abs/1404.1606}, 2014.

\bibitem{YaZh14}
Y.~Yang and G.~Zhao, \emph{{Cohomological Hall algebra of a preprojective
  algebra}}, \url{http://arxiv.org/abs/1407.7994}.

\end{thebibliography}
\end{document}